\documentclass[11pt,leqno]{amsart}
\makeatletter
\AtBeginDocument{%
  \let\original@@tocwrite=\@tocwrite
  \newif\ifAHVflag
  \def\AHV@uniqtoken{\AHV@uniqtoken}
  \def\AHV@endmark{\AHV@endmark}
  \long\def\AHV@getfirsttoken#1#{\AHV@getfirsttoken@#1\bgroup\AHV@endmark}
  \long\def\AHV@getfirsttoken@#1#2\AHV@endmark#3\AHV@endmark{#1}
  \renewcommand{\@tocwrite}[2]{%
    \begingroup
      \AHVflagfalse %
      \@ifempty{#2}{}{%
        \expandafter\expandafter\expandafter\ifx\AHV@getfirsttoken#2\AHV@uniqtoken{}\AHV@endmark\Sectionformat\expandafter\@firstoftwo\else\expandafter\@secondoftwo\fi
        {%
          \def\Sectionformat##1##2{\@ifempty{##1}{}{\AHVflagtrue}}%
          #2
        }{\AHVflagtrue}%
      }%
      \def\@tempa{}%
      \ifAHVflag\def\@tempa{\original@@tocwrite{#1}{#2}}\fi
    \expandafter\endgroup
    \@tempa
  }%
}

\let\original@maketitle=\maketitle
\let\origina@@maketitle=\@maketitle
\let\original@email=\email
\let\original@address=\address
\expandafter\def\expandafter\appendix\expandafter{\expandafter\enddoc@text\appendix
\let\addresses\@empty
\let\email=\original@email
\let\address=\original@address
\bigskip
}
\makeatother

\usepackage{a4wide}
\usepackage{amssymb}
\usepackage{amsmath}
\usepackage{amsthm}
\usepackage{hyperref}
\usepackage[all]{xy}
\usepackage[usenames,dvipsnames]{xcolor}
\usepackage{lmodern}
\usepackage[T1]{fontenc}

\newcommand{\id}{\operatorname{id}}
\newcommand{\Hom}{\operatorname{Hom}}
\newcommand{\Ind}{\operatorname{Ind}}
\newcommand{\ind}{\operatorname{c-Ind}}
\newcommand{\Ker}{\operatorname{Ker}}
\newcommand{\End}{\operatorname{End}}
\newcommand{\Gal}{\operatorname{Gal}}
\newcommand{\Res}{\operatorname{Res}}
\newcommand{\ord}{\operatorname{ord}}
\newcommand{\aff}{\operatorname{aff}}
\newcommand{\op}{\operatorname{op}}
\newcommand{\SSt}{\operatorname{St}}
\newcommand{\ad}{\operatorname{ad}} 
\DeclareMathOperator{\GL}{GL}
\DeclareMathOperator{\SL}{SL}
\DeclareMathOperator{\PU}{PU}
\DeclareMathOperator{\PGL}{PGL}
\DeclareMathOperator{\SU}{SU}
\DeclareMathOperator{\U}{U}

\DeclareMathOperator{\diag}{diag}
\newcommand{\trivrep}{{1}}
\DeclareMathOperator{\supp}{supp}
\DeclareMathOperator{\der}{{der}}
\DeclareMathOperator{\red}{{red}}

\newcommand{\loc}{\mathrm{loc}}
\newcommand{\abs}{^{\mathrm{abs}}}
\newcommand{\sep}{^{\mathrm{sep}}}

\newcommand{\Z}{\mathbb Z}
\newcommand{\Q}{\mathbb Q}
\newcommand{\R}{\mathbb R}

\newcommand{\into}{\hookrightarrow}
\newcommand{\onto}{\twoheadrightarrow}
\newcommand{\congto}{\xrightarrow{\,\sim\,}}
\renewcommand{\u}{\underline}
\newcommand{\cO}{\mathcal O}
\newcommand{\HH}{\mathcal H}
\renewcommand{\o}[1]{\overline{#1}}
\newcommand{\wt}[1]{\widetilde{#1}}

\renewcommand{\)}{\textup{)}}

\theoremstyle{plain} 
\newtheorem{theorem}{Theorem}[section]
\newtheorem{corollary}[theorem]{Corollary}
\newtheorem{lemma}[theorem]{Lemma}
\newtheorem{proposition}[theorem]{Proposition}

\newtheorem{proposition-definition}[theorem]{Proposition-Definition}

\theoremstyle{definition}
\newtheorem{definition}[theorem]{Definition}

\theoremstyle{remark}
\newtheorem{remark}[theorem]{Remark}
\newtheorem{example}[theorem]{Example}
\newtheorem{notation}[theorem]{Notation}

\numberwithin{equation}{section}

\newcommand{\cn}{\mathcal{N}}
\newcommand{\boldN}{\boldsymbol{\cn}}

\title{Inverse Satake isomorphism and change of weight}

\author{N.\ Abe}
\address[N.\ Abe]{Graduate School of Mathematical Sciences, the University of Tokyo, 3-8-1 Komaba, Meguro-ku, Tokyo 153-8914, Japan}
\thanks{The first-named author was supported by JSPS KAKENHI Grant Number 18H01107.}
\email{abenori@ms.u-tokyo.ac.jp}

\author{F.\ Herzig} 
\address[F.\ Herzig]{Department of Mathematics, University of Toronto,
  40 St.\ George Street, Room 6290, Toronto, ON M5S 2E4, Canada}
\thanks{The second-named author was partially supported by a Sloan Fellowship, a Simons Fellowship, and an NSERC grant.}
\email{herzig@math.toronto.edu}

\author{M.-F.\ Vign\'eras}
\address[M.-F.\ Vign\'eras]{Institut de Math\'ematiques de Jussieu, 4 place Jussieu, 75005 Paris, France}
\email{vigneras@math.jussieu.fr}

\keywords{change of weight, Satake transform, compact induction, parabolic induction, pro-$p$ Iwahori Hecke algebra}
\subjclass[2010]{primary 20C08, secondary  11F70}
\date{\today}

\begin{document} 

\maketitle

\newcommand{\F}{\mathbb F}
\newcommand{\fpb}{\overline \F_p}

\begin{abstract}
  Let $G$ be any connected reductive $p$-adic group. Let $K\subset G$ be any special parahoric
  subgroup and $V,V'$ be any two irreducible smooth $\fpb[K]$-modules. The main goal of
  this article is to compute the image of the Hecke bi-module $\End_{\fpb[K]}(\ind_K^G V,
  \ind_K^G V')$ by the generalized Satake transform and to give an explicit formula for its inverse,
  using the pro-$p$ Iwahori Hecke algebra of $G$.  This immediately implies the ``change of weight theorem''
  in the proof of the classification of mod $p$ irreducible admissible representations of $G$ in
  terms of supersingular ones. A simpler proof of the change of weight theorem, not using the pro-$p$ Iwahori
  Hecke algebra or the Lusztig-Kato formula, is given when $G$ is split (and in the appendix when $G$ is quasi-split,
  for almost all $K$).
\end{abstract}

\setcounter{tocdepth}{2}
\tableofcontents

\section{Introduction}
\label{sec:introduction}

\subsection{}Throughout this paper, $F$ is a local nonarchimedean field  with finite residue field $k$  of characteristic $p$,  $\bf G$ is a connected reductive $F$-group, and  $C$ is    an algebraically closed field of characteristic $p$. In our previous paper \cite{MR3600042}, we gave a classification of irreducible admissible smooth $C$-representations of $G = \mathbf{G}(F)$ in terms of supercuspidal representations of Levi subgroups of $G$.
The most subtle ingredient in our proofs is the so-called ``change of weight theorem'', which we
deduced from the existence of  certain elements in the image of the mod $p$ Satake transform.
The main goal of this paper is to determine its image entirely and give  an explicit formula for the inverse of the mod $p$ Satake transform,  we call it the \emph{inverse Satake theorem}, from which the change of weight  is an immediate consequence.  

To be a bit more precise, the mod $p$ Satake transform can be defined for the Hecke algebra of a single irreducible representation $V$ of a special parahoric subgroup, as well as more generally for the Hecke bimodule of a pair $(V,V')$ of such irreducible representations.
The image of the mod $p$ Satake transform was known in case of a single irreducible representation $V$ of a special parahoric subgroup, cf.\ \cite{MR3331726}, \cite{bib:satake}. However, for the change of weight theorem it is essential to allow pairs $(V,V')$ with $V \not\cong V'$.

In earlier work \cite[Prop.\ 5.1]{MR2845621}, we established the inverse Satake theorem when $\mathbf G$ is split with simply-connected derived subgroup and $V = V'$ by deducing it from the Lusztig-Kato formula, which is an inverse formula for the usual Satake transform in characteristic zero. (See also the related work of Ollivier \cite{MR3366919}.)
In this paper we establish the inverse Satake theorem in characteristic $p$ for arbitrary $\mathbf G$ and pairs $(V,V')$ by using the pro-$p$ Iwahori Hecke algebra.

\subsection{} We now explain our results in more detail.
Let $\mathbf{S}$ be a maximal split torus of $\mathbf{G}$, $\mathbf{Z}$ its centralizer, $\mathbf{B} = \mathbf{Z}\mathbf{U}$ a minimal parabolic subgroup and $\Delta$ the set of simple roots defined by $(\mathbf{G},\mathbf{B},\mathbf{S})$.
Put $Z = \mathbf{Z}(F)$ and $U = \mathbf{U}(F)$.
Let $X_*(\mathbf{S})$ be the group of cocharacters of $\mathbf{S}$ and $v_Z\colon Z\to X_*(\mathbf{S})\otimes\mathbb{R}$ be the usual homomorphism (see Section~\ref{Notation}).
Put $Z^+ = \{z\in Z\mid \text{$\langle \alpha,v_Z(z)\rangle \ge 0$ for any $\alpha\in\Delta$}\}$, so that $Z^+$ contracts
$U$ under conjugation.

Let $K$ be a special parahoric subgroup of $G$ corresponding to a special point of the apartment of $S$ and put $Z^0 = Z \cap K$ (the unique parahoric subgroup of $Z$), $U^0 = U \cap K$.
Let $V$ be an irreducible smooth $C$-representation of $K$.
It is parameterized by a pair $(\psi_V,\Delta(V))$, where $\psi_V : Z^0 \to C^\times$ describes the action of $Z^0$ on the line $V_{U^0}$ and $\Delta(V)\subset \Delta$
is a certain subset (see \S\ref{GS}).
Let $\ind_K^GV$ denote the compact induction of $V$.
If  $V'$ denotes another irreducible smooth $C$-representation of $K$, we define the Hecke bimodule $\mathcal{H}_G(V,V') := \Hom_{CG}(\ind_K^GV,\ind_K^GV')$.
This is non-zero if and only if $\psi_V$ is $Z$-conjugate to $\psi_{V'}$. 
Once we fix a linear isomorphism $\iota\colon V_{U^0}\simeq V'_{U^0}$, $\mathcal{H}_G(V,V')$ has a canonical $C$-basis $\{T_z=T_z^{V',V}\}$, where $z$ runs through a system of representatives of $Z_G^+(V,V')/Z^0$ and $Z_G^+(V,V')$ is a certain union of cosets of $Z^0$ in $Z^+\cap Z_{\psi_V, \psi_{V'}}$, where $Z_{\psi_V, \psi_{V'}}=\{ z \in Z \mid z\cdot\psi_V=\psi_{V'} \}$ (see (\ref{ZVV'})).
The element $T_z^{V',V}$ is determined up to scalar by the condition $\supp T_z^{V',V} = KzK$ and normalized by $\iota$ (see \S\ref{intVV'}).

Similarly, we have the Hecke bimodule $\mathcal{H}_Z(V_{U^0},V'_{U^0})$ with $C$-basis $\{\tau_z=\tau_z^{V'_{U^0},V_{U^0}}\}$, where $z$ runs through a system of representatives of $Z_{\psi_V, \psi_{V'}}/Z^0$.
Then we have the mod $p$ Satake transform $S^G \colon \mathcal{H}_G(V,V')\hookrightarrow \mathcal{H}_Z(V_{U^0},V'_{U^0})$ which is $C$-linear and injective~\cite{MR3331726}:
\[ 
 S^G(f)(z)(\overline{v})= \sum_{u\in  U^0\backslash U} \overline{f(uz)(v)}, \quad \text{for $f\in \mathcal H_G(V,V'),z\in Z$ and $v\in V$,} 
\]
where $v\mapsto \overline v: V\to V_{U^0}$ (resp.\ $V'\to V'_{U^0}$) is the quotient map from $V$ (resp.\ $V'$) onto its  $U^0$-coinvariants, and we realize $\mathcal H_G(V,V')$ as a set of compactly supported functions on $G$ with a certain $K$-bi-equivariance.

\subsection{}For $\alpha\in\Delta$, let $M'_\alpha$ be the subgroup of $G$ generated by the root subgroups $U_{\pm \alpha}$ for the roots $\pm \alpha$. (Note that this need not be the $F$-points of a closed subgroup of $\mathbf{G}$.)
Then $(Z\cap M'_\alpha)/(Z^0\cap M'_\alpha)\simeq \mathbb{Z}$ and we let $a_\alpha\in Z\cap M'_\alpha$ be a lift of a generator such that $\langle \alpha,v_Z(a_\alpha)\rangle < 0$ \cite[III.16~Notation]{MR3600042}.
Let $\Delta'(V)$ be the set of $\alpha\in\Delta(V)$ such that $\psi_V$ is trivial on $Z^0\cap M'_\alpha$.
The element $\tau^{V_{U^0},V_{U^0}}_{a_\alpha}$ is independent of the choice of $a_\alpha$ if $\alpha\in\Delta'(V)$.
For $z\in Z^+_G(V,V')$, note that
$$Z_z^+(V,V'):=Z^+\cap z\prod_{\alpha \in \Delta'(V)\cap \Delta'(V')}a_\alpha^\mathbb{N}$$
 is a finite subset of $Z_G^+(V,V')$ by Lemma~\ref{lm:Zz-contained-in-ZG}.
 
 \begin{theorem}[Inverse Satake theorem, Theorem \ref{EIST}]\label{EIST in intro}
  
A $C$-basis of the   image of  $S^G$ is given by the elements
\begin{equation}\label{image in intro}
\tau_z^{V'_{U^0},V_{U^0}}    \prod _{\alpha \in \Delta'(V') \setminus  \Delta'(V) }(1 -\tau_{a_\alpha}^{V_{U^0},V_{U^0}} )
\end{equation} for  $z$ running through a system of representatives of $Z_G^+(V,V')/Z^0$ in $Z_G^+(V,V')$.
  
 A $C$-basis of $\mathcal H_G(V,V')$  is given by the elements
 $$\varphi_z= \sum _{x\in Z^+_z(V,V')} T_x^{V',V}$$
  for $z$ running through a system of representatives of $Z_G^+ (V,V')/Z^0$.
  
For   $z\in Z_G^+ (V,V')$ we have:
  $$S^G(\varphi_z)= \tau_z^{V'_{U^0},V_{U^0}}\prod_{\alpha \in \Delta'(V')\setminus \Delta'(V)}(1  -\tau_{a_\alpha}^{V_{U^0},V_{U^0}}).$$   \end{theorem}

When  $\Delta'(V')\subset \Delta'(V)$, the convention is that $\prod_{\alpha \in \Delta'(V')\setminus \Delta'(V)}(1  -\tau_{a_\alpha}^{V_{U^0},V_{U^0}}) = 1$. 
  
There is  a Satake transform  $  S^G_{M} :\mathcal H_G( V,V') \to \mathcal H_M( V_{N\cap K}, V'_{N\cap K}) $ for  any parabolic subgroup ${\bf P}={\bf MN}$  containing $\bf B$ with   Levi  subgroup $\bf M$ containing $\bf Z$ \cite[Prop.\ 2.2, 2.3]{MR3001801} with $M = \mathbf{M}(F)$ and $N = \mathbf{N}(F)$.
We  compute also  $S^G_{M} (\varphi_z)$ (Theorem \ref{thm:satake for Levi}).

\subsection{}
From the above theorem, we can easily deduce the following result which implies the change of weight theorem (cf.\ Section~\ref{subsec:Change of weight}).
Suppose that $V,V'$ satisfies that $\psi_V = \psi_{V'}$ and $\Delta(V) = \Delta(V')\sqcup\{\alpha\}$ for some $\alpha\in\Delta$.
Let $Z_{\psi_V}^+$ the subset of $Z^+$ consisting of the elements which normalize $\psi_V$.
Define $c_\alpha$ by
\[
c_\alpha = 
\begin{cases}
1 & \text{if $\alpha\in\Delta'(V)$},\\
0 & \text{otherwise}.
\end{cases}
\]

\begin{theorem}[Theorem~\ref{weight2}]\label{thm2 in intro}
Let $z\in Z_{\psi_V}^+$ such that $\langle \alpha, v_Z(z)\rangle >0$.
Then there exist $G$-equivariant homomorphisms $\varphi:\ind_K^GV\to \ind_K^GV'$ and $\varphi':\ind_K^GV'\to \ind_K^GV$ satisfying
 \[
S^G(\varphi \circ \varphi')=\tau_{z^2}^{V'_{U^0},V'_{U^0}}  - c_\alpha \, \tau^{V'_{U^0},V'_{U^0}} _{ z^2a_\alpha },\quad 
 S^G(\varphi' \circ \varphi) =\tau_{z^2}^{V_{U^0},V_{U^0}}  - c_\alpha \, \tau^{V_{U^0},V_{U^0}} _{ z^2a_\alpha }.
\]
\end{theorem}

In Section~\ref{sec:simple-proof-change} we give a simple proof of Theorem \ref{thm2 in intro} (and hence of the change of weight theorem) when $\mathbf{G}$ is split.
It is more elementary than the other proofs we know in this case.
In particular, we do not use the pro-$p$ Iwahori Hecke algebra or the Lusztig-Kato formula.
In the proof we first reduce to the case where $\mathbf{G}$ has simply-connected derived subgroup and connected center, and $v_Z(z)$ is minuscule.
We construct many parabolically induced representations which contain $V$ but not $V'$.
From this we deduce that if $\varphi = T_z^{V',V}$ and $\varphi' = T_z^{V,V'}$, then $S^G(\varphi'\circ\varphi)$ is so constrained that it is forced to be equal to $\tau_{z^2}^{V_{U^0},V_{U^0}} - \tau_{z^2a_\alpha}^{V_{U^0},V_{U^0}}$.

In the appendix, two of us (N.A.\ and F.H.) show that the simple proof of the change of weight theorem
can be made to work, with some effort, for all quasi-split groups $\mathbf{G}$, at least for most choices
of special parahoric subgroup $K$. We do not know a simple proof for general $\mathbf{G}$ (or for
the remaining choices of $K$ when $\mathbf G$ is quasi-split), partly because
the method seems less powerful in the case where $c_\alpha = 0$.

\subsection{}
We  briefly explain the strategy of the proof of Theorem~\ref{EIST in intro}.
In \cite{MR2845621} when $\mathbf{G}$ is split and the derived subgroup is simply-connected, we assumed $V = V'$ and first made a reduction to the case where $\dim V = 1$.
Since $\mathbf{G}$ is split, the character $V$ of $K$ can be extended to a character of $G$ which allows us to reduce to the case where $V$ is trivial and use the characteristic zero formula of Lusztig-Kato.
This argument cannot work for general $\mathbf{G}$ since a character of $K$ need not extend to $G$.
For example, this can happen when $G = D^\times$ where $D$ is a (non-commutative) division algebra over $F$.

In our proof, we treat arbitrary pairs $(V,V')$.
First we make a reduction to the case where $\Delta(V')\subset\Delta(V)$ using properties of Satake transform and the convolution of Hecke operators (Lemmas~\ref{first}, \ref{second}).
When $\Delta(V')\subset \Delta(V)$, using a calculation in \cite[\S IV]{MR3600042}, we can express the inverse of the Satake transform using an alcove-walk basis of the pro-$p$ Iwahori Hecke algebra (Proposition~\ref{begin}).
Combining this with an explicit calculation of the alcove-walk basis (Proposition~\ref{EJ}), we get Theorem~\ref{EIST in intro}. More details are given below.

\subsection{} \label{S:1.14}
 Let  $\mathcal H_G$ be the Hecke $\mathbb Z$-algebra of the pro-$p$ Iwahori group  $I=K(1)U_{\op}^0$, 
 where $K(1)$ is the pro-$p$ radical of $K$ and   $U_{\op}^0=K\cap U_{\op}$, where $\mathbf{U}_{\op}$ is the opposite to $\bf U$ (with respect to $\mathbf Z$).  We also let $Z(1) = Z\cap K(1)$.
Until the end of this introduction we assume  $\Delta (V')\subset \Delta(V)$ and $z\in Z_G^+(V,V')$.  We now explain how the theory of $\mathcal H_G$ allows us to prove
$$\tau_z^{V'_{U^0},V_{U^0}}=S^G(\varphi_z)$$ 
in Theorem~\ref{EIST in intro}, hence the inverse Satake theorem.

 Once we choose a non-zero element $v\in V_{U^0}$ and let $v'\in V'_{U_0}$ correspond to $v$ under our fixed isomorphism $\iota : V_{U^0}\simeq V'_{U^0}$, we define embeddings 
$$\ind_K^G V \xrightarrow{I_v} \mathfrak X _G, \quad \ind_K^G V'  \xrightarrow{I_{v'}}\mathfrak X_G,\quad  \ind_{Z^0}^Z V_{U^0}\xrightarrow{j_v} \mathfrak X_Z ,\quad  \ind_{Z^0}^Z V'_{U^0}\xrightarrow{j_{v'}} \mathfrak X_Z,  $$
of $\ind_K^G V$ and $ \ind_K^G V'  $ in the  parabolically induced representation  $ \mathfrak X_G= \Ind_B^G( \ind_{Z(1)}^Z C)$ and of 
 $\ind_{Z^0}^Z V_{U^0}$ and $\ind_{Z^0}^Z V'_{U^0}$  in $\mathfrak X_Z=   \ind_{Z(1)}^Z C$. We have
 $$I_v=(\Ind_B^G j_v ) \circ  I_V,\quad  I_{v'}=(\Ind_B^G j_{v'} ) \circ I_{V'}$$
  for   the canonical
$C[G]$-embedding $\ind_K^G V\xrightarrow{I_V}  \Ind_B^G(\ind_{Z^0}^Z V_{U^0})$   \cite{MR3001801}, and similarly for $I_{V'}$. 
  The representation $\ind_K^G V$ is generated by the $I$-invariant element $f_v$, which is supported on $K$ and is such that
$f_v(1)$ lies in $V^{U_{\op}^0}$ and maps to $v \in V_{U^0}$. Similarly for $f_{v'}\in \ind_K^GV'$. 

  Then,  $ I_v(f_v), I_{v'}(f_{v'} )$ lie in the  $(\mathcal{H}_Z,\mathcal{H}_G)$-bimodule $\mathfrak X_G^I =( \Ind_B^G( \ind_{Z(1)}^Z C))^I$.  
Let  $\tau(z)\in \mathcal{H}_Z$ be the characteristic function of $zZ(1)$.  

The first key ingredient is Proposition \ref{end2} (which  generalizes \cite[IV.19 Thm.]{MR3600042}): 
$$\text{\it We give  an explicit element $h_z\in \mathcal{H}_G$ such that $\tau(z)I_v(f_v)= I_{v'}(f_{v'})h_z$}.$$
 We deduce  (Proposition \ref{begin}):   there exists an intertwiner $\phi_z:\ind_K^G V\to \ind_K^G V'$ defined by  
 $$\phi_z(f_v) = f_{v'} h_{z} .$$  
Moreover,  $\tau_z^{V'_{U^0},V_{U^0}}=S^G(\phi_z)$.    
The second key ingredient is the computation of   $f_{v'} h_z \in (\Ind_K^G V')^I$  on $Z^+$: 
$$\text{\it  The function $f_{v'} h_z $ vanishes on $Z^+\setminus Z^0 Z^+_z(V,V')$ and is equal to $v'$ on $Z^+_z(V,V')$}.$$
We prove that it implies $\varphi_z=\phi_z$  (proof of Proposition \ref{expa}).

\subsection{}   We develop in Section~\ref{S:4} the theory of the pro-$p$ Iwahori Hecke algebra $\mathcal{H}_G$ behind the  computation of  $f_{v'} h_z |_{Z^+}$.

 Let    $\mathcal N$ be the $G$-normalizer of $Z$, $W(1)=\mathcal N/ Z(1)$  the   pro-$p$ Iwahori Weyl group, $\lambda_x \in W(1)$ the   image of $x\in Z$ and  $Z_k $ the image on $Z^0$ in $W(1)$.  It is well known that the natural map 
$W(1) \to I\backslash G/I$ is bijective.  
The element $h_z\in \mathcal{H}_G$ is given as a product 
 (Propositions  \ref{begin}, \ref{end2}):  
$$ h_z =E'_{\lambda_z w_{V,V'}^{-1}}T^*_{w_{V,V'}},$$
where   $ (E'_w)_{w \in W(1)}$ is a certain alcove walk basis of $\mathcal{H}_G$ (which depends on $V'$),  $(T^*_w)_{w \in W(1)}$ a non alcove walk basis of $\mathcal{H}_G$,  and  $w_{V,V'}\in W(1)$ is a   lift of the product in $\mathcal N/Z$  of the longest elements of the finite Weyl groups associated to $\Delta(V)$ and $\Delta(V')$. 
  
   The  two bases are related  by   triangular matrices to the   classical Iwahori-Matsumoto basis $(T_w)_{w \in W(1)}$  of $\mathcal{H}_G$,   where $T_w$ is the characteristic function of  $InI$ for $n\in \mathcal N$ lifting $w$.  
 We have $$T_w^*=\sum_{u\in W(1), u\leq w}c^*(w,u)T_u$$ with coefficients $c^*(w,u)\in C$ and $c^*(w,w)=1$, where $\leq$ is the Bruhat (pre)order on $W(1) $ associated to $B$  (see \eqref{T*}). Let $M$ be the Levi subgroup of $G$   containing $Z$ associated to $\Delta(V')$; an  index $M$ indicates an object relative to $M$ instead of $G$.  
  It was a surprise to  discover (partially following an idea of Ollivier~\cite{MR3263136})  that the coefficients of the expansion of the alcove walk element  $E'_{\lambda_z w_{V,V'}^{-1}}$  in the   classical  basis of $\mathcal{H}_G$
 are given by  the coefficients  $c ^{M,*}(\lambda_z,u)$ of  the expansion of the non alcove walk basis element $T^{M, *} _{\lambda_z}\in \mathcal{H}_{M}$  in the classical basis  $(T^{M} _w)_{w\in W_{M(1)}}$ of $\mathcal{H}_{M}$. Recall that  $\mathcal{H}_{M}$ is not a subalgebra of $\mathcal{H}_G$, and that 
 the restriction to $W_{M}(1)$ of the Bruhat order $\leq$ on $W(1)$ is not equal to the Bruhat order   $\leq^{M}$ associated to $B_{M}=M\cap B$.  We show
  (Proposition \ref{EJ}):
$$E'_{\lambda_z w_{V,V'}^{-1}} = \sum_{u\in W_{M}(1), \ u\leq^{M} \lambda_z} c ^{M,*}(\lambda_z,u) T_{u w_{V,V'}^{-1}} .$$
  We carry out a detailed study  of   the sum
 $\sum_{t\in Z_k} c^*(w,tu)T_t$ modulo $q = \# k$ for $w,u\in W(1), u\leq w$. In particular,   we show (Theorems  \ref{*}, \ref{psic}),  for a character $\psi:Z_k\to C^\times$:
 $$\text{\it For} \  x\in Z^+ \ \text{\it and} \  \lambda_x\leq \lambda_z,  \ \text{\it we have} \  \sum_{t\in Z_k} c^*(\lambda_z ,t\lambda_x)\psi(t)= \begin{cases} 1 \ &\text{if } \ x\in Z^0z\prod _{\alpha \in \Delta'_{\psi}}a_\alpha^{\mathbb N},\\
  0 \ &\text{\it otherwise}.
  \end{cases}
  $$
Here $\Delta'_\psi=\{\alpha \in \Delta \ | \ \psi \ \text{is trivial on }\ Z^0\cap M'_\alpha\}$.  With a ``little more'' we deduce that on $Z^+$, 
\begin{equation*}
f_{v'} E'_{\lambda_z w_{V,V'}^{-1}}T^*_{w_{V,V'}}= f_{v'} \sum_{x\in Z_z^+(V,V')} \sum_{t\in Z_k} c ^{M,*}(\lambda_z,t\lambda_x) \psi_{V'}^{-1}(t)T_{\lambda_x}= f_{v'}\sum_{x\in Z_z^+(V,V')}T_{\lambda_x}.
\end{equation*}
 By the  ``little more'', we mean: if $u\in W_{M}(1)$ and $f_{v'}  T_{u w_{V,V'}^{-1}} T^*_{w_{V,V'}}$ does not vanish on $Z^+$ then $u\in Z^+/Z(1)$  (see \eqref{claim1}). The two conditions $u\in Z^+/Z(1)$ and $ u\leq^{M} \lambda_z$ are equivalent to $u=\lambda_x$  for   $x\in Z^0Z_z^+(V,V')$   (Proposition \ref{L+order}). For $x\in Z^0Z_z^+(V,V')$, we have   $f_{v'} T_{\lambda_x w_{V,V'}^{-1}}T^*_{w_{V,V'}}= f_{v'} T_{\lambda_x w_{V,V'}^{-1}}T_{w_{V,V'}}$ on $Z^+$  (see \eqref{claim2}). Then we use the braid relation  $T_{\lambda_x w_{V,V'}^{-1}}T_{w_{V,V'}}= T_{\lambda_x}$, that  $f_{v'} T_{t\lambda_x}= \psi_{V'}^{-1}(t)f_{v'}T_{\lambda_x}$ for $t\in Z_k$, and that $\Delta^{M}_{\psi_{V'}^{-1}}=\Delta' (V')=\Delta'(V)\cap \Delta' (V')$. 

From $f_{v'} h_z =f_{v'}\sum_{x\in Z_z^+(V,V')}T_{\lambda_x}$ on $Z^+$ --
and checking easily
that $ f_{v'} T_{\lambda_x}$ is supported on $KxI $ with value $v'$ at $x$, and $Z^+\cap KxI=Z^0x$,  for all $ x\in Z_z^+(V,V')$ -- we obtain  the desired value of $ f_{v'} h_z$ on $Z^+$ (\S \ref{S:1.14}). 

\section{Change of weight  and  Inverse Satake isomorphism}
\label{sec:change-weight-invers} 
 
\subsection{Notation}\label{Notation}Throughout this paper we follow the notation given in \cite{MR3600042}. As in loc.\ cit., let $F$ be a nonarchimedean field with ring of integers $\mathcal O$ and residue field $k$  of characteristic $p$ and cardinality $q$. Let $\ord_F : F^\times \to \Z$ denote the normalized valuation of $F$. A linear algebraic $F$-group is denoted with a boldface letter like $\bf H$ and the group of its $F$-points with the corresponding ordinary letter  $H={\bf H}(F)$; we use the similar convention for groups over $k$.  
   Let $\bf G$ be a connected reductive $F$-group.
   
   We fix a triple $({\bf S}, {\bf B}, x_0)$  where $\bf S$ is  a maximal torus in $\bf G$,   $\bf B$ a minimal $F$-parabolic subgroup of $\bf G$ containing $\mathbf{S}$ with unipotent radical $\bf U$ and Levi subgroup   the centralizer $\bf Z$ of $\bf S$ in $\bf G$, and $x_0$ 
     a special point in the apartment corresponding to $ S$ in the adjoint Bruhat-Tits building of $G$.

We write  $\boldN$ for the normalizer of $\bf S$ in $\bf G$. If $X^*(\bf S)$ is the group of characters of $\bf S$ and  $X_*(\bf S)$ is the group of cocharacters,
  we write $\langle \ , \ \rangle: X^*(\mathbf{S})\times X_*(\mathbf{S})\to \mathbb Z$ for the natural pairing. We let $\Phi \subset X^*(\mathbf{S})$ be the set of roots of $\bf S$ in $\bf G$ and we write $\Delta$ for the set of simple roots in the set $\Phi^+$ of positive roots with respect to $\bf B$. For $\alpha \in \Phi$, the corresponding coroot in $X_*({\bf S})$ is denoted by $\alpha^\vee$. For $\alpha, \beta\in \Phi$, we say that $\alpha$ is orthogonal to $\beta$ if and only if $\langle \alpha, \beta^\vee \rangle=0$. The Weyl group $W_0:=\cn/Z \simeq \boldN/\bf Z$ is isomorphic to the Weyl group of $\Phi$. 
  
  We say that $P$ is a parabolic subgroup of $G$ to mean that $P={\bf P}(F)$ where $\bf P$ is an $F$-parabolic subgroup of $\bf G$.   If $P$ contains $B$, we write $P=MN$ to mean that $N$ is the unipotent radical of $P$ and $M$ the (unique) Levi component containing $Z$; we   write $P_{\op}=MN_{\op}$ for the parabolic subgroup opposite to $P$ with respect to $M$. The parabolic subgroups containing $B$ are in one-to-one correspondence  with the subsets of $\Delta$; we denote by $P_J=M_JN_J$ the  group corresponding to $J\subset \Delta$  (when $J=\{\alpha\}$ we write simply $P_{\alpha}=M_\alpha N_\alpha$). 
  
     The apartment corresponding to $ S$ in the adjoint Bruhat-Tits building of $G$ is an affine space $x_0+ V_{\ad}$ where $V_{\ad}:= X_*({\bf S}_{\ad})\otimes \mathbb R$ and  $ {\bf S_{\ad}}$ is the torus image of $\bf S$ in the adjoint group  $\bf G_{\ad}$ of $\bf G$.  The group $\cn$ acts by affine automorphisms on the apartment, its subgroup $Z$ acting by translation by  $\nu=-v$ where $v:  Z\to V_{\ad}$ is the composite of  the map $v_Z:Z\to X_*(\bf S)\otimes \mathbb R$  defined in \cite[3.2]{MR3331726} and  of the natural quotient map $X_*({\bf S})\otimes \mathbb R\onto X_*({\bf S}_{\ad})\otimes \mathbb R$.  
(We recall that $v_Z$ is determined by the requirement that $\langle\chi,v_Z\rangle = \ord_F \circ \chi$ for all $F$-rational characters $\chi$ of $\bf Z$.)
The root system of $\bf S_{\ad}$ in $\bf G_{\ad}$ identifies with $\Phi$.  The coroot of $\alpha\in \Phi$ in $V_{\ad}$ is the image of the coroot $\alpha^\vee\in X_*(\bf S)\otimes \mathbb R$ by the quotient map, and is still denoted by $\alpha^\vee$. 
     
   As in \cite[I.5]{MR3600042} we write $K$ for the special parahoric subgroup of $G$ fixing $x_0$ and $K(1)$ for the pro-$p$ radical of $K$. For a subgroup $H$ of $G$, we put $H^0:=H\cap K$ and $\overline H :=(H\cap K)/(H\cap K(1))$. The group $S^0$ is the maximal compact subgroup of $S$,  $Z^0$ is the unique parahoric subgroup of $Z$ and   $Z(1):=Z\cap K(1)$ is the unique pro-$p$ Sylow subgroup of $Z^0$.
  The group $G_k:=\overline G=\overline K$ is naturally the group of $k$-points of a connected reductive $k$-group $\mathbf G_k$, of minimal parabolic subgroup  $B_k:=\overline B$ with Levi decomposition   $B_k=Z_kU_k$ where   $Z_k := \overline Z$ and $U_k:=\overline U$. The set of simple roots of the maximal split torus $S_k=\overline S$ of $G_k$ with respect to $B_k$ is  in natural bijection with $\Delta$  and   will be identified  with $\Delta$. For $J\subset \Delta$, the corresponding parabolic subgroup $P_{J,k}$ of $G_k$ containing $B_k$ is $ \overline P_J $; its Levi decomposition is $P_{J,k}=M_{J,k}N_{J,k} $ where $M_{k ,J}= \overline M_J$ and $N_{J,k}=\overline N_J$. We write $P_{J,k,\op}=M_{J,k}N_{J,k,\op} $ for the parabolic  group  opposite to 
$P_{J,k}$ with respect to    $M_{J,k}$.

We fix an algebraically closed field $C$ of characteristic $p$. In this paper, a representation  means a smooth representation on a $C$-vector space.

  \subsection{The Satake transform \texorpdfstring{$S_M^G$}{S\_M\^{}G}}  \label{GS}
  Let $V$ be an irreducible  representation of the special parahoric subgroup $K$ of $G$; the normal pro-$p$ subgroup $K(1)$ of $K$ acts trivially on $V$ and the action of $K$ on $V$ factors through the finite reductive group $G_k$. Seeing $V$ as an irreducible representation of $G_k$, we attach to $V$  a character $\psi_V$ of $Z_k$ and a subset $\Delta(V)\subset \Delta$ as in  \cite[III.9]{MR3600042}; the space of $U_k$-coinvariants $V_{U_k}$ of $V$  is a line on which $Z_k$ acts by $\psi_V$  and the $G_k$-stabilizer of the kernel of the natural map $V\to V_{U_k}$ is $P_{\Delta(V),k}$. The pair $(\psi_V, \Delta(V))$, called the parameter of $V$, determines $V$. The character  $\psi_V$ can be seen as the character of $Z^0$  acting on the space $U^0$-coinvariants $V_{U^0}$ of $V$.  
    
   Let $P=MN$ be the parabolic subgroup of $G$ containing $B$ corresponding to $J\subset \Delta$. Then $M^0$ is a special parahoric subgroup of $M$ and $V_{N^0}$ is an irreducible representation of $M^0$ with parameter $(\psi_V, J\cap \Delta (V))$ \cite[III.10]{MR3600042}.%
 
  The compact induction   $\ind_{K}^GV$ of $V$ to $G$ is the representation of $G$ by right translation on the space of functions $f:G\to V$ with compact support satisfying $f(kg)=kf(g)$ for all $k\in K, g
  \in G$.  We view the intertwining algebra $\End_{CG}(\ind_K^GV)$ as the convolution algebra $\mathcal H_G(V)$ of compactly supported functions $\varphi:G\to \End_C(V)$ satisfying $\varphi(k_1 g k_2)=k_1 \varphi(g) k_2$ for all $k_1,k_2\in K, g\in G$. The action of $\varphi\in  \mathcal H_G(V)$ on $f\in \ind_{K}^G(V)$ is given by convolution
  \begin{equation}\label{action0}
 ( \varphi * f)(g)= \sum_{x\in G/K} \varphi(x)( f(x^{-1}g)).
  \end{equation}
  We have also the algebra $\End_{CM}(\ind_{M^0}^M(V_{N^0}))\simeq \mathcal H_M(V_{N^0})$. The Satake transform  is a natural injective algebra homomorphism   \cite[III.3]{MR3600042} 
   \begin{equation*}  
    S_M^G: \mathcal H_G(V) \hookrightarrow \mathcal H_M(V_{N^0});
  \end{equation*}
  it induces an homomorphism between the centers $\mathcal Z_G(V) \to \mathcal Z_M(V_{N^0})$; 
  both homomorphisms are localizations at a central element  \cite[I.5]{MR3600042}.    

 For a representation $\sigma$ of $M$, the parabolic induction  $\Ind_P^G \sigma$   of $\sigma$ to $G$  is the representation of $G$ by right translation on the space of functions $f:G\to \sigma $ satisfying $f(mn gk)=mf(g)$ for all $m\in M, n\in N, g
  \in G, k$ in some open compact subgroup of $G$ depending on $f$.   The canonical isomorphism
  \begin{equation*} 
  \Hom_{CG}(\ind_K^G V, \Ind_P^G\sigma ) \xrightarrow{\sim}  \Hom_{CM}(\ind_{M^0}^M V_{N^0},  \sigma)
  \end{equation*} 
  is  $\mathcal H_G(V)$-equivariant via  $S_M^G$ \cite[\S 2]{MR3001801}.

\subsection{The Satake transform \texorpdfstring{$S^G=S^G_Z$}{S\^{}G=S\_Z\^{}G}}\label{sec:satake-transf}  As in \cite[III.4]{MR3600042}, the algebra $\mathcal H_Z(V_{U^0})$ is easily described. The unique parahoric subgroup $Z^0$ of $Z$  being normal, for $z\in Z$ we have the character $z\cdot\psi_V$ of $Z^0$ defined by $(z\cdot\psi_V)(x)= \psi_V (z^{-1}x z) , x\in Z^0$. Let $$Z_{\psi_V}=\{z\in Z \ | z\cdot\psi_V=\psi_V\} $$ be the $Z$-normalizer of $\psi_V$.  For $z\in Z_{\psi_V}$, there is a unique function $\tau_z\in  \mathcal H_Z(V_{U^0})$ of support $zZ^0$ with $\tau_z(z) = \id_{V_{U^0}}$. 
A basis of $\mathcal H_Z(V_{U^0})$  is given by  the functions $\tau_z$
where $z$ runs through a system of representatives of $Z_{\psi_V}/Z^0$ in $Z_{\psi_V}$.
The multiplication satisfies $\tau_{z_1} * \tau _{z_2}=\tau_{z_1 z_2} $.  The function $\tau_z$ belongs to the center $\mathcal Z_Z(V_{U^0})$  if and only if $\psi_V (z^{-1}x z x^{-1})=1$ for all $x\in Z_{\psi_V}$. We write also $\tau_z=\tau_z^{V_{U^0}}$.

 Let 
\begin{equation*} Z^+ =\{z\in Z \  | \ \langle \alpha, v_Z (z)\rangle \geq 0 \ \text{for all} \ \alpha\in \Delta\}.
\end{equation*}
be the dominant submonoid of $Z$. 
For a subset $H$ of $Z$  we write $H^+=H\cap Z^+$. 

When $M=Z$ we put $S^G=S_Z^G$.  The image of  $S^G$ is 
\begin{equation}\label{imS}S^G(\mathcal H_G(V))= \bigoplus_z C \tau_z \end{equation}
for $z$ in a system of representatives of $Z^+_{\psi_V}/Z^0$ in $Z^+_{\psi_V}$ (see \cite{bib:satake} when $\mathbf G$ is unramified and \cite{MR3331726} in general). 
 For another irreducible representation $V'$ of $K$ with $\psi_V=\psi_{V'}$,  we have a canonical $Z^0$-equivariant isomorphism $\End_{C}(V_{U^0}) \simeq \End_{C}(V'_{U^0})$
and hence a canonical isomorphism   $i_Z:  \mathcal H_Z(V_{U^0})\xrightarrow{\simeq} \mathcal H_Z(V'_{U^0})$  (sending  the function  $\tau _z\in \mathcal H_Z(V_{U^0})$ to the function  $\tau_z\in \mathcal H_Z(V'_{U^0})$ for all $z\in Z_{\psi_V}$). It 
induces a  canonical isomorphism \begin{equation}\label{iso}
 i_G: \mathcal H_G(V)\xrightarrow{\simeq} \mathcal H_G(V')
 \end{equation}
satisfying $S^G \circ i_G = i_Z \circ S^G$.

\subsection{The elements \texorpdfstring{$a_\alpha $}{a\_alpha}} 
\label{sec:elem-a-alpha}

 Let $G'$ be the group generated by $U$ and $U_{\op}$ (this is not the group of $F$-points of a linear algebraic group in general). The action of $\cn $ on the apartment $x_0+V_{\ad}$  induces an isomorphism from $(\cn\cap G')/(Z^0\cap G')$ onto the affine Weyl group $W^{\aff}$ of a reduced root system 
  \begin{equation}\label{Phia}
  \Phi_a =\{\alpha_a:=e_\alpha \alpha \ | \ \alpha \in \Phi\} 
   \end{equation} 
on $V_{\ad}$,     where 
 $e_\alpha$ for $\alpha\in \Phi$ are positive integers  \cite[Lemma 3.9]{MR3484112}, \cite[VI.2.1]{MR1890629}. The map $\alpha \to \alpha_a $ gives a bijection from $\Delta$ to a set  $\Delta_a$ of simple roots of $\Phi_a$; the coroot  in  $X_*({\bf S}_{\ad})\otimes \mathbb R$ associated to $\alpha_a$ is $\alpha_a^\vee=e_\alpha ^{-1}\alpha^\vee$; the homomorphism $\nu=-v: Z\to V_{\ad}$ induces a quotient map 
  $ Z\cap G'\onto  \oplus_{\alpha \in \Delta} \mathbb Z  \alpha_a ^\vee $  with kernel $Z^0\cap G'$.
 An element $z\in Z$ belongs to $Z^+$ if and only if $\nu(z)$ lies in the closed antidominant Weyl chamber 
 \begin{equation}\label{D-}
 \mathfrak D^-=\{x\in V_{\ad} \ | \ \langle \alpha_a, x\rangle\leq 0 \  \text{for} \ \alpha \in \Delta\}.
 \end{equation} 
 For $\alpha \in \Delta$  we also have $M'_\alpha$ and  the quotient map 
  $ Z\cap M'_\alpha\onto  \mathbb Z  \alpha_a ^\vee $  with kernel $Z^0\cap M'_\alpha$ induced by $\nu$ \cite[III.16]{MR3600042}.
\begin{definition}\label{delta'psi} For a character $\psi  :Z^0\to C^\times$ and $\alpha \in \Delta$,   let 
\begin{align*}&\Delta'_\psi=\{\alpha \in \Delta \ | \ \psi \ \text{ is trivial on } \ Z^0\cap M'_\alpha\},\\
&\text{$a_\alpha \in  Z\cap M'_\alpha$ such that $\nu(a_\alpha)= \alpha_a^\vee$.}
\end{align*}
\end{definition}
 If $\alpha\in \Delta'_\psi$, then $ Z\cap M'_\alpha$ is contained in the $Z$-normalizer $Z_\psi$  of $\psi$, 
   \begin{equation*} \tau_\alpha :=\tau_{a_\alpha}\in \mathcal H_Z(\psi)
   \end{equation*}
     does not depend on the choice of $a_\alpha$, and belongs to the center   $\mathcal Z_Z(\psi)$ \cite[III.16]{MR3600042}. The set $\Delta'_\psi $ is  included in the subset $ \Delta(\psi)$ of $\Delta$ defined by \eqref{deltapsi} (cf.\ Remark~\ref{a-1}).

 \subsection{Change of weight}\label{subsec:Change of weight}
Let $V'$ and $V$ be two irreducible representations of $K$ with parameters     $\psi_V=\psi_{V'} , \Delta(V)=\Delta(V')\sqcup\{\alpha\}$ where $\alpha \in \Delta-\Delta(V')$,  let $\chi:\mathcal Z_G(V)\to C$ be a character of the center of $\mathcal H_G(V)$,  let $P=MN$ denote the smallest parabolic subgroup of $G$ containing $B$ such that $\chi$ factors   through  $S^G_M$, and let  $\Delta(\chi)$ be the  subset of $\Delta$ corresponding to $P$ (denoted by $\Delta_0(\chi)$ in \cite[III.4 Notation]{MR3600042}). We have the  homomorphism $\chi': \mathcal Z_G(V')\to C$ corresponding to $\chi$ via the isomorphism \eqref{iso}. 
 
 \begin{theorem}[Change of weight]\label{weight}
 Assume $\alpha \not\in \Delta(\chi)$.
 The representations  $\chi \otimes_{\mathcal Z_G(V)} \ind_K^GV$ and $ \chi' \otimes_{\mathcal Z_G(V')} \ind_K^GV' $ of $G$ are isomorphic unless  
 $$\alpha \ \text{is orthogonal to} \ \Delta(\chi),\   \psi_V  \ \text{is  trivial on} \ Z^0\cap M'_\alpha, \ \chi (\tau_\alpha)=1.$$
 \end{theorem}

The change of weight theorem was proved in \cite[IV.2 Corollary]{MR3600042} (generalizing \cite{MR2845621} for $\GL_n$ and \cite{MR3143708}
for split groups) and was one of the key tools in establishing a classification result for irreducible representations of $G$ over $C$. 
The change of weight theorem  is a simple consequence of the next theorem. Define 
\begin{equation}\label{dicho}c_\alpha=\begin{cases} 1 & \text{if} \  \psi_V  \ \text{is  trivial on} \ Z^0\cap M'_\alpha,\\
0  & \text{otherwise}.
\end{cases}
\end{equation}

 \begin{theorem}\label{weight2} Let $z\in Z_{\psi_V}^+$ such that $\langle \alpha, v(z)\rangle >0$.
Then there exist $G$-equivariant homomorphisms $\varphi:\ind_K^GV\to \ind_K^GV'$ and $\varphi':\ind_K^GV'\to \ind_K^GV$ satisfying
 $$S^G(\varphi \circ \varphi')=\tau_{z^2}^{V'_{U^0}}  - c_\alpha \, \tau^{V'_{U^0}} _{ z^2a_\alpha },\quad 
 S^G(\varphi' \circ \varphi) =\tau_{z^2}^{V_{U^0}}  - c_\alpha \, \tau^{V_{U^0}} _{ z^2a_\alpha }.$$
  \end{theorem}

We will prove in Proposition \ref{ISimpliesCW} that 
Theorem \ref{weight2}  follows from  the   inverse  Satake theorem (Theorem \ref{EIST}) for the   pair $(V,V')$ and for the pair $(V',V)$. 
We now recall why Theorem \ref{weight2} implies Theorem \ref{weight} (compare with the proof of \cite[IV.2 Corollary]{MR3600042}).

\begin{proof}[Proof of Theorem \ref{weight}]
  As in \S\ref{sec:satake-transf} we can canonically identify $\mathcal H_G(V)$ with $\mathcal H_G(V')$ and similarly
  $\mathcal Z_G(V)$ with $\mathcal Z_G(V')$, denoting them $\mathcal H_G$ and $\mathcal Z_G$ for short.  We also identify
  $\chi$ and $\chi'$. Pick any $z\in Z_{\psi_V}^+$ such that $\langle \alpha, v(z)\rangle >0$,
  $\langle \beta, v(z) \rangle= 0$ for all $\beta\in \Delta-\{\alpha\}$, and such that $\tau_{z^2} \in \mathcal Z_Z(\psi_V)$
  (cf.\ \cite[III.4]{MR3600042}).  As $S^G$ is injective and compatible with compositions, the homomorphisms
  $\varphi,\varphi'$ of Theorem~\ref{weight2} for our chosen $z$ are $\mathcal Z_G$-equivariant and induce $G$-equivariant
  homomorphisms between $\chi \otimes_{\mathcal Z_G} \ind_K^GV$ and $ \chi \otimes_{\mathcal Z_G} \ind_K^GV' $ with
  composition in either direction equal to $\chi(\tau_{z^2} - c_\alpha \, \tau _{ z^2a_\alpha }) \in C$. It suffices to show
  that $\chi(\tau_{z^2} - c_\alpha \, \tau _{ z^2a_\alpha }) \ne 0$. First, $\chi(\tau_{z^2}) \ne 0$ by \cite[III.4
  Lemma]{MR3600042} and as $\alpha \not\in \Delta(\chi)$, so we are done if $c_\alpha = 0$. For the same reason, if
  $c_\alpha = 1$ and $\alpha$ is not orthogonal to $\Delta(\chi)$, then $\chi(\tau_{z^2 a_{\alpha}}) = 0$ and we are
  done. Finally, if $c_\alpha = 1$, $\alpha$ is orthogonal to $\Delta(\chi)$, and $\chi(\tau_\alpha) \ne 1$, then
  $\chi(\tau_{z^2} - c_\alpha \, \tau _{ z^2a_\alpha }) = \chi(\tau_{z^2})(1-\chi(\tau_\alpha)) \ne 0$.
\end{proof}
 
\subsection{Intertwiners  from \texorpdfstring{$\ind_K^GV$}{ind\_K\^{}G V} to \texorpdfstring{$ \ind_K^GV'$}{ind\_K\^{}G V'}}\label{intVV'}
Let $V$ and $V'$ be  two irreducible representations  of $K$. We extend  to the  space of intertwiners $\Hom_{CG}(\ind_K^GV, \ind_K^GV')$ our previous discussion  on  $\End_{CG}(\ind_K^G V)$ in \S \ref{GS}. We view  $\Hom_{CG}(\ind_K^GV, \ind_K^GV')$ as the space  $\mathcal H_G(V,V')$ of compactly supported functions $\varphi:G\to \Hom_C(V,V')$ satisfying $\varphi(k_1 g k_2)=k_1 \varphi(g) k_2$ for all $k_1,k_2\in K, g\in G$. 
For $z\in Z$, we write 
\begin{equation}\label{deltaz}\Delta_z=\{\alpha \in \Delta \ | \ \langle 
\alpha, v(z)\rangle =0 \}.
\end{equation}
\begin{remark}\label{dz}  When $z,z'\in Z^+$, we have $\Delta_{z'z}=\Delta_{z'}\cap \Delta_z$.
\end{remark}
The quotient map $p:V\onto V_{U^0}$  induces a  $Z^0$-equivariant isomorphism between the lines  $V^{U_{\op}^0}\congto V_{U^0}$;  similarly for $V'$.  
We fix compatible  linear isomorphisms 
\begin{align}\label{iopi}
\iota^{\op}:V^{U_{\op}^0}\xrightarrow{\sim} (V')^{U_{\op}^0}  \ \text{and} \ \iota:V_{U^0}\xrightarrow{\sim}V'_{U^0}.
\end{align}
When $V=V'$ we  suppose that $\iota^{\op}$ and $\iota $ are the identity maps. 
We now recall  the description of $\mathcal H_G(V,V')$. 
By the Cartan decomposition \cite[6.4 Prop.]{MR3331726}, the map $Z\to K\backslash G/K, z\mapsto KzK$ induces a bijection $Z^+/Z^0\xrightarrow{\sim}K\backslash G/K$.  Recalling from \S \ref{GS} the parameters $(\psi_V, \Delta(V))$ of $V$ and  $(\psi_{V'}, \Delta(V'))$ of $V'$, a double coset $KzK$ with $z\in Z^+$ supports a non zero function of $\mathcal H_G(V,V')$ if and only if $z$ lies in 
\begin{align}\label{ZVV'}
Z_G^+(V,V')&=\{z\in Z^+\ | \ z\cdot\psi_V=\psi_{V'} \ \text{and}\ \Delta_z\cap (\Delta(V)\triangle \Delta(V')) = \varnothing\} \\ \label{triangle} 
&=\{z\in Z^+\ | \ z\cdot\psi_V=\psi_{V'} \ \text{and}\  \langle \alpha, v (z)\rangle >0 \ \text{for all} \ \alpha \in\Delta(V)\triangle\,  \Delta(V')\}. 
 \end{align}
where $\Delta(V)\triangle\Delta(V') = (\Delta(V)\setminus\Delta(V'))\cup(\Delta(V')\setminus\Delta(V))$ is the symmetric difference.

The space of such functions has dimension $1$ and contains a unique function $T_z$  such that the restriction of $T_z(z)$ to $V^{U^0_{\op}}$ is  $\iota^{\op}$. 
The function $T_z$ is also denoted by $T_z= T_z^{V',V}$ or $T_z^{V',V,\iota}$. 

\begin{proposition}[{\cite[7.7]{MR3331726}}]\label{basisVV'}A basis of $\mathcal H_G(V,V')$ consists of the  $T_z$  for $z$ running  through a system of representatives of $Z_G^+(V,V')/Z^0$ in $Z_G^+(V,V')$. 
\end{proposition}
We will write that $(T_z)_{z\in Z_G^+(V,V')/Z^0}$ is a basis of $\mathcal H_G(V,V')$.  

These considerations  apply  also to the group $Z$ and to the representations $V_{U^0}, V'_{U^0}$ of $Z^0$.
We write $Z_{\psi_V, \psi_{V'}} = \{ z \in Z \mid z\cdot\psi_V=\psi_{V'} \}$.
Then the function $\tau_z\in \mathcal H_Z(V_{U^0}, V'_{U^0})$ of support $Z^0z$ and value $\iota$ at $z$ for $z\in Z_{\psi_V, \psi_{V'}}$ is denoted also  by $\tau_z^{V'_{U^0}, V_{U^0}}$ or $\tau_z^{V'_{U^0}, V_{U^0},\iota}$. 
A basis of $\mathcal H_Z(V_{U^0}, V'_{U^0})$ is $(\tau_z)_{z\in Z_{\psi_V, \psi_{V'}}/Z^0}$. 
\begin{example}\label{ex} If  $V=V'$, then $Z_G^+(V,V)=  Z^+_{\psi_V}$. If  $\psi_V=\psi_{V'}$, then $Z_G^+(V,V')=Z_G^+(V',V)  \subset Z^+_{\psi_V}$. If $\Delta(V)=\Delta(V')$, then $Z_G^+(V,V')=Z^+_{\psi_V, \psi_{V'}}$.
\end{example}

\begin{remark}\label{not0} (i) We have $\mathcal H_G(V,V')\neq 0$ if and only if  
$Z_{\psi_V, \psi_{V'}} $ 
is not empty \cite[7.8 Prop.]{MR3331726}. In this case   $\Delta'_{ \psi_{V}}= \Delta'_{ \psi_{V'}}$ (Definition \ref{delta'psi})  because $Z^0\cap M'_\alpha$ is a normal subgroup of $Z$.

(ii) Let   $z\in Z_{\psi_V, \psi_{V'}} $, 
   $\alpha\in \Delta'_{\psi_V}= \Delta'_{\psi_{V'}}$ and    $a_\alpha\in Z\cap M'_\alpha$ (Definition \ref{delta'psi}). Then   $a_\alpha z a_\alpha ^{-1} z^{-1} \in Z^0\cap M'_\alpha$ ($Z\cap M'_\alpha$ is also a normal subgroup of $Z$) hence $za_\alpha=ta_\alpha z\in Z_{\psi_V, \psi_{V'}}$, some $t \in Z^0 \cap M_\alpha'$. The convolution satisfies
  \begin{align*}  \tau^{V'_{U^0}, V_{U^0}, \iota}_{z }  \tau^{V_{U^0}, V_{U^0}} _{\alpha}=\tau_{z a_\alpha }^{V'_{U^0}, V_{U^0}, \iota }= \tau_{ ta_\alpha z}^{V'_{U^0}, V_{U^0}, \iota }= \tau^{V'_{U^0}, V'_{U^0}} _{\alpha} \tau^{V'_{U^0}, V_{U^0}, \iota}_{z } .
  \end{align*} 
\end{remark}

Let  $V''$ be a third irreducible representation of $K$. The composition of intertwiners corresponds to the convolution.  We fix compatible  linear  $\iota'^{\op}:(V')^{U_{\op}^0}\xrightarrow{\sim} (V'')^{U_{\op}^0}$ and $\iota':V'_{U^0}\xrightarrow{\sim}V''_{U^0}$ and we  define  as above $T_{z}^{V'',V'}=T_{z}^{V'',V', \iota'}$ when $z\in Z_G^+(V',V'')$ and $T_{z}^{V'',V}=T_{z}^{V'',V, \iota' \circ \iota}$ when $z\in Z_G^+(V,V'')$.
 
For $g\in G$ we note that $(T_{z'}^{V'',V'}*T_z^{V',V})(g)$ equals
\begin{align*}
\sum_{x\in Kz'K/K}T_{z'}^{V'',V'}(x) \circ T_z^{V',V}(x^{-1}g)
&=\sum_{x\in K /(K\cap z' K z'^{-1})}T_{z'}^{V'',V'}(xz') \circ T_z^{V',V}(z'^{-1}x^{-1}g).
\end{align*}

\begin{remark}\label{zz'} (i) When $\psi_{V'}=\psi_{V''}$ and $\Delta(V)\cap \Delta(V')\subset \Delta(V'')\subset \Delta(V)\cup \Delta(V')$, we have $Z_G^+(V,V')\subset Z_G^+(V,V'')$.

(ii) For $z\in Z_G^+(V,V'), z'\in Z_G^+(V',V'')$ we have $z'z\in Z_G^+(V,V'')$ because $z'z\cdot\psi_V=z'\cdot\psi_{V'}=\psi_{V''}$, $\Delta_{z'z}=\Delta_z \cap \Delta_{z'}$ (as $z,z'\in Z^+$), and $\Delta(V) \triangle \Delta(V'') \subset (\Delta(V) \triangle \Delta(V')) \cup (\Delta(V') \triangle \Delta(V''))$. %

(iii) For $z\in Z_{\psi_V, \psi_{V'}} , z'\in Z_{\psi_{V'}, \psi_{V''}} $ we have  $\tau^{V''_{U^0}, V'_{U^0},\iota'}_{z '}  \tau^{V'_{U^0}, V_{U^0},\iota} _{z }=\tau_{z' z }^{V''_{U^0}, V_{U^0},\iota'\circ \iota }$. 
\end{remark}

We will later use the following lemma concerning the support of $S^G(T_{z}^{V',V})$.

\begin{lemma}\label{lm:support-satake}
  If $z \in Z_G^+(V,V')$, $z' \in Z$ and $S^G(T_{z}^{V',V})(z')\neq 0$, then $v_Z(z')\in v_Z(z) +
  \mathbb{R}_{\le 0}\Delta^\vee$.
\end{lemma}

\begin{proof}
  Letting $w_G$ denote the Kottwitz homomorphism, we have $\ker w_G = Z^0 G'$ \cite[Rk.\
  3.37]{MR3484112}.  If $S^G(T_z)(z') \ne 0$, then $z' \in Z \cap UKzK$, hence $w_G(z') = w_G(z)$,
  so $z' \in z \ker(w_G|_Z) = z Z^0 (Z \cap G')$. By \cite[II.6 Prop.]{MR3600042}\ with $I =
  \varnothing$ it follows that $Z \cap G'$ is generated by all $Z \cap M'_\alpha$ for $\alpha \in
  \Delta$.  As $v_Z(Z \cap M'_\alpha) = \Z v_Z(a_\alpha) \subset \R \alpha^\vee$, we see that
  $v_Z(z') \in v_Z(z) + \R \Delta^\vee$. By \cite[6.10 Prop.]{MR3331726}\ we deduce $v_Z(z') \in
  v_Z(z) + \mathbb R_{\le 0} \Delta^\vee$.
\end{proof}

\begin{remark}
  In fact, we know that $v_Z(a_\alpha) = -e_\alpha^{-1} \alpha^\vee$ \cite[IV.11 Example
  3]{MR3600042}.  So the the proof shows that $v_Z(z') \in v_Z(z) + \sum_{\alpha \in\Delta} \Z_{\le
    0} e_\alpha^{-1} \alpha^\vee$.  This improves on \cite[Lemma 3.6]{bib:satake} when $\mathbf G$ is
  unramified %
  and \cite[6.10 Prop.]{MR3331726}\ when $\mathbf G$ is general.
\end{remark}

 \subsection{The generalized Satake transform}\label{sec:gener-satake-transf}Let $P=MN$ be a parabolic subgroup of $G$ containing $B$. 
 
\begin{definition}[{\cite[Prop.\ 2.2 and 2.3]{MR3001801}, \cite[Prop.\ 7.9]{MR3331726}}] \label{SMGVV'}
The generalized Satake transform  is the  injective  linear   homomorphism  
\[
S_M^G: \mathcal H_G(V,V') \hookrightarrow \mathcal H_M(V_{N^0}, V'_{N^0})
\]
defined as follows. Let $\varphi\in \mathcal{H}_G(V,V'),m\in M$ and let $p : V\onto V_{N^0}$, $p'\colon V'\onto V'_{N^0}$ denote the natural quotient maps.
Then $S_M^G$ is determined by the relation
\[
(S_M^G\varphi)(m) \circ p = p'\circ\sum_{x\in N^0\backslash N}\varphi(xm).
\]
\end{definition}

For $\varphi \in \mathcal H_G(V,V')$ and  $\varphi' \in \mathcal H_G(V',V'')$ we have $S_M^G(\varphi' * \varphi)= S_M^G(\varphi') * S_M^G(\varphi)$ \cite[Formula (6)]{MR3001801}. 

\bigskip When $M=Z$, we write  $S^G=S_Z^G$.

\subsection{Inverse   Satake theorem}\label{sec:inverse-satake-theor}
We now give our main result. Let $V$ and $V'$ be irreducible representations of $K$. Our main theorem determines the  image of the Satake transform
\[S^G:\mathcal H_G(V,V')\hookrightarrow \mathcal H_Z(V_{U^0 },V'_{U^0 })\]
and moreover gives  an explicit formula for the inverse of $S^G$ on a basis of 
the image. (Of course this theorem is only interesting when $\mathcal H_G(V,V') \ne 0$. See Remark \ref{not0} for when this happens.)

We fix compatible isomorphisms $ \iota^{\op}:V^{U^0_{\op}}\to V'^{U^0_{\op}}$ and $\iota:V_{U^0 }\to V'_{U^0 }$ as in \eqref{iopi} and  $a_\alpha\in Z\cap M'_\alpha$  for $\alpha \in \Delta$ (Definition \ref{delta'psi}). 
 Recalling  $\Delta'_{\psi}$ (Definition \ref{delta'psi}), we denote
\begin{equation}\label{delta'V} \Delta'(V)=\Delta(V) \cap \Delta'_{\psi_V} =\{\alpha \in \Delta(V) \ | \ \psi_V \ \text{is trivial on} \ Z^0\cap M'_\alpha\} .
\end{equation}

\begin{theorem}[Inverse Satake theorem]\label{EIST} 

A basis of the   image of  $S^G$ is given by the elements \begin{equation}\label{image}
\tau_z    \prod _{\alpha \in \Delta'(V') \setminus  \Delta'(V) }(1 -\tau_{a_\alpha} )
\end{equation} for  $z$ running through a system of representatives of $Z_G^+(V,V')/Z^0$ in $Z_G^+(V,V')$.
The   inverse of $S^G$ sends \eqref{image} to
\begin{equation*}
\varphi_z^{V',V} :=  \sum _{x\in  Z_z^+(V,V')} T_x^{V',V}, \quad \text{where} \quad  Z_z^+(V,V'):=Z^+\cap z\prod_{\alpha \in \Delta'(V)\cap \Delta'(V')}a_\alpha^\mathbb N.
  \end{equation*}
\end{theorem}
The function  $\varphi_z^{V',V}$  is well defined for $z\in Z_G^+(V,V')$ because of the following lemma.
 
\begin{lemma} \label{lm:Zz-contained-in-ZG}
   For $z\in Z_G^+(V,V')$,  the  set $Z_z^+(V,V')$  is finite and  contained in $Z_G^+(V,V')$.
\end{lemma}
\begin{proof}  For $z\in Z$, the set $Z^+\cap z\prod_{\alpha \in \Delta}a_\alpha^\mathbb N$
  is finite. Indeed, 
 $z \prod_{\alpha \in \Delta} a_\alpha^{n(\alpha)}, \ n(\alpha)\in \mathbb N=\{0,1, \ldots \}$ lies in $Z^+$ if and only if 
 $$ \langle \beta_a,\nu(z) \rangle + \sum_{\alpha \in \Delta} n(\alpha)\langle \beta_a, \alpha_a^\vee\rangle \leq 0 \ \text{for all} \  \beta\in \Delta.$$  
 These inequalities admit only finitely solutions $n(\alpha)\in \mathbb N$ for $\alpha \in \Delta$,  because the   matrix $(d_\beta \langle \beta_a, \alpha_a^\vee\rangle)_{\alpha, \beta\in \Delta}$ is positive definite for some $d_\beta > 0$.
 
 For $z\in Z_G^+(V,V')$, an element $x=z \prod_{\alpha \in \Delta'(V)\cap \Delta'(V')} a_\alpha^{n(\alpha)}$ of $ Z_z^+(V,V')$ lies in $Z_{\psi_V, \psi_{V'}} $ 
 as $a_\alpha \in Z_{\psi_V}$ for $\alpha \in \Delta'_{\psi_V}$ (see \S \ref{delta'psi}). For
 $$\alpha \in \Delta'(V)\cap \Delta'(V')= \Delta(V) \cap \Delta(V') \cap \Delta'_{\psi_V}$$
 and  $\beta \in \Delta(V)\triangle\,  \Delta(V')$ we have $\langle \beta_a, \alpha_a^\vee\rangle \leq 0$. By \eqref{triangle},  $z\in Z_G^+(V,V')$ satisfies $\langle \beta_a, \nu(z) \rangle <0$, so the same is true for $x$. Hence $x\in  Z_G^+(V,V')$.
\end{proof}
  
\begin{remark}
  When $V=V'$, and $\mathbf G$ is split with simply-connected derived subgroup, the inverse Satake theorem was obtained by
  \cite[Prop.\ 5.1]{MR2845621} using the Lusztig-Kato formula.
  The proof of the inverse Satake theorem for arbitrary $G$ and an arbitrary pair $(V,V')$ uses the pro-$p$ Iwahori Hecke
  algebra.  It is inspired by the work of Ollivier~\cite{MR3366919}.
\end{remark}

\begin{remark}
  When $V=V'$ the image of $S^G$ was known, see \eqref{imS}.  The description of the image of $S^G$ for a pair $(V,V') $
  with $V\not\simeq V'$ was an open question in \cite[\S 7.9]{MR3331726}.  Theorem \ref{EIST} shows that the image of $S^G$
  for a pair $(V,V')$ with $V\not\simeq V'$ is not always contained in the subspace of functions in $\mathcal H_Z(V_{U^0},
  V'_{U^0})$ supported in $Z^+$. This was noticed for many split groups in \cite[Prop.\ 6.13]{MR2845621}.
\end{remark}

\begin{remark}
  We establish a similar theorem for $S^G_M$ in the next section (Corollary \ref{cor:satake for Levi}), at least when $\Delta'(V')\subset \Delta'(V)\cup\Delta_M$.%
\end{remark}

We mentioned earlier that Theorem \ref{weight2} (hence the change of weight theorem) follows from  the inverse  Satake theorem; it is now the time to justify this assertion.

\begin{proposition}\label{ISimpliesCW} The   inverse  Satake theorem \(Theorem \ref{EIST}\) implies Theorem \ref{weight2} \(and hence the change of weight theorem\).
\end{proposition}
Our first proof only uses the ``image of $S^G$'' part of Theorem~\ref{EIST} (for $V \not\cong V')$, whereas our second proof uses the explicit formula
in Theorem~\ref{EIST} (but only for $V = V'$).

\begin{proof}[First proof] As in  Theorem \ref{weight2}, we suppose that  the parameters of the irreducible representations $V$ and $V'$  of $K$ satisfy $\psi_V=\psi_{V'}$ and $\Delta(V)=\Delta(V')\sqcup\{\alpha\}$. In the proof, we will use only that we know the image of the Satake homomorphisms   for $(V,V')$ and for $(V',V)$. 

As in Theorem \ref{weight2},  let  $z \in Z^+_{\psi_V}$ satisfying $\langle \alpha, v (z)\rangle >0$. This is equivalent to   $z\in Z^+_G(V,V')= Z^+_G(V',V)$ (Example \ref{ex}). By the definition of $c_\alpha$  \eqref{dicho}  and of $\Delta'(V)$ \eqref{delta'V}, 
\[ \Delta'(V) \setminus  \Delta'(V') =\begin{cases} \{\alpha\}\ & \ \text{if}\ c_\alpha=1, \\
\varnothing\ & \ \text{if}\ c_\alpha=0.
\end{cases}
\]

The  inverse Satake  theorem (Theorem \ref{EIST}) gives  two functions $\varphi_z^{V',V} \in \mathcal H _G(V,V')$ and $\varphi_z^{V,V'} \in \mathcal H _G (V',V)$ satisfying
\[ S^G(\varphi_z^{V',V} )=  \tau_z^{V'_{U^0},V_{U^0}} \quad \text{and} \quad  S^G(\varphi_z^{V,V'} )=  \tau_z^{V_{U^0},V'_{U^0}} - c_\alpha \, \tau_{z a_\alpha}^{V_{U^0},V'_{U^0}} .
\]
By Remark~\ref{not0}, the two convolution products are
\begin{align*}S^G(\varphi_z^{V',V} * \varphi_z^{V,V'})&=S^G(\varphi_z^{V',V}) S^G(\varphi_z^{V,V'})= \tau^{V'_{U^0},V_{U^0}} _z (\tau^{V_{U^0},V'_{U^0}} _z - c_\alpha \, \tau^{V_{U^0},V'_{U^0}} _{z a_\alpha}) \\
&=\tau_{z^2}^{V'_{U^0},V'_{U^0}}  - c_\alpha \, \tau_{z ^2a_\alpha}^{V'_{U^0},V'_{U^0}} ,\\
S^G(\varphi_z^{V,V'} * \varphi_z^{V',V})&=S^G(\varphi_z^{V,V'}) S^G(\varphi_z^{V',V})=(\tau_z^{V_{U^0},V'_{U^0}}  - c_\alpha \, \tau^{V_{U^0},V'_{U^0}} _{z a_\alpha}) \tau_z^{V'_{U^0},V_{U^0}}\\&=\tau_{z^2}^{V_{U^0},V_{U^0}}  - c_\alpha \, \tau^{V_{U^0},V_{U^0}} _{z a_\alpha z}=\tau_{z^2}^{V_{U^0},V_{U^0}}  - c_\alpha \, \tau^{V_{U^0},V_{U^0}} _{ z^2a_\alpha }.
\end{align*}
In the second product we used that $\tau_\alpha \in \mathcal Z_Z(V_{U^0})$.
\end{proof}

\begin{proof}[Second proof]
In this proof, we prove Theorem \ref{weight2} for $z\in Z^+_{\psi_V}$ such that $\langle\alpha,v(z)\rangle > 0$ and $\langle\beta,v(z)\rangle = 0$ for any $\beta\in Z^+$.
As we mentioned after Theorem~\ref{weight2}, this implies Theorem~\ref{weight}.

In this proof, we use Theorem~\ref{EIST} only for $V = V'$.
We also use Lemma~\ref{first} and \ref{second} from the next section.
The argument is almost the same as the proof in \cite{MR2845621,MR3143708}.

Set $\varphi = T_z^{V',V}\in \mathcal{H}_G(V,V')$ and $\varphi' = T_z^{V,V'}\in \mathcal{H}_G(V',V)$.
By the assumption on $z$, we have $\Delta_z = \Delta\setminus\{\alpha\}$.
On the other hand, we have $\alpha\notin \Delta(V')$.
Hence $\Delta(V')\subset \Delta_z$.
By Lemma~\ref{second}, we have $\varphi'*\varphi = T_{z^2}^{V,V}$.

We calculate $S^G(T_{z^2}^{V,V})$ using Theorem~\ref{EIST}.
From Lemma~\ref{lem:Z_z^2(V,V') singular case} below and Theorem~\ref{EIST}, we get the following:
\begin{itemize}
\item If $\alpha\in\Delta'(V)$, then 
\begin{align*}
\tau_{z^2}^{V_{U^0},V_{U^0}} & = \sum_{z'\in Z_{z^2}^+(V,V)}S^G(T_{z'}^{V,V})\\
& = S^G(T_{z^2}^{V,V}) + \sum_{z'\in Z_{z^2a_\alpha}^+(V,V)}S^G(T_{z'}^{V,V})\\
& = S^G(T_{z^2}^{V,V}) + \tau_{z^2a_\alpha}^{V_{U^0},V_{U^0}}.
\end{align*}
Hence $S^G(T_{z^2}^{V,V}) = \tau_{z^2}^{V_{U^0},V_{U^0}} - \tau_{z^2a_\alpha}^{V_{U^0},V_{U^0}}$.
\item If $\alpha\not\in\Delta'(V)$, then $\tau_{z^2}^{V_{U^0},V_{U^0}} = S^G(T_{z^2}^{V,V})$.
\end{itemize}
Therefore we get $S^G(\varphi'*\varphi) = \tau_{z^2}^{V_{U^0},V_{U^0}} - c_\alpha \tau_{z^2a_\alpha}^{V_{U^0},V_{U^0}}$.

Since $\Delta(V')\subset \Delta_z$, Lemma~\ref{first} implies $S^G(\varphi) = \tau_{z}^{V'_{U^0},V_{U^0}}$.
Hence $S^G(\varphi')\tau_z^{V'_{U^0},V_{U^0}} = \tau_{z^2}^{V_{U^0},V_{U^0}} - c_\alpha \tau_{z^2a_\alpha}^{V_{U^0},V_{U^0}}$.
Canceling $\tau_z^{V'_{U^0},V_{U^0}}$ and keeping in mind that $\tau_\alpha$ is central, we get $S^G(\varphi') = \tau_{z}^{V_{U^0},V'_{U^0}} - c_\alpha \tau_{za_\alpha}^{V_{U^0},V'_{U^0}}$.
Hence we have
\begin{align*}
S^G(\varphi*\varphi') & = \tau_{z}^{V'_{U^0},V_{U^0}}(\tau_{z}^{V_{U^0},V'_{U^0}} - c_\alpha \tau_{za_\alpha}^{V_{U^0},V'_{U^0}})\\
& = \tau_{z^2}^{V'_{U^0},V'_{U^0}} - c_\alpha \tau_{z^2a_\alpha}^{V'_{U^0},V'_{U^0}}.\qedhere
\end{align*}
\end{proof}

\begin{lemma}\label{lem:Z_z^2(V,V') singular case}
Let $\alpha\in\Delta$, $z\in Z^+$ such that $\langle \alpha,v(z)\rangle > 0$ and $\langle\beta,v(z)\rangle = 0$ for $\beta\in\Delta\setminus\{\alpha\}$.
\begin{enumerate}
\item We have $z^2a_\alpha\in Z^+$.
\item We have $z_1\in Z^+ \cap z^2\prod_{\beta\in\Delta}a_\beta^{\mathbb{N}}$ if and only if $z_1 = z^2$ or $z_1\in Z^+ \cap z^2a_\alpha\prod_{\beta\in\Delta}a_\beta^{\mathbb{N}}$.
In particular, for any irreducible representation $V$ of $K$, we have 
\[
Z_{z^2}^+(V,V) = 
\begin{cases}
\{z^2\}\sqcup Z_{z^2a_\alpha}^+(V,V) & (\alpha\in\Delta'(V)),\\
\{z^2\} & (\alpha\notin\Delta'(V)).
\end{cases}
\]
\end{enumerate}
\end{lemma}
\begin{proof}
Let $\beta\in\Delta$.
If $\beta \ne \alpha$, then $\langle \beta_a,v(a_\alpha)\rangle = \langle\beta_a,-\alpha_a^\vee\rangle\ge 0$.
Hence $\langle\beta_a,v(z^2a_\alpha)\rangle \ge \langle \beta_a,v(z^2)\rangle\ge 0$.
For $\beta = \alpha$, we have $\langle \alpha_a,v(a_\alpha)\rangle = \langle \alpha_a,-\alpha_a^\vee\rangle = -2$.
Hence $\langle \alpha_a,v(z^2a_\alpha)\rangle = 2\langle \alpha_a,v(z)\rangle - 2\ge 0$.

For (ii), the ``if'' part is trivial.
We prove the ``only if'' part.
Let $z_1\in z^2\prod_{\beta\in\Delta}a_\beta^{\mathbb{N}}\cap Z^+$ and take $n(\beta)\in\mathbb{N}$ such that $z_1 = z^2\prod_{\beta\in\Delta}a_\beta^{n(\beta)}$.
Assume that $z_1\notin z^2a_\alpha\prod_{\beta\in\Delta}a_\beta^{\mathbb{N}}\cap Z^+$, namely $n(\alpha) = 0$.
Then for $\gamma\in\Delta\setminus\{\alpha\}$, we have
\[
0\le \langle\gamma_a,v(z_1)\rangle = \langle\gamma_a,v(z^2)\rangle - \sum_{\beta\in\Delta\setminus\{\alpha\}}n(\beta)\langle\gamma_a,\beta_a^\vee\rangle
\]
Hence 
\[
\sum_{\beta\in\Delta\setminus\{\alpha\}}n(\beta)\langle\gamma_a,\beta_a^\vee\rangle \le 2\langle\gamma_a,v(z)\rangle = 0
\]
from the assumption on $z$.
Since the matrix $(d_\gamma\langle \gamma_a,\beta_a^\vee\rangle)_{\beta,\gamma\in\Delta\setminus\{\alpha\}}$ is positive definite for some $d_\gamma > 0$, we get $n(\beta) = 0$ for all $\beta\in\Delta\setminus\{\alpha\}$.
Hence $z_1 = z^2$.
\end{proof}

\subsection{Inverse Satake for Levi subgroups}\label{sec:inverse-satake-levi}
Let $P = MN$ be a parabolic subgroup containing $B$.
By the inverse Satake theorem (Theorem \ref{EIST}) for $S^G = S_Z^G$, we can get the following formula for $S_M^G$.
Let $V,V'$ be irreducible $K$-representations.
We denote the function $T_z^{V_{N^0},V'_{N^0}}\in \mathcal{H}_M(V_{N^0},V'_{N^0})$ for $M$ by $T_z^{V'_{N^0},V_{N^0},M}$.
Also, for $X \subset \Delta$ we write $a_X := \prod_{\gamma \in X} a_\gamma$.
\begin{theorem}\label{thm:satake for Levi}
For $z\in Z_G^+(V,V')$, we have
\[
\sum_{x\in Z_z^+(V,V')}S_M^G(T_x^{V',V}) = \sum_{X\subset \Delta'(V')\setminus(\Delta'(V)\cup \Delta_M)}(-1)^{\#X}\sum_{x\in Z_{z a_X}^{+,M}(V_{N^0},V'_{N^0})}T_x^{V'_{N^0},V_{N^0},M}.
\]
\end{theorem}
\begin{remark}
In the theorem we have $za_X\in Z^+_M(V_{N^0},V'_{N^0})$ since $z\in Z_G^+(V,V')\subset Z_M^+(V_{N^0},V'_{N^0})$ and $\langle\beta_a,\gamma_a^\vee\rangle\le 0$ for any $\beta\in\Delta_M$ and $\gamma\in X\subset \Delta\setminus\Delta_M$. 
\end{remark}
\begin{proof}[Proof of Theorem~\ref{thm:satake for Levi}]
Apply $S^M$ to both sides of the formula given in the theorem.
For the left-hand side, we have
\begin{align*}
S^M\left(\sum_{x\in Z_z^+(V,V')}S_M^G(T_x^{V',V})\right)
& = \sum_{x\in Z_z^+(V,V')}S^G(T_x^{V',V})\\
& = \tau_z\prod_{\alpha\in\Delta'(V')\setminus\Delta'(V)}(1 - \tau_\alpha)
\end{align*}
by Theorem~\ref{EIST}.
For the right-hand side, applying Theorem~\ref{EIST} to $M$ and using an inclusion-exclusion formula, we have
\begin{align*}
& S^M\left(\sum_{X\subset \Delta'(V')\setminus(\Delta'(V)\cup \Delta_M)}(-1)^{\#X}\sum_{x\in Z_{za_X}^{+,M}(V_{N^0},V'_{N^0})}T_x^{V'_{N^0},V_{N^0},M}\right)\\
& =
\sum_{X\subset \Delta'(V')\setminus(\Delta'(V)\cup \Delta_M)}(-1)^{\#X} \sum_{x\in Z_{za_X}^{+,M}(V_{N^0},V'_{N^0})}S^M(T_x^{V'_{N^0},V_{N^0},M})\\
& = \sum_{X\subset \Delta'(V')\setminus(\Delta'(V)\cup \Delta_M)}(-1)^{\#X}\tau_{za_X}\prod_{\alpha\in\Delta'(V'_{N^0})\setminus\Delta'(V_{N^0})}(1 - \tau_\alpha)\\
& = \tau_z\prod_{\alpha\in \Delta'(V')\setminus(\Delta'(V)\cup \Delta_M)}(1 - \tau_\alpha)\prod_{\alpha\in\Delta'(V'_{N^0})\setminus\Delta'(V_{N^0})}(1 - \tau_\alpha)\\
& = \tau_z\prod_{\alpha\in\Delta'(V')\setminus\Delta'(V)}(1 - \tau_\alpha),
\end{align*}
noting also that $\Delta'(V'_{N^0})\setminus\Delta'(V_{N^0}) = (\Delta_M\cap \Delta'(V'))\setminus\Delta'(V)$
(since $\Delta(V_{N^0}) = \Delta_M\cap \Delta(V)$ by \cite[III.10 Lemma]{MR3600042}).
Since $S^M$ is injective, we get the theorem.
\end{proof}
In a special case the formula is simple. In particular this happens when $V \simeq V'$.
\begin{corollary}\label{cor:satake for Levi}
If $\Delta'(V')\subset \Delta'(V)\cup \Delta_M$, then we have for $z \in Z^+_G(V,V')$,
\[
\sum_{x\in Z_z^+(V,V')}S_M^G(T_x^{V',V}) = \sum_{x\in Z_z^{+,M}(V_{N^0},V'_{N^0})}T_x^{V'_{N^0},V_{N^0},M},
\]
and the image of $S_M^G$ is spanned by $\{T_z^{V'_{N^0},V_{N^0},M}\mid z\in Z^+_G(V,V')\}$.
\end{corollary}
\begin{proof}
The first part is immediate. For the last part fix $z \in Z^+_G(V,V')$. We note that $Z_z^{+,M}(V_{N^0},V'_{N^0})\subset Z_z^+(V,V')\subset Z_G^+(V,V')$.
Let $\preceq$ denote the partial order on the finite set $Z_z^{+,M}(V_{N^0},V'_{N^0})$ defined by
$x \preceq y$ if $x \in Z_y^{+,M}(V_{N^0},V'_{N^0})$.
Then the first part applied to $y \in Z_z^{+,M}(V_{N^0},V'_{N^0})$ shows that $\sum_{x \preceq y} T_x^{V'_{N^0},V_{N^0},M}$ is in the image of $S_M^G$.
A triangular argument now shows that $T_y^{V'_{N^0},V_{N^0},M}$ is in the image of $S_M^G$ for any $y \in Z_z^{+,M}(V_{N^0},V'_{N^0})$, in particular
this is true when $y = z$.
\end{proof}

 \section{Reduction    to \texorpdfstring{$\Delta(V')\subset \Delta(V)$}{Delta(V') subset Delta(V)}}\label{sec:reduction-delta}
 
Let $V,V'$ be two irreducible representations of $K$.  We reduce the proof of the  inverse Satake theorem for $(V,V')$  to the particular case where their parameters satisfy $\Delta(V')\subset \Delta(V)$. First, we establish  some lemmas that are of independent interest.

\subsection{First lemma}\label{sec:first-lemma} Let $P=MN$ be a parabolic subgroup of $G$ containing $B$ corresponding to $\Delta_P\subset \Delta$. 
Our first lemma  is the computation in a particular case of the generalized Satake transform $S_M^G:\mathcal H_G(V,V')\to \mathcal H_M(V_{N^0},V'_{N^0})$ (Definition \ref{SMGVV'}); it
 is a generalization of \cite[Cor.\ 2.18]{MR2845621}. 
 
 We fix linear isomorphisms $\iota^{\op}, \iota$ as in \eqref{iopi} for $(V,V')$;  for $z\in Z^+_G(V,V')$ we recall the elements  $T_z^{V',V}\in \mathcal H_G(V,V')$, $T_z^{V'_{N^0},V_{N^0}}\in \mathcal H_M(V_{N^0},V'_{N^0})$  defined  in \S \ref{intVV'}, and the subset $\Delta_z$ of $\Delta $ defined by \eqref{deltaz}.

\begin{lemma}\label{first} Let   $z\in Z_G^+(V,V')$.  We have $S_M^G(T_z^{V',V})=T_z^{V'_{N^0},V_{N^0}}$ if  $\Delta(V')$ is contained in $\Delta_P$ or in $\Delta_z$. 
\end{lemma}
We will use the lemma only when $P=B, M=Z$.

\begin{proof} Let   $z\in Z_G^+(V,V')$. Suppose $m \in M$. Definition \ref{SMGVV'} shows that
$(S_M^G T_z^{V',V})(m) = \sum_{x\in N^0\backslash N}(p'\circ T_z^{V',V})(xm)$, 
where $p':V'\onto V'_{N^0}$  is the quotient map.
The description of $T_z^{V',V} $ given in  \S \ref{intVV'} shows that  the support of $T_z^{V',V} $ is $KzK$,  and  the image of  $T_z^{V',V}(k_1 z k_2)=k_1T_z^{V',V} (z) k_2$ is $k_1V'^{N_{z,\op}^0}$ for $k_1,k_2\in K$ \cite[\S 7.4]{MR3331726}.  One knows that \cite[Cor.\ 3.20]{MR3001801}
\begin{equation}\label{HV}
p'(k_1V'^{N_{z,\op}^0}) \neq 0  \Leftrightarrow k_1 \in P^0 M_{V'}^0 P_{z, \op}^0,
\end{equation}
where $P_{V'}=M_{V'}N_{V'}$  is the parabolic subgroup of $G$ corresponding to $\Delta(V')$.
 
Observe that $\Delta(V') \subset \Delta_P$ implies $M_{V'}^0 \subset M^0$ and that $\Delta(V') \subset \Delta_z$ implies  $ M_{V'}^0 \subset M_z^0$,
so in either case we know that $P^0 M_{V'}^0 P_{z,\op}^0= P^0  P_{z,\op}^0$.
  If  $k_1 \in P^0 P_{z, \op}^0$ then $k_1zk_2$ lies in  $ P^0 P_{z, \op}^0z K= P^0 zK$  as 
   $z\in Z^+$  and $z^{-1} P_{z,\op}^0 z \subset P_{z,\op}^0$.  

Therefore, if $(p'\circ T_z^{V',V})(xm) \ne 0$ for $x \in N$ we deduce that $xm \in P^0 zK \cap P = P^0 z P^0 = N^0 (M^0 z M^0)$.
It follows that $m \in M^0 z M^0$ and $n \in N^0$. In particular, the support of $S_M^G (T_z^{V',V})$ is contained in $M^0 z M^0$ and
$(S_M^G T_z^{V',V})(z) = p'\circ T_z^{V',V}(z)$, which induces the map $\iota : V_{U^0} \to V'_{U^0}$. The lemma follows.
  \end{proof}
  \subsection{Second lemma}\label{sec:second-lemma}
  Our second lemma  is the computation  of the composite of two particular intertwiners; it is  done in \cite[Prop.\ 6.7]{MR2845621}, \cite[Lemma 4.3]{MR3143708} when $\mathbf G$ is split. Let $V''$ be a third irreducible representation of $K$; we fix linear isomorphisms as in \eqref{iopi} for $(V,V')$ and $(V',V'')$ and by composition  for $(V,V'')$. For $z\in Z^+_G(V,V')$ and $z'\in Z^+_G(V',V'')$,   the product  $z'z$ lies in $Z^+_G(V,V'')$ (Remark \ref{zz'})  and we have the elements  $T_z^{V',V}\in \mathcal H_G(V,V')$, $T_{z'}^{V'',V'}\in \mathcal H_G(V',V'')$ and  $T_{z'z}^{V'',V}\in \mathcal H_G(V,V'')$ (\S \ref{intVV'}). 
  
\begin{lemma} \label{second}Let   $z\in Z_G^+(V,V')$ and $z'\in Z^+_G(V',V'')$.  We have $T_{z'}^{V'',V'}* T_z^{V',V}=T_{z'z}^{V'',V}$ if  $\Delta(V')$ is contained in $\Delta_z$ or in $\Delta_{z'}$. 
\end{lemma} 
\begin{proof} By the formula for the convolution product in \S \ref{intVV'}, we are lead to analyse the elements $(x,g)\in K\times G$ such that 
$T_{z'}^{V'',V'}(xz') \circ T_z^{V',V}(z'^{-1}x^{-1}g)\neq 0$.
We follow the arguments of the proof of Lemma \ref{first}. The non-vanishing of $T_z^{V',V}(z'^{-1}x^{-1}g) $ implies  $z'^{-1}x^{-1}g=k_1zk_2$ with $k_1,k_2\in K$; the homomorphism $T_{z'}^{V'',V'}(xz')=xT_{z'}^{V'',V'}(z')$ factors through the quotient map $p_{z'}: V'\onto V'_{N_{z'}^0}$ (see  \S \ref{intVV'}). The image of $T_z^{V',V}(z'^{-1}x^{-1}g)$ 
is  $k_1V'^{N_{z,\op}^0}$  and by 
  \eqref{HV}, 
$p_{z'}(k_1V'^{N_{z,\op}^0})\neq 0$ if and only if $k_1\in P_{z'}^0 M_{V'}^0 P_{z, \op}^0$.

We know that $P_{z'}^0 M_{V'}^0 P_{z, \op}^0=P_{z'}^0  P_{z, \op}^0$, since $\Delta(V')\subset \Delta_z$ or $\Delta(V')\subset \Delta_{z'}$. The non-vanishing of $T_{z'}^{V'',V'}(xz') \circ T_z^{V',V}(z'^{-1}x^{-1}g) $ implies $z'^{-1}x^{-1}g=k_1zk_2\in P_{z'}^0zK$. As $z' P_{z'}^0 z'^{-1} \subset P_{z'}^0$ we deduce $KgK=Kz'zK$.  
We suppose  $g=z'z$  and  we  analyze the  elements $x \in K$ such that $T_{z'}^{V'',V'}(xz') \circ T_z^{V',V}(z'^{-1}x^{-1}z'z)\neq 0$.
 We have  $z'^{-1}x^{-1}z'z\in P_{z'}^0zK$ and $x\in K$, or equivalently $x \in  z' zKz^{-1}z'^{-1}z'P_{z'}^0 z'^{-1} \cap K = (z' zKz^{-1}z'^{-1}\cap K)z' P_{z'}^0 z'^{-1} $. 
 The group $z' K z'^{-1}\cap K$ contains $z' P_{z'}^0 z'^{-1} $ and we claim that it contains also $z' zKz^{-1}z'^{-1}\cap K$. The  formula for the convolution product  given in \S \ref{intVV'} and this claim imply the lemma. The claim is proved in Lemma~\ref{lem:claim-for-second} below.
\end{proof} 
We now check the claim used in the proof of  Lemma \ref{second}.
\begin{lemma}\label{lem:claim-for-second}Let $z,z'\in Z^+$. Then $z' zK(z' z)^{-1}\cap K$ is contained in $ z' K z'^{-1}\cap K$.
\end{lemma}
\begin{proof}  
  For $z \in Z^+$ consider the bounded subset $\Omega_z = \{x_0, z x_0\}$ of the apartment of $S$, so $z K z^{-1}\cap K$ is the
  pointwise stabilizer of $\Omega_z$ in the kernel of the Kottwitz homomorphism \cite[Def.\ 3.14]{MR3484112}.  For $\alpha \in \Phi$
  let $r_{\Omega_z}(\alpha) = \max\{ 0, -\langle \alpha, \nu(z)\rangle \}$.
  By Bruhat-Tits theory
  (following \cite[\S 3.6]{MR3484112}, noting that the description of the pointwise stabilizer in 
  equation \cite[(42)]{MR3484112} is valid not just for points $x$ but for bounded subsets of the apartment of $S$)
  we then know that $z K z^{-1}\cap K$ is generated by the groups $U_{\alpha + r_{\Omega_z}(\alpha)} \subset U_\alpha$ for $\alpha \in \Phi$ and
  the cosets $s_\beta Z^0 \subset \cn^0$ for $\beta \in \Phi$ such that $\langle \beta, \nu(z)\rangle = 0$.
  The lemma follows by noting that $r_{\Omega_{zz'}}(\alpha) \ge r_{\Omega_{z'}}(\alpha)$ and that
  $\langle \beta, \nu(zz')\rangle = 0$ implies $\langle \beta, \nu(z')\rangle = 0$ for any roots $\alpha,\beta \in \Phi$.
 \end{proof} 
 
 \subsection{Third lemma}\label{sec:third-lemma}
  
 \begin{lemma}\label{plus} Let  $z\in Z^+$ and $x=z \prod_{\alpha \in \Delta} a_\alpha^{n(\alpha)}$  with $n(\alpha)\in \mathbb N$.
  If $\langle \alpha , v(z) \rangle $ is large enough for those $\alpha \in \Delta$ with $n(\alpha)>0 $, then $x\in Z^+$. 
 
 \end{lemma}
\begin{proof} Recall that $v=-\nu$ and that $Z^+$ is the monoid of $z\in Z$ such that the integers 
$ \langle \beta_a,\nu(z) \rangle $ are $\leq 0$ for all $\beta\in \Delta$. We have $\nu(a_\alpha)=\alpha_a^\vee$  (Definition \ref{delta'psi}) and 
 $\langle \beta_a,\nu(x) \rangle= \langle \beta_a,\nu(z) \rangle + \sum_{\alpha \in \Delta} n(\alpha)\langle \beta_a, \alpha_a^\vee\rangle$  for all $\beta\in \Delta.$  
We have  $\langle \beta_a, \alpha_a^\vee\rangle \leq 0$ if $\alpha\neq \beta$ and $\langle \alpha_a, \alpha_a^\vee\rangle=2$.
The integer $ \langle \beta_a,\nu(z) \rangle $  is $\leq 0$ as $z\in Z^+$.
If $n(\beta)=0$  then  $\langle \beta_a,\nu(x) \rangle \leq 0$. 
  If $n(\beta)>0$ and $  \langle \beta_a,\nu (z) \rangle + 2 n(\beta) \leq 0$ then $\langle \beta_a,\nu(x) \rangle \leq 0$. \end{proof}

Later we will  use it in the following form.

 \begin{lemma}\label{plusJ} Suppose  $z\in Z$, $J\subset \Delta$, and $n(\alpha) \in \mathbb N$ for $\alpha \in J$.  
Then there exists $y\in Z^+\cap M'_J$ such that $yz\prod_{\alpha \in J} a_\alpha^{m(\alpha)}$ lies in $Z^+$ for all $m(\alpha)\in \mathbb N, m(\alpha)\leq n(\alpha)$.
\end{lemma}
\begin{proof}  
We can find $y\in Z^+\cap M'_J$ with  $ \langle \alpha_a,v(y) \rangle \geq   2 n(\alpha) -  \langle \alpha_a,v(z) \rangle$ for all $\alpha\in J$. Then we have $ \langle \alpha_a,v(yz) \rangle \geq   2 m(\alpha)$ for  $m(\alpha)\leq n(\alpha)$. The proof  of Lemma \ref{plus} implies  $yz \prod_{\alpha \in J} a_\alpha^{m(\alpha)}$ lies in $Z^+$ for all $m(\alpha)\in \mathbb N, m(\alpha)\leq n(\alpha)$.
\end{proof}

\subsection{Reduction to \texorpdfstring{$\Delta(V')\subset \Delta(V)$}{Delta(V') subset Delta(V)}}\label{sec:reduction2}
We are ready to prove that (a special case of) the   inverse Satake theorem for a pair  $(V,V')$
with parameters satisfying  $\Delta(V')\subset \Delta(V)$ implies the  inverse Satake transform for a  general pair.
Note that when $\Delta(V')\subset \Delta(V)$,  then $\Delta'(V') \subset \Delta' (V)$. %

\begin{theorem}\label{VcontV'} 
Assume $\Delta(V')\subset \Delta(V)$.  
For $z\in Z_G^+(V,V')$, %
we have $S^G(\varphi_z) = \tau_z$, where
$$
\varphi_z=  \sum _{x\in  Z_z^+(V,V')} T_x \quad \text{and} \quad  Z_z^+(V,V')=Z^+\cap z\prod_{\alpha \in \Delta'(V')}a_\alpha^\mathbb N.
$$
\end{theorem}

 \begin{proposition}\label{prop:3.7} Theorem \ref{VcontV'} implies  the   inverse Satake theorem \(Theorem \ref{EIST}\).
 \end{proposition}
 
 \begin{proof} The proof is  divided into several  parts.  
  
 {\bf A}) Let $(V,V')$ be an arbitrary pair of irreducible representations of $K$. We introduce: 
 
 \begin{enumerate}
\item The  irreducible representation $V''$  of $K$ with parameters $\psi_{V''}=\psi_{V'}$ and $\Delta(V'')=\Delta(V)\cap \Delta(V')$. Such a representation exists \cite[Thm.\ 3.8]{MR3001801},    $Z_G^+(V,V')\subset Z_G^+(V,V'')$ (Remark \ref{zz'}) and $Z_G^+(V',V'')=Z_G^+(V'',V')$ (Example \ref{ex}).
\item A central element $z'$   of $Z$ (hence normalizing any character $\psi$ of $Z^0$) lying in $Z^+$ (hence   in $Z^+_\psi$ for any $\psi$) and such that $\Delta_{z'}\cap (\Delta(V)\cup \Delta(V'))=\Delta(V)$. Hence  $z'\in   Z^+_G(V'',V')$  by \eqref{ZVV'}.
\end{enumerate}

Let $z\in Z^+_G(V,V')$ and let  $\varphi_z^{V',V}=
 \sum _{x\in  Z_z^+(V,V')} T_x^{V',V}$ as in Theorem \ref{EIST}. 
 We reduce  the computation of $S^G(\varphi_z^{V',V})$   to the  single computation of $ S^G(T_{z'}^{V',V''})$ using Theorem \ref{VcontV'} for $(V,V'')$. As $z\in Z^+_G(V,V'')$
and   $\Delta(V'')\subset \Delta(V)$,  Theorem \ref{VcontV'} implies 
\begin{equation}\label{a}S^G(\varphi_z^{V'',V})=\tau_z^{V''_{U^0},V_{U^0}},
\end{equation}
where $\varphi_z^{V'',V}=\sum_xT_x^{V'',V}$  for $x\in Z^+\cap z \prod_{\alpha \in J} a_\alpha^{\mathbb N}$ with
  \[J:=\Delta(V)\cap \Delta(V')\cap \Delta'_{\psi_{V'}}=\Delta(V'')\cap \Delta'_{\psi_{V''}}.\]
Such an $x$ is contained in $Z^+_G(V,V')$ by Lemma~\ref{lm:Zz-contained-in-ZG} and hence in $Z^+_G(V,V'')$. Also,
the sets  $\Delta(V'')$ and $\Delta(V)$ are contained in  $\Delta_{z'}$, and  $z' \in Z_G^+(V'',V') \cap Z_{\psi_V}^+ $. Lemma \ref{second} applied  twice gives
\begin{align*} 
T_{z'}^{V',V''}* T_{x}^{V'',V}=T_{z'x}^{V',V}, \quad T_{x}^{V',V}* T_{z'}^{V,V}=T_{x z'}^{V',V},
\end{align*}
and   Lemma \ref{first} applied to $M=Z$, $V=V'$ and  $z'\in  Z_{\psi_V}^+$  gives
 \begin{align*} S^G(T_{z'}^{V,V})=\tau_{z'}^{V_{U^0},V_{U^0}}.\end{align*} Since $z'$ is central in $Z$, we can permute $z'$ and $x$ on the right-hand side, hence  $T_{z'x}^{V',V}=T_{xz'}^{V',V}$. We deduce
 \begin{align}\label{d} S^G(T_{z'}^{V',V''}) S^G(T_{x}^{V'',V}) =S^G( T_{x}^{V',V}) \tau_{z'}^{V_{U^0},V_{U^0}}.
\end{align}
  Taking the sum of \eqref{d} for $x\in Z^+\cap z \prod_{\alpha \in J} a_\alpha^{\mathbb N}$,  we get
 \begin{equation}\label{e}S^G(T_{z'}^{V',V''}) S^G(\varphi_z^{V'',V}) =S^G(\varphi_z^{V',V}) \tau_{z'}^{V_{U^0},V_{U^0}}.
\end{equation}
We used only Lemmas \ref{first} and \ref{second} to get \eqref{e}.
Using \eqref{a} in \eqref{e} and taking the right convolution  by  $ \tau_{(z')^{-1}}^{V_{U^0},V_{U^0}}$, we obtain
\begin{align}\label{redu}S^G(\varphi_z^{V',V})=S^G(T_{z'}^{V',V''}) \tau_z^{V''_{U^0},V_{U^0}} \tau_{(z')^{-1}}^{V_{U^0},V_{U^0}}= S^G(T_{z'}^{V',V''}) \tau_{ z(z')^{-1}}^{V''_{U^0},V_{U^0}}.
\end{align}
The computation of $S^G(\varphi_z^{V',V})$  is reduced to the computation of $ S^G(T_{z'}^{V',V''})$. 

{\bf B}) We cannot directly apply Theorem \ref{VcontV'}  to compute 
 $ S^G(T_{z'}^{V',V''})$  because $\Delta(V') $ is not contained in $\Delta(V'')$. But we show that the computation of $ S^G(T_{z'}^{V',V''})$ reduces to the computation of $S^G(T_{{z'}^2}^{V',V'})$ using Lemmas \ref{first} and \ref{second}.
  
As $\Delta(V'')\subset \Delta_{z'}$,  Lemma \ref{first} applied to  $M=Z$, $V',V''$ and $z'\in Z^+_G(V', V'')$   gives
\begin{align}\label{eq:1} S^G(T_{z'}^{V'', V'})=\tau_{z'}^{V''_{U^0}, V'_{U^0}},
\end{align}
and Lemma \ref{second} applied to $z'\in Z_G^+(V',V'')$ and $z'\in Z_G^+(V'',V')$  gives 
\begin{align*} T_{z'}^{V',V''}* T_{z'}^{V'',V'}= T_{{z'}^2}^{V',V'}.
\end{align*}
Applying the Satake transform, using~\eqref{eq:1} and taking a right convolution by $\tau_{(z')^{-1}}^{V'_{U^0}, V''_{U^0}}$ we get
\begin{align*} 
 S^G(T_{z'}^{V',V''}) = S^G(T_{{z'}^2}^{V',V'})
\tau_{(z')^{-1}}^{V'_{U^0}, V''_{U^0}}.
 \end{align*}
Plugging this value of $S^G(T_{z'}^{V',V''})$ into \eqref{redu} and using that $z'$ is central in $Z$ we get  \begin{align}\label{redu2}S^G(\varphi_z^{V',V})= S^G(T_{{z'}^2}^{V',V'})
\tau_{(z')^{-1}}^{V'_{U^0}, V''_{U^0}} \tau_{z (z')^{-1}}^{V''_{U^0},V_{U^0}}= S^G(T_{{z'}^2}^{V',V'})\tau_{(z')^{-2}z}^{V'_{U^0}, V_{U^0}}.
\end{align}

{\bf C}) We  now compute $S^G(T_{{z'}^2}^{V',V'})$. Applying Theorem \ref{VcontV'} to $V=V'$ and to $z'^2\in Z^+_G(V',V')$ gives
\begin{align*}
S^G(\varphi_{z'^2}^{V',V'})= \tau_{{z'}^2}^{V'_{U^0}, V'_{U^0}} 
\end{align*}
for $\varphi_{z'^2}^{V',V'}=  \sum _{x\in  Z_{z'^2}^+(V',V')} T_x^{V',V'}$ where $Z_{z'^2}^+(V',V')=Z^+\cap {z'^2}\prod_{\alpha \in \Delta'(V')}a_\alpha^\mathbb N.$

 But we want to compute $S^G(T_{z'^2}^{V',V'})$. We can choose any element $z'$ that satisfies {\bf A}) (ii). 
We choose such a  $z'$ with the property that   ${z'}^2 \prod_{\alpha \in \Delta'(V')\setminus  \Delta'(V)}a_\alpha ^{\epsilon (\alpha)}$ lies in $Z^+$ for all $\epsilon (\alpha)\in \{0,1\}$ (this is possible by Lemma \ref{plus}). For such a $z'$  and $\alpha \in \Delta'(V')\setminus  \Delta'(V)$, we have ${z'}^2 a_\alpha\in Z_{\psi_{V'}}^+  $ (recall   from Definition \ref{delta'psi} that $a_\alpha\in Z_{\psi_{V'}}$ as $\psi_{V'}$ is trivial on $Z^0\cap M'_\alpha$). Theorem \ref{VcontV'} applied  to $V=V'$ and  ${z'}^2 a_\alpha\in Z^+_{\psi_{V'}}$ gives
\begin{align*}
S^G(\varphi_{z'^2 a_\alpha}^{V',V'})= \tau_{{z'}^2 a_\alpha}^{V'_{U^0}, V'_{U^0}}= \tau_{{z'}^2}^{V'_{U^0}, V'_{U^0}} \tau_{  \alpha}^{V'_{U^0}, V'_{U^0}}.
\end{align*}
We see that $\varphi^{V',V'}_{z'^2}-\varphi_{z'^2 a_\alpha}^{V',V'}$ is the sum of $T^{V',V'}_x$ for $x\in Z^+\cap z'^2 \prod_{\beta \in \Delta'(V')-\{\alpha\}}a_{\beta}^{\mathbb N}$ and 
\begin{align*}
S^G(\varphi_{z'^2}^{V',V'}-\varphi_{z'^2 a_\alpha}^{V',V'})= \tau_{{z'}^2}^{V'_{U^0},V'_{U^0}}(1- \tau_{  \alpha}^{V'_{U^0}, V'_{U^0}}).
\end{align*}
By iteration we obtain that \[\tau_{{z'}^2}^{V'_{U^0}, V'_{U^0}}  \prod_{\alpha \in \Delta'(V')\setminus  \Delta'(V)} (1-\tau_{\alpha}^{V'_{U^0}, V'_{U^0}})\] is the sum of $S^G(T^{V',V'}_x)$ for $x\in Z^+\cap  z'^2\prod_{\beta\in \Delta'(V')\cap \Delta'(V) }a_{\beta}^{\mathbb N}$. But $z'^2$ is the only  element $ z'^2\prod_{\beta\in \Delta'(V')\cap \Delta'(V) }a_{\beta}^{n(\beta)}$ with $n(\beta)\in \mathbb N$  such that 

\[\langle \alpha_a, \nu (z'^2)\rangle + \sum_{\beta\in \Delta'(V')\cap \Delta'(V) }n(\beta)\langle \alpha_a,  \beta_a^\vee \rangle \leq 0 \quad  \forall\alpha \in \Delta.\] 

\noindent The reason is that all the  $\beta\in \Delta'(V')\cap \Delta'(V)$ are contained in $\Delta(V)$ hence in $\Delta_{z'}$, and that the matrix $(d_\alpha \langle \alpha_a,  \beta_a^\vee \rangle)_{\alpha, \beta\in \Delta'(V')\cap \Delta'(V)}$ is positive definite for some $d_\alpha > 0$. We deduce:
 \begin{align}\label{z2V'V'}
S^G(T_{z'^2}^{V',V'})=\tau_{{z'}^2}^{V'_{U^0}, V'_{U^0}}  \prod_{\alpha \in \Delta'(V')\setminus  \Delta'(V)} (1-\tau_{\alpha}^{V'_{U^0}, V'_{U^0}}).
\end{align}
{\bf D}) Plugging the value of  $S^G(T_{z'^2}^{V',V'})$ given by \eqref{z2V'V'} into \eqref{redu2} we get 
\begin{align}\label{redu3}S^G(\varphi_z^{V',V})= \tau_{{z'}^2}^{V'_{U^0}, V'_{U^0}}  \prod_{\alpha \in \Delta'(V')\setminus  \Delta'(V)} (1-\tau_{\alpha}^{V'_{U^0}, V'_{U^0}})\tau_{(z')^{-2}z}^{V'_{U^0}, V_{U^0}}.
 \end{align}
 As $z'$ is central in $Z$, the first term on the right-hand side commutes with the product and  using $\tau_{{z'}^2}^{V'_{U^0}, V'_{U^0}} \tau_{(z')^{-2}z}^{V'_{U^0}, V_{U^0}}=
 \tau_{z}^{V'_{U^0}, V_{U^0}}$,  the element $z'^2$ disappears from the formula \eqref{redu3}.  As $\tau_{\alpha}^{V'_{U^0}, V'_{U^0}} \tau_{z}^{V'_{U^0}, V_{U^0}}=\tau_{z}^{V'_{U^0}, V_{U^0}}\tau_{\alpha}^{V_{U^0}, V_{U^0}} $ for $\alpha \in \Delta'_{\psi_V} = \Delta'_{\psi_{V'}}$ (Remark \ref{not0}), we obtain the formula of Theorem \ref{EIST}:
 \begin{align}\label{eq:2} S^G(\varphi_z^{V',V})=\tau_{z}^{V'_{U^0}, V_{U^0}}\prod_{\alpha \in \Delta'(V')\setminus  \Delta'(V)} (1-\tau_{\alpha}^{V_{U^0}, V_{U^0}}).
 \end{align}
 
 {\bf E}) Choose a system of representatives $X$ for $Z_G^+(V,V')/Z^0$ in $Z_G^+(V,V')$ such that $x \in X$, $x a_\alpha \in Z_G^+(V,V')$ implies that
$x a_\alpha \in X$. In particular, the $T_x^{V',V}$ for $x \in X$ form a basis of $\mathcal H_G(V,V')$.
Recalling that $\varphi_z^{V',V}= \sum _{x\in  Z_z^+(V,V')} T_x^{V',V}$ and that $Z_z^+(V,V') = Z^+\cap z\prod_{\alpha \in \Delta'(V)\cap \Delta'(V')}a_\alpha^\mathbb N$,
Lemma~\ref{lm:Zz-contained-in-ZG} implies that the expansion of the $\varphi_z^{V',V}$ in terms of the basis $T_x^{V',V}$ ($z, x \in X$) is triangular.
Therefore 
  the $\varphi_z^{V',V}\in \mathcal H_G(V,V')$  for  $z \in X$ form a basis  of $\mathcal H_G(V,V')$.  
  As $S^G$ is injective, this implies that the elements  on the right-hand side of 
 the formula \eqref{eq:2} form a basis of the image of $S^G$.
 \end{proof}

\section{Pro-\texorpdfstring{$p$}{p} Iwahori Hecke ring}\label{S:4}
The  inverse  Satake  theorem for a pair $(V,V')$ of irreducible representations of $K$ with parameters satisfying  $\Delta(V')\subset \Delta(V)$  (Theorem \ref{VcontV'})  relies on the theory of the pro-$p$ Iwahori Hecke ring of $G$  \cite{MR3484112} and on the results presented in this  chapter.

\subsection{Bruhat order on the Iwahori Weyl group} \label{BO}
The Iwahori subgroups of $G$ are the conjugates of the Iwahori subgroup  $K(1)B_{\op}^0$; their pro-$p$ Sylow subgroups are   the pro-$p$ Iwahori subgroups of $G$, and are the conjugates of the pro-$p$ Iwahori subgroup 
$$I=K(1)U_{\op}^0.$$ 
We have  $K(1)B_{\op}^0=IZ^0$ and $I= U_{\op}^0Z(1) (U\cap I)$ (in any order) with the notation of \S  \ref{Notation}. The map $n\mapsto IZ^0 n IZ^0$ induces a bijection from the Iwahori Weyl group $W=\cn/IZ^0$  onto the set  $IZ^0 \backslash G /IZ^0$ of double cosets of $G$ modulo the Iwahori group $IZ^0$, and  the map  $n\mapsto I n I $ induces a bijection from the pro-$p$ Iwahori Weyl group $W(1)=\cn/Z(1)$ onto  the set $I\backslash G/I$ of  double cosets of $G$ modulo the pro-$p$ Iwahori group $I$; the group $W(1)$  is an extension of $W$ by  $Z_k=Z^0/Z(1)$. The action of $\cn$ on  the apartment $x_0+V_{\ad}$ factors through $W$.   We identify $x_0+V_{\ad}$ with $V_{\ad}$ by sending $x_0$ to $0\in V_{\ad}$. 
The  Iwahori Weyl group $W$  contains the group $W^{\aff}=(\cn\cap G')/(Z^0\cap G')$ identified with the affine Weyl group of $\Phi_a$ via the action of $\cn$ on $V_{\ad}$.  The quotient map $W\onto W_0=\cn/Z$  splits as it induces an isomorphism from $\cn^0/Z^0$ onto $W_0$, and the  kernel $\Lambda=Z/Z^0 $ of  $W\to W_0$ is commutative and finitely generated.  
 The  homomorphism $\nu:Z \to V_{\ad}$ factors through $\Lambda $  and  induces an isomorphism  from  $  \Lambda \cap W^{\aff}$ onto  the coroot lattice $\nu (Z\cap G')=\oplus_{\alpha \in \Delta} \mathbb Z \alpha^\vee_a$ of $\Phi_a$ (defined in \eqref{Phia}).  The   lattice $\nu(Z)$ contains the  coroot lattice  and  is contained in
   the lattice of coweights 
   $$P(\Phi_a^\vee)= \{ x\in V_{\ad} \ | \ \langle \alpha_a, x\rangle \in \mathbb Z \ \text{for all} \ \alpha \in \Delta\}.$$

 The Iwahori  group  $K(1)P_{\op}^0=IZ^0$ is the fixator of the fundamental antidominant alcove   $\mathfrak C^-$ of vertex $0$ contained in the antidominant closed Weyl chamber $\mathfrak D^-$ (defined in \eqref{D-}). For $\alpha \in \Phi, n\in \mathbb Z,$
the  reflection  $s_{\alpha_a -n}:x\mapsto x - (\langle \alpha_a, x\rangle - n) \alpha_a^\vee$ of $V_{\ad}$ with respect to a wall $\langle \alpha_a, x\rangle=n$ of $V_{\ad}$  is  conjugate in $W^{\aff}$ to a reflection with respect to a wall of  $\mathfrak C^-$; let $\mathfrak S$ (resp.\ $S^{\aff}$) denote the set of  reflections  with respect to the walls $\Ker (\alpha_a -n)$ of $V_{\ad}$ (resp.\ of  $\mathfrak C^-$).
Let $\Omega$ be the $W$-normalizer of $S^{\aff}$.
The Iwahori Weyl group admits two semidirect product decompositions 
\begin{equation*} W= \Lambda \rtimes W_0 = W^{\aff}\rtimes \Omega.
\end{equation*}
  The image $ {}_1W^{\aff}$   of $\cn\cap G'$ in $W(1)$ is a normal subgroup and is  an extension of $W^{\aff}$ by a subgroup $Z_k^{\aff}$ of $Z_k$.  The inverse image $W^{\aff}(1)$ of $W^{\aff}$ in $W(1)$ is   $ {}_1W^{\aff}Z_k$. Denoting by $\mathfrak S(1)$ (resp.\ $S^{\aff}(1)$, resp.\ $\Omega(1)$) the inverse image of $\mathfrak S$ (resp.\ $S^{\aff}$, resp.\ $\Omega$) in $W(1)$, we have \begin{equation}\label{Waf} W(1)= {}_1W^{\aff} \Omega(1), \quad {}_1W^{\aff} \cap\Omega(1)= Z_k^{\aff}, \end{equation}
  $\mathfrak S(1)={}_1 \mathfrak S Z_k, \ S^{\aff}(1)={}_1 S^{\aff} Z_k$ where $ {}_1W^{\aff} \cap \mathfrak S(1)=  {}_1 \mathfrak S, \  {}_1W^{\aff} \cap  S^{\aff}(1)=  {}_1 S^{\aff}$.
   
\begin{definition}\label{lambdaalpha} Let $\lambda_\alpha\in \Lambda $ be the image of $a_\alpha\in Z\cap M'_\alpha$  (Definition \ref{delta'psi}).
\end{definition}
Note that $\lambda_\alpha$ is independent of any choices.
By  Definition \ref{delta'psi},  $\nu(\lambda_\alpha)=\nu(a_\alpha)=\alpha_a^\vee,$ and 
 \begin{equation}\label{Laf}   \Lambda \cap W^{\aff}=  \prod _{\alpha \in \Delta}\lambda _\alpha^{\mathbb Z} .
\end{equation}

The length $\ell$ of the Coxeter system $(W^{\aff},S^{\aff})$ extends to a length on $W$ (by $\ell(wu) = \ell(w)$ for $w \in W^{\aff}$, $u \in \Omega$)
and further inflates to a length on $W(1)$, still denoted by $\ell$. 
For $\tilde w, \tilde u\in W(1)$ lifting $w\in W^{\aff}, u\in \Omega$, we have $\ell(\tilde{ w} \tilde u)=\ell(wu) =\ell(w)$. There is a useful formula for the length of 
$\lambda w$ where $\lambda\in \Lambda, w\in W_0$ \cite[Cor.\ 5.10]{MR3484112} (the  signs are different because $S^{\aff}$ is the set of reflections with respect to the walls of the dominant alcove $\mathfrak C^+= - \mathfrak C^-$ in loc.\ cit.):
\begin{align}\label{length} \ell(\lambda w)&=\sum_{\alpha_a\in \Phi_a^+ \cap w( \Phi_a^+)}|\langle \alpha_a, \nu(\lambda)\rangle|+\sum_{\alpha_a\in\Phi_a^+ \cap w( \Phi_a^-)}|\langle \alpha_a, \nu(\lambda)\rangle +1|\\
\label{length2} &= \ell (\lambda) - \ell(w) +2 |\{\alpha \in \Phi_a^+ \cap w( \Phi_a^-), \  \langle \alpha_a, \nu(\lambda)\rangle \geq 0\}|.\end{align}
In particular, for $\lambda \in \Lambda^+ = Z^+/Z^0$ we have $\ell(\lambda )=-\langle  2 \rho, \nu (\lambda)\rangle$, where    $ 2\rho$ is the sum of positive roots of $\Phi_a$, and  $\ell(w\lambda )=\ell(\lambda )+\ell(w)$.

\begin{definition} \label{Bru}The Bruhat partial order $\leq $ of $(W^{\aff},S^{\aff})$ inflates  to a partial order $\leq $ on $W$ and to a preorder $\leq $ %
on $W(1)$.  

  \begin{itemize}
\item $w_1u_1 \leq  w_2u_2\Leftrightarrow w_1 \leq w_2, u_1=u_2$  for  $w_1,w_2\in W^{\aff},u_1,u_2\in \Omega$ {\rm \cite[Appendix]{MR2192484}}.

\item   $\tilde w_1\leq  \tilde w_2 \Leftrightarrow w_1 \leq w_2$ for $\tilde w_1, \tilde w_2\in W(1)$ with images $w_1,w_2\in W$ {\rm \cite[Appendix]{MR2192484}}. 
 \end{itemize} 
\end{definition}

  There is the  partial order $\preceq$ on $V_{\ad}$ determined by $- \Delta_a^\vee$ (the basis of $\Phi_a$ corresponding to the anti-dominant  closed Weyl chamber $\mathfrak D^-$ \eqref{D-}): $x_1\preceq x_2$ if and only if 
$x_1-x_2 \in \sum_{\alpha \in \Delta} \mathbb N \alpha^\vee_a$.    The next proposition compares  the ``Bruhat order'' $\leq $ on $\Lambda^+=Z^+/Z^0$  and the  partial order $\preceq$ on $\nu(\Lambda^+)$.
 
\begin{proposition}\label{L+order} Let $\lambda_1,\lambda_2 \in \Lambda^+ $. Then 
$$\lambda_1\leq \lambda_2 \ \Leftrightarrow \ \lambda_1\in  \lambda_2  \prod _{\alpha \in \Delta}\lambda _\alpha^{\mathbb N}\ \Leftrightarrow\ \big(\nu(\lambda_1) \preceq \nu(\lambda_2),\  \lambda_1\in  \lambda_2  W^{\aff}\big).$$ \end{proposition}
 The latter equivalence is clear because $\nu (\lambda_\alpha)=\alpha_a^\vee$ and by~\eqref{Laf}.
 The first one follows from  the next two  lemmas  \cite{MR2141705} (we thank Xuhua He  for drawing our attention to them).

  \begin{lemma}\label{firststep}Let $\alpha \in \Delta$  and $ \lambda \in \Lambda^+$  such that $  \lambda  \lambda_\alpha\in \Lambda^+$. Then  \[ \lambda  \lambda_\alpha  <    \lambda s_{\alpha} <  \lambda .\]
\end{lemma}
\begin{proof}   \cite[Remark 3.9]{MR2141705}. Recall $\nu (\lambda_\alpha) =\alpha_a^\vee$ (Definition \ref{lambdaalpha}).  We have $\langle 2 \rho, \alpha_a^\vee\rangle =2$ where $2\rho$ is the sum of positive roots $\alpha_a\in\Phi_a^+$  \cite[VI.1.11, Prop.\ 29 (iii)]{MR1890629}. We deduce
 $$\ell(\lambda )= \langle 2 \rho, v(\lambda)\rangle = \langle 2 \rho, v(\lambda \lambda_\alpha)\rangle - \langle  2\rho, v(\lambda_\alpha)\rangle= \langle 2 \rho, v(\lambda \lambda_\alpha)\rangle 
 +\langle 2 \rho,   \alpha_a^\vee\rangle=  \ell(\lambda \lambda_\alpha)+2 .$$
Also, $\ell(\lambda s_\alpha) = \ell(\lambda)-1$, as $\langle \alpha_a, \nu(\lambda) \rangle \le -2$, since $\lambda\lambda_\alpha \in \Lambda^+$.
We have that $  s_\alpha \lambda_\alpha =   s_{\alpha_a +1}$ is an affine reflection in  $\mathfrak S$.  Also, $\lambda \lambda_{\alpha}=(\lambda s_{\alpha}) (s_\alpha \lambda_\alpha)$,  
 $\ell (\lambda s_\alpha)=\ell(\lambda)-1$  and 
 $\ell(\lambda \lambda_\alpha)=\ell (\lambda s_\alpha)-1$.
 Recalling the Definition \ref{Bru} of the Bruhat order, we get the lemma.  \end{proof}

Half of the  first equivalence of Proposition \ref{L+order} follows from this lemma  (proof of \cite[Prop.\ 3.5]{MR2141705}). Indeed, let $\lambda_1,\lambda_2\in \Lambda^+$ such that $\lambda_1\in  \lambda_2  \prod _{\alpha \in \Delta}\lambda _\alpha^{n(\alpha)}$ with $n(\alpha)\in \mathbb N$.  
 By Lemma \ref{plusJ}, there exists $\lambda \in \Lambda^+$ such that $ \lambda \lambda_2  \prod _{\alpha \in \Delta}\lambda _\alpha^{m(\alpha)}$ lies in  $\Lambda^+$ for all integers $m(\alpha)\in \mathbb N, m(\alpha)\leq n(\alpha)$. There is a chain $(x_i)_{1\leq i \leq n}$ from $x_1 = \lambda \lambda_2$ to $x_n = \lambda \lambda_1 $ in $\Lambda^+$ such that $x_{i+1}=x_i \lambda_\alpha$ for some $\alpha \in \Delta$.  Lemma \ref{firststep} implies $x_{i+1}< x_i$. Hence $\lambda \lambda_1 \leq\lambda \lambda_2$.
  We have 
 $\ell(\lambda \lambda_i)=\ell(\lambda)+\ell(\lambda_i)$ by the length formula \eqref{length} and  $\lambda \lambda_1\leq \lambda \lambda_2$ is equivalent to $\lambda_1\leq  \lambda_2$. Therefore  if $\lambda_1,\lambda_2\in \Lambda^+$ are such that $\lambda_1\in  \lambda_2  \prod _{\alpha \in \Delta}\lambda _\alpha^{\mathbb N}$  we have  $\lambda_1\leq  \lambda_2$. 
\begin{lemma}\label{poly} Let  $\mathcal P$ be a $W_0$-invariant  convex subset of $V_{\ad}$ and let $x_1,x_2 \in W$ such that $x_1\leq x_2$. If $x_2(0)\in \mathcal P$ then  $x_1(0)\in \mathcal P$.
\end{lemma}
\begin{proof} \cite[Lemma 3.3]{MR2141705}. We can reduce to $x_1= s_{\alpha_a+m} x_2$ for a simple affine reflection  $s_{\alpha_a+m} $  with $s_{\alpha_a+m} x_2<x_2$  and $\alpha_a\in \Phi_a, m\in \mathbb Z$. In particular $\alpha_a+m$ is positive on the alcove $\mathfrak C^-$. Then  $\alpha_a +m$ is negative  on the alcove $x_2(\mathfrak C^-)$. Hence $m\geq 0$ and $\langle \alpha_a,  x_2(0)\rangle +m\leq 0$. This implies that $x_1(0)=x_2(0)- (\langle \alpha_a,  x_2(0)\rangle +m)\alpha_a^\vee$ lies between 
$x_2(0)$ and $ s_\alpha (x_2(0))=x_2(0)- \langle \alpha_a,  x_2(0)\rangle \alpha_a^\vee$. 
The lemma is now clear.  The lemma is true (with the same argument) for any element in the closure of  $\mathfrak C^-$ instead of the origin $0$.
\end{proof}

The second half of the first equivalence in Proposition \ref{L+order} follows from this lemma.  
For  $w\in W_0$ and $\lambda\in \Lambda^+$,  $w(v(\lambda))\in  v(\lambda) - \sum _{\alpha\in \Delta} \mathbb N \alpha_a^\vee$ because 
$v(\lambda)$ lies in the cone  $\mathfrak D^+ \cap P(\Phi_a^\vee)$ of  dominant coweights \cite[VI.1.6, Prop.\ 18]{MR1890629}.  
The convex envelope  in $V_{\ad}$ of the $W_0$-conjugate of $\nu(\lambda)$  is a convex $W_0$-invariant polygon $\mathcal P(\lambda)$ contained in $\nu(\lambda)+\sum_{\alpha\in \Delta} \mathbb R_{\geq 0} \alpha_a^\vee.$ 
Let   $\lambda_1, \lambda_2 \in \Lambda^+$ such that $\lambda_1 \le \lambda_2$, hence $\lambda_1 \in \lambda_2 \prod _{\alpha \in \Delta}\lambda _\alpha^{\mathbb Z}$ by \eqref{Laf}. By Lemma \ref{poly},
  $\nu(\lambda_1)\in  \mathcal P(\lambda_2)$  hence $\lambda_1\in \lambda_2 \prod_{\alpha \in \Delta} \lambda_\alpha^{ \mathbb N}$. 
 This ends the proof of Proposition \ref{L+order}.

 \subsection{Bases of the pro-\texorpdfstring{$p$}{p} Iwahori Hecke ring}\label{B3}
The  pro-$p$ Iwahori  Hecke ring  of $G$ is a ring isomorphic to $\End_G (\ind_I^G \mathbb Z)$, where $I$ acts trivially on $ \mathbb Z$. We see the  pro-$p$ Iwahori ring  of $G$  as the convolution algebra $\mathcal H_{ \mathbb Z}$ of functions  $\varphi:G\to  \mathbb Z$ which are compactly supported and constant on the double cosets of $G$ modulo $I$. 
  The $ \mathbb Z$-module  $\mathcal H_{ \mathbb Z}$ has several  important bases  indexed by $w\in W(1)$.
  
 \bigskip  {\bf I})  A double coset $IxI$ for $x\in \cn$ depends only on the image $w\in W(1)$  of $x$ in the pro-$p$ Iwahori Weyl group $W(1)=\cn/Z(1)$ and is also denoted by $IwI$. The characteristic functions $T_w\in \mathcal H_{ \mathbb Z}$ of $IwI$ for $w\in W(1)$ form a natural basis of the $ \mathbb Z$-module $\mathcal H_{ \mathbb Z}$, called the Iwahori-Matsumoto basis.  Let $R$ be a commutative ring. We still denote by  $T_w$  the element $1 \otimes T_w$ in the $R$-algebra $\mathcal H_R=R\otimes_{\mathbb Z} \mathcal H_{\mathbb Z}$. 
 The definition of the other bases of  $\mathcal H_{ \mathbb Z}$ is more elaborate.

The relations verified by the basis elements $T_w\in \mathcal H_{ \mathbb Z}$  for $w \in W(1)$  are: 
\begin{itemize}
\item The braid relations $T_{w_1}T_{w_2}=T_{w_1w_2}$ if  $\ell(w_1)+\ell(w_2)=\ell(w_1w_2)$; hence $t\mapsto T_t$ gives an embedding $\mathbb Z[Z_k]\hookrightarrow  \mathcal H_{\mathbb Z}$.
\item The quadratic relations $T_{\tilde s}^2 =  q(s) T_{\tilde s^2} + c(\tilde s) T_{\tilde s}$ for $\tilde s \in   S^{\aff} (1)$ lifting a simple reflection $s\in S^{\aff}$. We have $\tilde s ^2\in Z_k,  q: \mathfrak S\to q^{\mathbb N}-\{1\}$ is a $W$-invariant function (for conjugation),  $c: \mathfrak S(1)\to  \mathbb Z[Z_k]$ is a $W (1)$-invariant function (for the conjugation action on   $Z_k$ and on  $\mathfrak S(1)$)  satisfying $c(wt)= c(tw)=tc(w)$ for $w\in \mathfrak S(1), t\in Z_k$. \end{itemize}
\begin{remark}[{\cite[\S 3.8, \S 4.2]{MR3484112}}]\label{cs} Let $s\in  S^{\aff}$. We denote by $H_s$  the affine hyperplane of $V_{\ad}$ fixed by $s$, $\alpha+r \in \Phi^{\aff}$ an affine root of $G$ \cite[3.5]{MR3484112} such that 
$H_s=\Ker (\alpha+r)$. Let $u\in (U_\alpha \cap \mathfrak K_s) \setminus \mathfrak K_s(1)$, $m(u)$ the only element in $\cn \cap U_{-\alpha}u U_{-\alpha}$ where $\mathfrak K_s$ is the parahoric subgroup of $G$ fixing the face of  $\mathfrak C^-$ contained in $H_s$. We have $q(s)=|Im(u)I/I|$ and the  image of $m(u)$ in $W(1)$ is a lift $\tilde s$ of $s$ contained in ${}_1W^{\aff}$. A lift  $\tilde s$ obtained in this way is called \emph{admissible}. 

The quotient of $\mathfrak K_s$ by its pro-$p$ radical $\mathfrak K_s(1)$ is  the group $G_{k,s}$ of rational points of a finite connected reductive $k$-group with maximal torus $Z_k$ and of semisimple rank $1$. Let $G'_{k,s}$ the subgroup of $G_{k,s}$ generated by the unipotent elements, $Z_{k,s}=Z_k\cap G'_{k,s}$. We have $Z_{k,s}\subset Z_k^{\aff}$ and   $c(\tilde s)\in \mathbb Z[Z_{k,s}]$.  This implies $c(w)\in  \mathbb Z[Z_k^{\aff}]$ for 
  $w\in {}_1\mathfrak S$.
   \end{remark}
 
{\bf II}) We now give the second basis \cite[Lemma 4.12, Prop.\ 4.13]{MR3484112}. There exist unique elements $T^*_w\in  \mathcal H_{ \mathbb Z}$ for $w\in W(1)$ such that 
\begin{itemize}
\item $T^*_{w_1}T^*_{w_2}=T^*_{w_1w_2}$ if $\ell(w_1)+\ell(w_2)=\ell(w_1w_2)$,
\item $T^*_u=T_u$ if $u\in \Omega(1)$ (i.e.\ $\ell(u)=0$),
\item $T^*_{\tilde s}=T_{\tilde s}- c(\tilde s)$ if $\tilde s\in S ^{\aff} (1)$.
\end{itemize}
They  form a  basis of  $\mathcal H_{ \mathbb Z}$, as the Iwahori-Matsumoto expansion of $T_w^*$ is triangular: \begin{equation}\label{T*}T^*_{\tilde w} =  \sum_{x\in W, x\leq w } h^*_x, \quad h^*_x=c^*(\tilde w,\tilde x) T_{\tilde x},
\end{equation}
where $\tilde w, \tilde x\in W(1)$  lift    $w,x\in W $, $c^*(\tilde w,\tilde x) \in  \mathbb Z[Z_k]$ 
($h^*_x$ does not depend on the choice of  $\tilde x$ lifting $x$) and $c^*(\tilde w,\tilde w)=1$.

\begin{remark} When the characteristic of $R$ is $p$ (in particular when $R=C$), we have  $q(s)=0$ in $R$ and $T_{\tilde s}^2 =  c(\tilde s) T_{\tilde s}, \ T_{\tilde s}^* T_{\tilde s}=T_{\tilde s} T_{\tilde s}^*=0$  for $\tilde s\in S^{\aff}(1)$; for an admissible lift  $\tilde s\in {}_1S^{\aff}$, 
\begin{equation}\label{cp}c(\tilde s )= -  | Z_{k,s}|^{-1}\sum _{t\in Z_{k,s}} T_t.
\end{equation}
\end{remark}

The $ \mathbb Z$-submodule $\mathcal H_ \mathbb Z^{\aff}$   with basis $T_w$ for $w\in {}_1W^{\aff}$ is a subalgebra,  $T^*_w$ for $w\in {}_1W^{\aff}$   is also a basis of  $\mathcal H_ \mathbb Z^{\aff}$, and  $c^*(\tilde w,\tilde x)\in \mathbb Z[Z_k^{\aff}]$ for $\tilde w, \tilde x\in {}_1W^{\aff}$. 

 For $\tilde w\in W(1)$ lifting $w\in W$, we have  \cite[Prop.\ 4.13]{MR3484112} 
\[T_w T^*_{w^{-1}}=q_w,\]
where 
 $w\mapsto q_w:W\to q^{\mathbb N}$  is the function defined by  \cite[Def.\ 4.14]{MR3484112} with properties
\begin{itemize}
\item $q_{w_1}q_{w_2}=q_{w_1w_2}$ if $\ell(w_1)+\ell(w_2)=\ell(w_1w_2)$,
\item $q_u=1$ if $u\in \Omega$ (i.e.\ $\ell(u)=0$),
\item $q_s=q(s)$ for $  s\in S^{\aff}$ as in the quadratic relation of $T(\tilde s)$.  
\end{itemize}
For $w_1,w_2\in W$, the positive square root
\[q_{w_1,w_2}= (q_{w_1} q_{w_2} q_{w_1 w_2}^{-1})^{1/2}\] 
belongs to $q^{\mathbb N}$ \cite[Lemma 4.19]{MR3484112} and   
$q_{w_1,w_2}= 1 \ \text {if and only if } \ \ell(w_1)+\ell(w_2)=\ell(w_1w_2) $ \cite[Lemma 4.16]{MR3484112}.
We inflate $q_w$ and $q_{w_1,w_2}$ to $W(1)$,  we put $q_{\tilde w}=q_w$ and $ q_{\tilde w_1,\tilde w_2}=q_{w_1,w_2}$ for $\tilde w, \tilde w_1, \tilde w_2   \in W(1)$ lifting $w, w_1,w_2$.

  \begin{remark}\label{c(s)}
 \cite[Prop.\ 4.13(6)]{MR3484112}. There is also a unique function $w\mapsto c_w: W^{\aff}(1)\to  \mathbb Z[Z_k]$ satisfying
$c_{w_1}c_{w_2}=c_{w_1w_2}$ if $\ell(w_1)+\ell(w_2)=\ell(w_1w_2)$,  
$c_{\tilde s}=c(\tilde s)$  for $\tilde s\in S^{\aff}(1)$, and $c_t = t$ for $t \in Z_k$.
\end{remark}

\begin{remark}\label{braid} Some properties of  $c^*(w, x) $ for $x,w\in W(1), x\leq w$, follow easily from the braid relations for $T^*_{w}$ and $T_{x}$:  
\begin{enumerate} 
\item  For $t\in Z_k$, we have $c^*(t w, x) = t c^*( w, x) $ and 
$c^*( w, x t)  x t  x^{-1} = c^*( w,t x) t=c^*( w, x)$ because $T^*_{t  w}=T_t T^*_{ w}$ and 
$c^*( w, x) T_{ x}=c^*( w, x t)T_{ x t}=c^*( w, x t)T_{ x t  x^{-1}}T_{ x}=c^*( w,t x)T_{t x }=c^*( w, t x)T_{t}T_{ x }$.

\item  For $ v\in \Omega(1)$ %
we have $c^*( w  v, x  v)=c^*( w, x)$ because  $T^*_{ w}T_{ v}=T^*_{ w  v}$ and $T_{ x}T_{ v}=T_{ x v}$. 
 \end{enumerate}\end{remark}

{\bf III}) The other bases   of $ \mathcal H_{\mathbb Z}$  are associated to spherical orientations   of $V_{\ad}$; they generalize the Bernstein basis of an affine Hecke algebra. The spherical orientations are in one-to-one correspondence  with the Weyl chambers of $V_{\ad}$ (cf.\ \cite[Def.\ 5.16]{MR3484112}). 
If $\mathfrak D_o$ is the Weyl chamber of a spherical orientation $o$  and $w\in W(1)=\cn/Z(1)$ an element  of image $w_0\in W_0=\cn/Z$, we denote by $o\cdot w$ the orientation of Weyl chamber $w_0^{-1}(\mathfrak D_o)$. In particular $o\cdot \lambda=o$ when $\lambda \in \Lambda(1)=Z/Z(1)$. There is a  basis $E_o(w)$ for $w\in W(1)$ of $ \mathcal H_{\mathbb Z}$  associated to each spherical orientation $o$ \cite[\S 5.3]{MR3484112}.

The main properties of the elements $E_o(w)$ are: \begin{itemize}
\item Multiplication formula $E_o(w_1)E_{o\cdot w_1}(w_2)=q_{w_1,w_2} E_o(w_1 w_2)$  for $w_1,w_2\in W(1)$.
\item Triangular Iwahori-Matsumoto expansion  \cite[Cor.\ 5.26]{MR3484112}
\begin{equation}\label{Et}E_o (\tilde w) =  \sum_{x\in W, x\leq w } h_o(x), \quad h_o(x)=c_o (\tilde w,\tilde x) T_{\tilde x},
\end{equation}
where $\tilde w, \tilde x\in W(1)$  lift    $w,x\in W $,  $c_o (\tilde w,\tilde x)\in  \mathbb Z[Z_k]$ ($ h_o(x) $ does not depend on the choice of  $\tilde x$ lifting $x$)  and $c_o (\tilde w,\tilde w)=1$.
 \item $E_o(\lambda)=\begin{cases} T_\lambda  &\ \text{if} \  \nu(\lambda)\in \mathfrak D_o\\
T^*_\lambda  &\ \text{if} \   \nu(\lambda)\in  - \mathfrak D_o
\end{cases}$ for $\lambda\in \Lambda(1)$.
\end{itemize}
 When $R$ is a ring of characteristic $p$ (in particular $R=C$),
 in $\mathcal H_{R} $ we have

$E_o(w_1)E_{o\cdot w_1}(w_2)=\begin{cases} E_o(w_1 w_2)  &\ \text{if}  \ \ell(w_1)+\ell(w_2)=\ell(w_1w_2), \\ 
0   &\ \text{otherwise}.
\end{cases}$

\begin{remark}\label{bernstein}
The integral Bernstein basis $(E(w)=E_{o^-}(w))_{w\in W(1)}$ is the basis associated to the spherical orientation $o^{-}$ corresponding to the antidominant Weyl chamber $\mathfrak D^-$ \eqref{D-}. %
\end{remark}

 For $x\in \cn$ of image $w\in W(1)$ we write also   $T(x)=T_w, T^*(x)=T_w^*, E_o(x )=E_o(w)$. 

\subsection{Representations of \texorpdfstring{$K$}{K} and Hecke modules}\label{sec:repr-of-K-hecke}
The  submodule $\mathcal H_{\mathbb Z}(K,I)$ of  functions with support in $K$ in the pro-$p$ Iwahori Hecke algebra  $\mathcal H_{\mathbb Z}$ is the submodule of  basis $T_{ w}$ for $w\in W_0(1)$; it  is a   subalgebra of $\mathcal H_{\mathbb Z}$  canonically isomorphic to the algebra of intertwiners $\End _K (\ind_I^K \mathbb Z)$. 

We may view  $\mathcal H_{\mathbb Z}(K,I)$ as the   convolution algebra $\mathcal H_{\mathbb Z}(G_k,U_{k,\op} )$ of functions $G_k\to  \mathbb Z$ which are constant on the double cosets modulo $U_{k,\op} $.
The irreducible representations $V$ of $G_k$ are in one-to-one correspondence with the characters 
      of $\mathcal H_C (G_k,U_{k,\op} )$  \cite[Cor.\ 7.5]{MR0396731}, \cite[Thm.\ 6.10]{MR2057756}. The representation $V$ corresponds to the character $\chi $  giving the action of $\mathcal H_C (G_k,U_{k,\op} )$ on the line $V^{U_{k,\op} }$. We consider $V$ as an irreducible  representation   of $K$ and $\chi $ as  a character of $\mathcal H_C (K,I)$ giving the action of $\mathcal H_C (K,I)$ on $V^I= V^{U_{k,\op} }$.

      A character $\chi$ of $\mathcal H_C (K,I )$ is determined by a $C$-character $\psi_{\chi}$ of $Z^0$ 
     such that  $\psi_\chi (t)= \chi (T(t))$ for $t\in Z^0$  and by the subset $\Delta(\chi)$ of $\Delta_{\psi_\chi }$ \eqref {deltapsi} defined by
     \begin{equation}\label{zero}\chi (T_{\tilde s_\alpha})=\begin{cases} -1\ & \ \text{if} \  \alpha \in  \Delta_{\psi_\chi  }\setminus \Delta(\chi) \\
      0\ & \ \text{if} \   \alpha \in  \Delta (\chi) \ \text{or} \ \alpha\not\in  \Delta_{\psi_\chi } 
      \end{cases},
  \end{equation} 
  where $\tilde s_\alpha$ is an admissible lift of $s_\alpha$ (Remark \ref{cs}).  The pair $(\psi_\chi, \Delta(\chi))$ is called the \emph{parameter of $\chi$}.    \begin{itemize}
   \item
 $V= V(U_k)\oplus V^{U_{k,\op}}$ where $V(U_k)$ is the kernel of the quotient map $V\onto V_{U_k}$  \cite[Thm.\ 6.12]{MR2057756}.
 In particular,  $Z_k$ acts on the lines $V^{U_{k,\op}}$ and $V_{U_k}$ by the same character $\psi_V$.
 \item  The stabilizer of $V^{U_{k,\op}}$ in $G_k$ is the parabolic subgroup  $P_{\Delta(\chi),k,\op}$  \cite[Prop.\ 6.6, Thm.\ 7.1]{MR0396731}.  
 \item The stabilizer of $V(U_k)$ in $G_k$ is the parabolic subgroup $P_{\Delta(V),k}$ (see \S \ref{GS}).
 \end{itemize}
    \begin{lemma}\label{repKmodH}    The parameter  $(\psi_V, \Delta (V))$ of $V$ and the parameter $(\psi_\chi, \Delta (\chi) )$ of $\chi$ satisfy $ \psi_V=\psi_\chi^{-1}, \ \Delta (V) =\Delta (\chi) $.
    \end{lemma}
  \begin{proof}  
We have $fT(t^{-1})=t f$ for $t\in Z_k$ hence $\psi_{\chi}=\psi _V^{-1}$, because
\begin{equation*}fh = \sum_{x\in I \backslash K} h(x) x^{-1} f \quad \text{for} \ h\in \mathcal H_C (K,I), \ f\in V^{I}.
\end{equation*}
  Let $w_{\Delta}$ be the longest element of $W_0$. The group $U_{k,\op}$ is conjugate to $U_k$ by $w_{\Delta}$, 
  the stabilizer $P_{\Delta(\chi),k,\op}$ of $V^{U_{k,\op}}$ is the conjugate by $w_{\Delta}$  of the stabilizer of the line $V^{U_k}$, which  is $P_{-w_{\Delta}(\Delta(V)),k}$ \cite[III.9 Remark 1]{MR3600042}. Hence $\Delta (V) =\Delta (\chi)$. 
   \end{proof}

 \subsection{The elements \texorpdfstring{$c_w^x\in \mathbb Z[Z_k]$}{c\_w\^{}x in Z[Z\_k]}}\label{sec:elements-cwx}

Our motivation is to explicitly  compute  the expansion of  $T^*_{ w}$ in the Iwahori-Matsumoto basis  in  $\mathcal H_{\mathbb Z}$  modulo $q$ (Theorem \ref{*}). We  associate to the function $c :\mathfrak S(1) \to \mathbb Z[Z_k]$ defining the quadratic relation  of $T_s$ for $s\in S^{\aff}(1)$, 
 elements $$c_w^x\in\mathbb Z[Z_k] \quad \text{ for } x,w\in W(1), \ x\leq w,$$
 and we study their properties.

\begin{notation}\label{notationc} The  action of $W(1)$ by conjugation on $Z_k$   factors through $W$ and we write $w\cdot c=\tilde wc \tilde w^{-1}$ for $c\in \mathbb Z[Z_k]$ and  $\tilde w\in W(1) $ lifting $w\in W $. We write also $  w_1\cdot w_2= w_1 w_2 w_1^{-1}$ for $ w_1 ,  w_2$ in $W(1)$ (or $ w_1 ,  w_2$ in $W$).

For a sequence $\underline {\tilde w}=(\tilde s_1, \ldots, \tilde s_n)$ in $S^{\aff}(1)$ lifting a sequence $\underline w= (s_1, \ldots, s_n)$ in $S^{\aff}$, write
$\tilde w :=\tilde s_1 \cdots  \tilde s_n, w:=s_1\cdots s_n$  %
for the  products of the terms of the  sequences. %
We take $1$  for  the   ``product of the terms'' of  the empty sequence $( \ )$. 
The   lifts   of the sequence $\underline  w$ in $S^{\aff}$ are the sequences $(t_1\tilde s_1, \dots, t_n\tilde s_n)$ in $S^{\aff}(1)$, where $t_i \in Z_k$.
 \end{notation}

\begin{definition} \label{c_^} Let   $\underline {\tilde w}=(\tilde s_1, \ldots, \tilde s_n)$ be  a sequence in  $S ^{\aff} (1)$ and $\underline {\tilde x}=(\tilde s_{i_1} , \ldots,  \tilde s_{i_r}) $ with $1 \le i_1 < \cdots < i_r \le n$
a subsequence of $\underline {\tilde w}$.  We  define
$c_{\underline {\tilde w}}^{\underline{ \tilde x}} $ as the product  of the following elements of $\mathbb Z[Z_k]$:

$c(\tilde s_1)\cdots  c(\tilde s_{i_1-1})$

$ s_{i_1}\cdot (c(\tilde s_{i_1+1})\cdots  c(\tilde s_{i_2-1}))$

$ s_{i_1} s_{i_2}\cdot (c(\tilde s_{i_2+1})\cdots  c(\tilde s_{i_3-1}))
$

$\cdots$

$ s_{i_1} \cdots s_{i_r}\cdot (c(\tilde s_{i_{r}+1})\cdots c(\tilde s_{i_n})).$
\end{definition}

\begin{remark}
  Strictly speaking, for the subsequence $\underline {\tilde x}$ we need to remember the sequence of integers $i_1 < \cdots < i_r$.
\end{remark}

\begin{example}\label{exc}
We have 
 $c_{\underline {\tilde w}}^{\underline {\tilde w}} =1$.
 
 When  $\tilde w=\tilde s_1 \cdots  \tilde s_n$ is a reduced decomposition, we have 
 $c_{\underline {\tilde w}}^{( \ )} =c_{ {\tilde w}} $ (Remark \ref{c(s)}).
\end{example}

Take $1\leq m \leq n$ and cut the sequences $\underline {\tilde w}$ and  $\underline {\tilde x}$ in two:
 $\underline {\tilde w}=  \underline {\tilde w}_1   \underline {\tilde w}_2 $  and  $\underline {\tilde x}=  \underline {\tilde x}_1   \underline {\tilde x}_2 $ with $ \underline {\tilde w}_1 =(\tilde s_1, \ldots, \tilde s_m),  \underline {\tilde w}_2 =(\tilde s_{m+1}, \ldots, \tilde s_n)$, $ \underline {\tilde x}_1 =(\tilde s_{i_1} , \ldots,  \tilde s_{i_t}),  \underline {\tilde x}_2 =(\tilde s_{i_{t+1}} , \ldots,  \tilde s_{i_r})$ where $i_t\leq m <i_{t+1}$.   The sequence decompositions $\underline {\tilde w}=  \underline {\tilde w}_1   \underline {\tilde w}_2 $  and  $\underline {\tilde x}=  \underline {\tilde x}_1   \underline {\tilde x}_2 $  are called \emph{compatible}.  For $i=1,2,$
the sequence  $ \underline {\tilde x}_i $ is a subsequence of  $ \underline {\tilde w}_i $ and we have   $c_{ \underline {\tilde w}_i }^{\underline{ \tilde x}_i} $. The  terms in the product defining  $c_{ \underline {\tilde w}_1 }^{ \underline {\tilde x}_1 } $
or $x_1\cdot c_{ \underline {\tilde w}_2 }^{ \underline {\tilde x}_2 } $ appear in  the product defining $c_{\underline {\tilde w}}^{\underline{ \tilde x}} $ except
the last term $x_1\cdot (c(\tilde s_{i_{t}+1})\cdots c(\tilde s_m))$ of  $c_{ \underline {\tilde w}_1 }^{ \underline {\tilde x}_1 } $ 
and the first term
$x_1\cdot (c(\tilde s_{m+1})\cdots     c(\tilde s_{i_{t+1}-1}))$  of $x_1\cdot c_{ \underline {\tilde w}_2 }^{ \underline {\tilde x}_2 } $;  their  product  $x_1\cdot (c(\tilde s_{i_{t}+1})\cdots   c(\tilde s_{i_{t+1}-1}))$ appears in $c_{\underline {\tilde w}}^{\underline{ \tilde x}} $.  Then, we get a one-to-one correspondence with the terms appearing in  the product defining $c_{\underline {\tilde w}}^{\underline{ \tilde x}} $:
 \begin{equation}\label{cwx}
 c_{\underline {\tilde w}}^{\underline{ \tilde x}} = c_{ \underline {\tilde w}_1 }^{ \underline {\tilde x}_1 } \, (x_1\cdot c_{ \underline {\tilde w}_2 }^{ \underline {\tilde x}_2 }). 
 \end{equation}
   This  useful formula  allows us to study $c_{\underline {\tilde w}}^{\underline{ \tilde x}}$ by  induction on the length $n$ of $\underline {\tilde w}$.
  
\begin{example}\label{induction}  When $ \underline {\tilde x}_2 = \underline {\tilde w}_2 $ we have 
  $c_{\underline {\tilde w}}^{\underline{ \tilde x}} = c_{ \underline {\tilde w}_1 }^{ \underline {\tilde x}_1 }$.  
  
    When $ m=n-1$ and $i_r<n$, we have  
 $c_{\underline {\tilde w}}^{\underline{ \tilde x}} = 
c_{ \underline {\tilde w}_1 }^{\underline{ \tilde x}} (x\cdot c(\tilde s_n))
   $.
   \end{example}
By iteration of \eqref{cwx} we deduce:

\begin{lemma}\label{inductions} Let   $\underline {\tilde w} $   and $\underline {\tilde x}$ be two sequences in $ S^{\aff}(1)$   such that  $\underline {\tilde x}$ is a subsequence of $\underline {\tilde w} $ and  consider compatible sequences decompositions  $\underline {\tilde w}= \underline {\tilde w}_1\cdots  \underline {\tilde w}_k $ and   $\underline {\tilde x}= \underline {\tilde x}_1\cdots   \underline {\tilde x}_k $. Then    $$c_{\underline {\tilde w}}^{\underline{ \tilde x}}   \ = \ c_{ \underline {\tilde w}_1 }^{ \underline {\tilde x}_1 } \ 
( x_1\cdot c_{ \underline {\tilde w}_2 }^{ \underline {\tilde x}_2 }) \ 
 (x_1x_2\cdot c_{ \underline {\tilde w}_3 }^{ \underline {\tilde x}_3 }) \ 
 \cdots \ 
(x_1\cdots x_{k-1} \cdot c_{ \underline {\tilde w}_k }^{ \underline {\tilde x}_k }) .$$
\end{lemma}
The function $c:S^{\aff}(1)\to \mathbb Z[Z_k]$ satisfies: \begin{lemma}\label{s.s} For $\tilde s\in  S^{\aff}(1)$ lifting $s\in S^{\aff}$ and $c\in \mathbb Z[Z_k]$, we have 
 $s\cdot c(\tilde s)=c(\tilde s)$ and $c(\tilde s)\, c=c(\tilde s)\, (s\cdot c).$ 
 \end{lemma}
\begin{proof} The equalities $c(\tilde s)\, t=c(\tilde s)\, (s\cdot t)$ for $t\in Z_k$ and  $c(\tilde s)\, c=c(\tilde s)\, (s\cdot c)$  for $c\in \mathbb Z[Z_k]$ are equivalent. 
Suppose that  $\tilde s$ is an admissible lift of $s$ (Remark \ref{cs}). Then, the  lemma  is proved in  \cite[Prop.\ 4.4]{MR3484112}. The other lifts of $s$ are $\tilde s t$ for   $t\in Z_k$ and 
$s\cdot c(\tilde s t)=s\cdot (c(\tilde s )t)= (s\cdot c(\tilde s)) \, (s\cdot t)=c(\tilde s) t =c(\tilde s t)$.  For $t,t'\in Z_k$, we have  $c(\tilde s t)\, t'=c(\tilde s) tt' = c(\tilde s) \, (s\cdot tt')=c(\tilde s)t \, (s\cdot t') = c(\tilde s t) \, (s\cdot t') $.
\end{proof}

\begin{lemma}\label{xw} Let   $\underline {\tilde w} $   and $\underline {\tilde x}$ be two sequences in $ S^{\aff}(1)$   such that  $\underline {\tilde x}$ is a subsequence of $\underline {\tilde w} $  and let $c\in \mathbb Z[Z_k]$. Then,  
$c_{\underline {\tilde w}}^{\underline{ \tilde x}} \, (x\cdot c) = c_{\underline {\tilde w}}^{\underline{ \tilde x}} \, (w\cdot c)$.
\end{lemma}
\begin{proof} We cut  the sequences $ \underline{\tilde w}$ and  $ \underline{\tilde x}$ in two (as above with $m=n-1$). Let   $ \underline{\tilde w}_1=(\tilde s_1, \ldots, \tilde s_{n-1}), \underline{\tilde w}_2=(\tilde s_n)$. 

When $i_r=n$, applying  Example \ref{induction}  we have
$c_{\underline {\tilde w}}^{\underline{ \tilde x}}= c_{ \underline {\tilde w}_1 }^{ \underline {\tilde x}_1 }  $ where    $ \underline {\tilde x}_1 =(\tilde s_{i_1}, \ldots, \tilde s_{i_{r-1}})$. By induction on $n$, $c_{ \underline {\tilde w}_1 }^{ \underline {\tilde x}_1 } 
\, (x_1\cdot c) =c_{ \underline {\tilde w}_1 }^{ \underline {\tilde x}_1 }  \, (w_1\cdot c)$. Hence   $c_{\underline {\tilde w}}^{\underline{ \tilde x}} \, (x\cdot c) = c_{ \underline {\tilde w}_1 }^{ \underline {\tilde x}_1 } 
\, (x_1s_n\cdot c) =c_{ \underline {\tilde w}_1 }^{ \underline {\tilde x}_1 }  \, (w_1s_n\cdot c) =  c_{\underline {\tilde w}}^{\underline{ \tilde x}} \, (w\cdot c)$.

 When $i_r\neq n$, applying Example \ref{induction} (twice),  Lemma~\ref{s.s}, as well as induction on $n$  we have 
    $c_{\underline {\tilde w}}^{\underline{ \tilde x}} \, (x\cdot c) = 
c_{ \underline {\tilde w}_1 }^{\underline{ \tilde x}} (x\cdot c(\tilde s_n)c)  =    c_{ \underline {\tilde w}_1 }^{\underline{ \tilde x}} (x\cdot c(\tilde s_n)(s_n \cdot c))
= c_{ \underline {\tilde w}_1 }^{\underline{ \tilde x}} (x\cdot c(\tilde s_n))(x s_n \cdot c) = c_{ \underline {\tilde w}_1 }^{\underline{ \tilde x}} 
(x\cdot c(\tilde s_n))(w_1 s_n \cdot c) = c_{\underline {\tilde w}}^{\underline{ \tilde x}}  \, (w\cdot c)$.
\end{proof}

\begin{proposition} \label{Waf1} Let   $\underline {\tilde w} $ be a   sequence  in $S^{\aff}(1)$ and   $\underline {\tilde x}$ a subsequence of $\underline {\tilde w} $ such that  $\tilde w=\tilde s_1 \cdots  \tilde s_n$ and $\tilde x= \tilde s_{i_1}  \cdots \tilde s_{i_r}$  are reduced decompositions \(i.e.\ $n=\ell(w), r=\ell(x)$\), and $t, u\in Z_k$. Then the product $t u^{-1}c_{\underline {\tilde w}}^{\underline{ \tilde x}}$ depends only on $t \tilde w,u \tilde x \in W(1)$.
\end{proposition}
\begin{proof}  We have to prove 
 $t u^{-1}c_{\underline {\tilde w}}^{\underline{ \tilde x}}=t' u'^{-1}c_{  \underline {\tilde w}' }^{ \underline {\tilde x}' }$, when 
 $   \underline {\tilde w}' =(\tilde s'_1, \ldots  ,\tilde s'_n)$ is a sequence in $S^{\aff}(1)$, 
 $\underline {\tilde x}' = (\tilde s'_{j_1},  \ldots, \tilde s'_{j_r})$ is a subsequence of $\underline {\tilde w}' $ and $t', u'$ are elements in $ Z_k$, satisfying  
  $t\tilde w=t'\tilde w' $ and $u\tilde x= u'\tilde x'$.  Then $w,w'$ have the same length $n$,  and  $x,x'$ have the same length $r$.  The proof  is divided into several  steps and uses   induction on $n$. 
 
 {\bf A}) Assume    $ \underline {\tilde w} = \underline {\tilde w}'$. Then $t =t'$ and we will prove  $u^{-1}c_{\underline {\tilde w}}^{\underline{ \tilde x}}=u'^{-1}c_{  \underline {\tilde w}}^{ \underline {\tilde x}' }$. By symmetry, we have three cases:
 $$\text{(1)  $i_r=j_r=n$, \ (2) $i_r<n$ and $j_r<n$, \ (3)  $i_r=n$ and $j_r<n$}.$$
We denote by  $\underline {\tilde w}^\flat, \underline w^\flat$ the sequences obtained by erasing the last term of in the sequences $\underline {\tilde w}, \underline w$;  the products of the terms in $\underline {\tilde w}^\flat$ and of $\underline w^\flat$ are denoted by $ {\tilde w}^\flat$ and $w^\flat$.  We examine each case separately, using 
 Example \ref{induction}. We have:
 
(1)    $ c_{\underline {\tilde w}}^{\underline{ \tilde x}}=  c_{\underline {\tilde w}^\flat}^{\underline{ \tilde x}^\flat},  c_{  \underline {\tilde w}^\flat}^{\underline {\tilde x}'^\flat}= c_{  \underline {\tilde w}}^{\underline {\tilde x}'}$. By induction on $n$,  $u^{-1}c_{\underline {\tilde w}}^{\underline{ \tilde x}}=u'^{-1}c_{  \underline {\tilde w}}^{ \underline {\tilde x}' }$.

(2) $ c_{\underline {\tilde w}}^{\underline{ \tilde x}}= c_{\underline {\tilde w}^\flat }^{\underline{ \tilde x}} (x\cdot c(\tilde s_n)) $  and   $c_{\underline {\tilde w}}^{\underline {\tilde x}'}=c_{\underline {\tilde w}^\flat}^{\underline {\tilde x}'} (x'\cdot c(\tilde s_n)) $.  By induction on $n$, and noting that $x = x'$, $u^{-1}c_{\underline {\tilde w}}^{\underline{ \tilde x}}=u'^{-1}c_{  \underline {\tilde w}}^{ \underline {\tilde x}' }$.

(3)   $c_{\underline {\tilde w}}^{ \underline {\tilde x}}= c_{\underline {\tilde w}^\flat}^{ \underline {\tilde x}^\flat}$  and $c_{\underline {\tilde w}}^{ \underline {\tilde x}' }=c_{\underline {\tilde w}^\flat}^{ \underline {\tilde x}' } (x'\cdot c(\tilde s_{i_r}))$.
 Since $s_{i_1}\cdots s_{i_r}= s_{j_1}\cdots s_{j_r}$ are reduced decompositions, by the exchange condition there exists $1\leq k \leq r$  such that $s_{j_{k+1}}\cdots s_{j_r}s_{i_r}=s_{j_k}\cdots s_{j_r}$ and $x^\flat= s_{i_1}\cdots s_{i_{r-1}}=s_{j_1}   \cdots s_{j_{k-1}} s_{j_{k+1}} \cdots  s_{j_r}$.   Suppressing the $k$-th term   of the sequence   $ \underline {\tilde x}'  $  we get  $\underline {\tilde x}'^\star = (\tilde s_{j_1},  \ldots,\tilde s_{j_{k-1}},\tilde s_{j_{k+1}},\ldots, \tilde s_{j_r})$ and  ${\tilde x}'^\star=\tilde s_{j_1}   \cdots\tilde s_{j_{k-1}} \tilde s_{j_{k+1}} \cdots \tilde s_{j_r}$ lifting  $x^\flat$. 
  Let $u''\in Z_k$ such that 
 $u \tilde  x^\flat= u'' \tilde x'^\star $. By induction on $n$,  $u^{-1}c_{\underline {\tilde w}^\flat}^{\underline {\tilde x}^\flat}=
u''^{-1}c_{\underline {\tilde w}^\flat}^{ \underline {\tilde x}'^\star}$;   hence 
 \begin{equation} \label{s1} u''^{-1}c_{\underline {\tilde w}^\flat}^{ \underline {\tilde x}'^\star}= u'^{-1}c_{\underline {\tilde w}^\flat}^{ \underline {\tilde x}' } (x'\cdot c(\tilde s_{i_r})) 
\end{equation}
implies $u^{-1}c_{\underline {\tilde w}}^{\underline{ \tilde x}}=u'^{-1}c_{  \underline {\tilde w}}^{ \underline {\tilde x}' }$. 
We now prove \eqref{s1}.
Applying  Lemma \ref{inductions} to the compatible decompositions   $\underline {\tilde w}^\flat= \underline {\tilde w}_1  ( \tilde s_{j_k}) \underline{ \tilde w}_3$,   $\underline {\tilde x}'^\star=\underline{ \tilde x}'_1  ( \ ) \underline {\tilde x}'_3$, and $ \underline {\tilde x}' =\underline {\tilde x}'_1( \tilde s_{j_k}) \underline {\tilde x}'_3$ we get 
 $c_{\underline {\tilde w}^\flat}^{ \underline {\tilde x}'^\star}= c_{\underline{\tilde w }_1}^{ \underline{\tilde x}'_1 }\, (x'_1\cdot c( \tilde s_{j_k}) c_{\underline{\tilde w }_3}^{ \underline{\tilde x}'_3})$ and $c_{\underline {\tilde w}^\flat}^{ \underline {\tilde x}' }=c_{\underline{\tilde w }_1}^{ \underline{\tilde x}'_1 }\,  ( x'_1  s_{j_k}\cdot c_{\underline{\tilde w }_3}^{ \underline{\tilde x}'_3} )$. We have $c(  \tilde s_{j_k}) c_{\underline{\tilde w }_3}^{ \underline{\tilde x}'_3} = c(  \tilde s_{j_k})\, (s_{j_k}\cdot c_{\underline{\tilde w }_3}^{ \underline{\tilde x}'_3})$ by Lemma \ref{s.s} so that $c_{\underline {\tilde w}^\flat}^{ \underline {\tilde x}'^\star}= c_{\underline {\tilde w}^\flat}^{ \underline {\tilde x}' }\, (x'_1\cdot c( \tilde s_{j_k})) $. Hence
\begin{equation} \label{s2} 
u''^{-1} (x'_1\cdot c(  \tilde s_{j_k}))= u'^{-1}  (x'\cdot c(\tilde s_{i_r}))
\end{equation}
implies \eqref{s1}. We now prove \eqref{s2}.
We have  $u' \tilde s_{j_1} \cdots \tilde s_{j_r}= u'' \tilde s_{j_1}\cdots \tilde s_{j_{k-1}}  \tilde s_{j_{k+1}} \cdots  \tilde s_{j_r}  \tilde s_{i_r}$. 
Therefore $u'((\tilde s_{j_1} \cdots \tilde s_{j_r})\cdot \tilde s_{i_r}^{-1}) =u''((\tilde s_{j_1} \cdots \tilde s_{j_{k-1}})\cdot \tilde s_{j_k}^{-1}) $. Taking the inverse shows $(\tilde x'\cdot \tilde s_{i_r}) {u' }^{-1}=(\tilde  x'_1\cdot \tilde s_{j_k})u''^{-1}$ and $ u'^{-1}  (x'\cdot c(\tilde s_{i_r}))  = (x'\cdot c(\tilde s_{i_r}))  u'^{-1} = c((\tilde x'\cdot \tilde s_{i_r}){u'}^{-1})=c((\tilde  x'_1\cdot \tilde s_{j_k})u''^{-1})   =c(\tilde  x'_1\cdot \tilde s_{j_k})u''^{-1} $
$=u''^{-1} (x'_1\cdot c(  \tilde s_{j_k})).$ This ends the proof of case  {\bf A}).
 
\bigskip {\bf B}) Assume $ \underline {w} = \underline {w}'$. We will prove that $t u^{-1}c_{\underline {\tilde w}}^{\underline{ \tilde x}}=t' u'^{-1}c_{  \underline {\tilde w}' }^{ \underline {\tilde x}' }$ by induction on $n$. When $n = 1$ this follows from the following identities for $a \in Z_k$:
$c_{(a \tilde s_1)}^{(\ )} = c(a \tilde s_1) = a c(\tilde s_1) = a c_{(\tilde s_1)}^{(\ )}$ and 
$c_{(a \tilde s_1)}^{(a \tilde s_1)} = 1 = c_{(\tilde s_1)}^{(\tilde s_1)}$. For $n > 1$ we will reduce to case \textbf{A}) as follows.
Let $\underline x'' = (\tilde s'_{i_1} , \ldots,  \tilde s'_{i_r})$. Choose non-trivial decompositions
$\underline {\tilde w}=\underline {\tilde  w }_1\underline {\tilde  w}_2$, $\underline {\tilde w}'=\underline {\tilde  w }'_1\underline {\tilde  w}'_2$
with $\ell(w_i) = \ell(w_i') > 0$ for $i = 1,2$. Then we have compatible decompositions $\underline {\tilde x}=\underline {\tilde  x }_1\underline {\tilde  x}_2$
and $\underline {\tilde x}''=\underline {\tilde  x}''_1\underline {\tilde  x}''_2$. In particular, $w_i = w'_i$, $x_i = x''_i$,
and we can choose $t_i, u_i \in Z_k$ such that
$\tilde w_i = t_i \tilde w'_i$, $u_i \tilde x_i = \tilde x''_i$ for $i = 1,2$.
By induction we have that $u_i^{-1} c_{\underline {\tilde w}_i}^{\underline{ \tilde x}_i} = t_i c_{\underline {\tilde w}'_i}^{\underline{ \tilde x}''_i}$.
Hence from \eqref{cwx} and Lemma~\ref{xw} we get
\begin{multline*}
  c_{\underline {\tilde w}}^{\underline{ \tilde x}} = c_{ \underline {\tilde w}_1 }^{ \underline {\tilde x}_1 } \, (x_1\cdot c_{ \underline {\tilde w}_2 }^{ \underline {\tilde x}_2 }) = 
t_1 u_1 c_{ \underline {\tilde w}'_1 }^{ \underline {\tilde x}''_1 } \, (x_1''\cdot t_2 u_2 c_{ \underline {\tilde w}'_2 }^{ \underline {\tilde x}''_2 })
= t_1 (w_1' \cdot t_2) u_1 (x_1 \cdot u_2) c_{ \underline {\tilde w}'_1 }^{ \underline {\tilde x}''_1 } (x_1'' \cdot c_{ \underline {\tilde w}'_2 }^{ \underline {\tilde x}''_2 }) \\
= t_1 (w_1' \cdot t_2) u_1 (x_1 \cdot u_2) c_{ \underline {\tilde w}' }^{ \underline {\tilde x}''}.
\end{multline*}
Hence $t u^{-1} c_{\underline {\tilde w}}^{\underline{ \tilde x}} = t' u^{-1} u_1 (x_1 \cdot u_2) c_{ \underline {\tilde w}' }^{ \underline {\tilde x}''}$.
This equals $t' u'^{-1}c_{  \underline {\tilde w}' }^{ \underline {\tilde x}' }$ by case \textbf{A}), since
$u u_1^{-1} (x_1 \cdot u_2)^{-1} \tilde x'' = u' \tilde x'$.

\bigskip {\bf C}) Assume  that $\u w = (s, s', s,\dots)$, $\u w' = (s', s, s',\dots)$, where $w = s s' s \cdots = s' s s' \cdots = w'$ is a braid relation in $W^{\aff}$.
Choose lifts $\tilde s, \tilde s' \in S^{\aff}(1)$ of $s, s' \in S^{\aff}$. Then by part \textbf{B}) we may assume without loss of generality that
$\u {\tilde w} = (\tilde s, \tilde s', \tilde s, \dots)$, $\u {\tilde w}' = (\tilde s', \tilde s, \tilde s', \dots)$. 
(Use the same integers $i_1 < \cdots < i_r$ for the old and the new $\u {\tilde w}$, and similarly for $\u {\tilde w}'$.)
Then the case  $r=n$ is obvious because  $\underline {\tilde w}=\underline {\tilde x}, \ \underline {\tilde w}' = \underline {\tilde x}' , \ tu^{-1}=t'{u'}^{-1}$ and $c_{\underline {\tilde w}}^{\underline{ \tilde x}}=c_{ \underline {\tilde w}' }^{ \underline {\tilde x}' }=1$, so we  assume $r<n$.
We prove $tu^{-1}c_{\underline {\tilde w}}^{\underline{ \tilde x}}=t' u'^{-1}c_{\underline {\tilde w}'}^{\underline {\tilde x}'}$. 

As $r < n$ the sequence  $ \underline x'=   \underline x$ is unique. By symmetry we suppose that the last terms of $\underline{ w}$ and  $ \underline { x}$ are equal.

(1) We reduce   to the case where $i_k=n-r+k$ and $j_k=n-1-r+k$ for all $1 \le k \le r$. 
For  $\underline{ \tilde y}=(\tilde s_{n-r+1},\ldots ,\tilde s_n)$ and $\tilde y=\tilde s_{n-r+1}\cdots \tilde s_n$, we have $\tilde x= \tilde y$. By {\bf A}),   $ c_{\underline {\tilde w}}^{\underline{ \tilde x}}=c_{\underline {\tilde w}}^{\underline{ \tilde y}}$. As  $s'_{j_r}=s_{i_r}=s_n= s'_{n-1}$,  we have similarly for $\underline{ \tilde y}'=(\tilde s'_{n-r},\ldots ,\tilde s'_{n-1})$, $\tilde x'=  \tilde y'$ and $ c_{\underline {\tilde w}'}^{\underline {\tilde x}'}= c_{\underline {\tilde w}'}^{\underline{ \tilde y}'}$.  We have $u' \tilde y'=   u\tilde y$ and the equalities $tu^{-1}c_{\underline {\tilde w}}^{\underline{ \tilde x}}=t' u'^{-1}c_{\underline {\tilde w}'}^{\underline {\tilde x}'}$ and 
$tu^{-1}c_{\underline {\tilde w}}^{\underline{ \tilde y}}=t' u'^{-1}c_{\underline {\tilde w}'}^{\underline {\tilde y}'}$ are equivalent.

(2) We assume $i_k=n-r+k$ and $j_k=n-1-r+k$ for $1 \le k \le r$. %
Then  $\underline {\tilde x}= \underline {\tilde x}'$ and $u = u'$ as $\underline x=\underline x'$. We prove $tc_{\underline {\tilde w}}^{\underline{ \tilde x}}=t'c_{\underline {\tilde w}'}^{\underline{ \tilde x}'} $ where   $t\tilde w= t' \tilde w'$. We consider the  sequence decompositions $\underline {\tilde w}= \underline{ \tilde w}_1 \underline{\tilde x}, \ \underline {\tilde w}'= \underline{ \tilde w}'_1 \underline{\tilde x}' (\tilde s'_n)$. Applying    Lemma \ref{inductions},  Example \ref{exc}, and Lemma \ref{xw}, we have
  $c_{\underline {\tilde w}}^{\underline{ \tilde x}}=c_{\underline {\tilde w}_1}^{( \ )} c_{\underline {\tilde x}}^{\underline{ \tilde x}}= c_{\tilde w_1} , \ c_{\underline {\tilde w}'}^{\underline{ \tilde x}}
 =c_{\underline {\tilde w}'_1}^{( \ )} c_{\underline {\tilde x}}^{\underline{ \tilde x}}(x\cdot c(\tilde s'_n))= c_{\tilde w'_1}(x\cdot c(\tilde s'_n)) =c_{\tilde w'_1 } (w'_1x\cdot c(\tilde s'_{n}))$. We have $w'_1x\cdot c(\tilde s'_{n})=c(\tilde w'_1\tilde x\cdot \tilde s'_{n})= t t'^{-1} c(\tilde s_1)$  because  $\tilde w'_1\tilde x \tilde s'_{n}=  \tilde w'  =t t'^{-1} \tilde w= t t'^{-1} \tilde s_1 \tilde w'_1\tilde x$. Therefore 
 $t'c_{\underline {\tilde w}'}^{\underline{ \tilde x}}=t   c( \tilde s_1)c_{\tilde w'_1 }=t  c_{\tilde w_1}= tc_{\underline {\tilde w}}^{\underline{ \tilde x}}$.
  
\bigskip {\bf D}) To end the proof we  reduce to case  {\bf A})  using  {\bf B}) and {\bf C}).  
Since the change of reduced expressions in $W$ is given by iteration of the braid relations, we may assume that there are   sequence decompositions $\underline {\tilde w}=\underline {\tilde  w }_1\underline {\tilde  w}_2\underline {\tilde  w}_3, \ \underline w'=\underline  {\tilde w}_1'\underline  {\tilde w}'_2\underline  {\tilde  w}_3'$ where  $\underline  {w}_2, \underline {w}'_2$ correspond to a braid relation $w_2=w'_2$ as in {\bf C}) and $\underline w_1=\underline w'_1,\underline  w_3=\underline w'_3$. 
Again by {\bf B}) we may assume without loss of generality that $\underline {\tilde  w }_1 = \underline {\tilde  w }'_1$, $\underline {\tilde  w }_3 = \underline {\tilde  w }'_3$, and that $\u {\tilde w}_2 = (\tilde s, \tilde s', \tilde s, \dots)$, $\u {\tilde w}'_2 = (\tilde s', \tilde s, \tilde s', \dots)$ for some $\tilde s, \tilde s' \in S^{\aff}(1)$.
We will reduce to case  {\bf A}) by extracting a subsequence $\underline{\tilde x}'' $ from $\underline {\tilde w}'$ such that  $b'\tilde x=  \tilde x''$ 
(for some $b' \in Z_k$) and $tb'^{-1} c_{\underline {\tilde w}}^{\underline{ \tilde x}}= t' c_{\underline {\tilde w}'}^{\underline{ \tilde x}''}$. 

From $t\tilde w=t'\tilde w'$ we deduce that $t=w_1\cdot a, t'= w_1\cdot a'$ for some  $a,a'\in Z_k$ such that $a \tilde w_2= a' \tilde w'_2$.
We have the compatible
 decomposition 
 $\underline {\tilde x}= \underline {\tilde  x}_1\, \underline {\tilde  x}_2\,  \underline {\tilde  x}_3  $.
Choose a subsequence $\underline {\tilde x}''_2$ of $\underline {\tilde w}'_2$ such that $b\tilde x_2={\tilde x}''_2$ (for some $b \in Z_k$), hence
$(x_1 \cdot b) \tilde x = \tilde x''$.
Then by  {\bf C}) we have $ab^{-1} c_{\underline {\tilde w}_2}^{\underline{ \tilde x}_2}=a' 
c_{\underline {\tilde w}'_2}^{\underline{ \tilde x}''_2} $. 
The sequence    $\underline{\tilde x}''=\underline{\tilde x}_1 \underline{ \tilde x}''_2\underline{\tilde x}_3 $ is a subsequence of $\underline {\tilde w}'$.
Applying Lemmas \ref{inductions} and \ref{xw}:
 $$c_{\underline {\tilde w}}^{\underline{ \tilde x}}=   c_{\underline {\tilde w}_1}^{ \underline{ \tilde x}_1} (x_1  \cdot c_{\underline {\tilde w}_2}^{\underline{ \tilde x}_2})\,
(x_1 x_2\cdot c_{\underline {\tilde w}_3}^{ \underline{ \tilde x}_3}) =  c_{\underline {\tilde w}_1}^{ \underline{ \tilde x}_1} (w_1  \cdot c_{\underline {\tilde w}_2}^{\underline{ \tilde x}_2}) \,( x_1 x_2\cdot c_{\underline {\tilde w}_3}^{ \underline{ \tilde x}_3}).  $$ 
We deduce that
$t(x_1 \cdot b)^{-1} c_{\underline {\tilde w}}^{\underline{ \tilde x}}= c_{\underline {\tilde w}_1}^{ \underline{ \tilde x}_1}( x_1  \cdot ab^{-1} 
c_{\underline {\tilde w}_2}^{\underline{ \tilde x}_2}) (x_1 x_2\cdot c_{\underline {\tilde w}_3}^{ \underline{ \tilde x}_3})=  c_{\underline {\tilde w}_1}^{ \underline{ \tilde x}_1}( x_1  \cdot a' 
c_{\underline {\tilde w}'_2}^{\underline{ \tilde x}''_2}) (x_1 x_2''\cdot c_{\underline {\tilde w}_3}^{ \underline{ \tilde x}_3})= t'c_{\underline {\tilde w}}^{\underline{ \tilde x}}$.
   \end{proof}

We denote $t u^{-1}c_{\underline {\tilde w}}^{\underline{ \tilde x}}=c_{t \tilde w}^{u \tilde x}$ in Proposition \ref{Waf1}. This defines  $c_{ w}^{  x}\in \mathbb Z[Z_k]$ for $  x, w\in   W^{\aff}(1) $ and  $  x \leq    w$. 

When $  x,\   w\in   W(1) $ satisfy $  x \leq    w$ there exists  $v \in \Omega(1)$ unique modulo $Z_k$ such that $xv , wv \in W^{\aff}(1)$ with  $xv\leq wv$ by definition of the Bruhat order (Definition \ref{Bru}). 
By Lemma~\ref{xw} the element 
$c_{ wv}^{  xv}$ does not depend on the choice of $v$ and we can define  $c_{ w}^{  x}=c_{ wv}^{  xv}$.

To summarize:
 
\begin{definition}\label{cxw}  
Let $  x,\   w\in   W(1) $ such that $  x \leq   w$.  
We define $c_{ w}^{  x}$ as 
$$c_{ w}^{  x}=    c_{  wv}^{ xv} = t\,  c_{\underline {wv}}^{\underline {txv}}\in \mathbb Z[Z_k] $$
where   $v\in \Omega(1)$,  $t \in Z_k$,  
$\underline {txv}  =(s_{i_1},\ldots, s_{i_r})$ is a subsequence of  $ \underline {wv}= (s_1,\ldots ,s_n)$ in $S^{\aff}(1)$ such that  $wv=s_1\cdots s_n$ and $txv=s_{i_1}\cdots s_{i_r}$ are reduced decompositions.  \end{definition}

\begin{proposition}\label{prop} %
The elements $c_{ w}^{  x}\in \mathbb Z[Z_k]$ for $x,w\in W(1), x\leq w$ satisfy the following properties:
\begin{enumerate}
\item\label{item:5} $c_{ w}^{w}=1$.%
\item $c_{twv}^{uxv} = tu^{-1} c_w^x$ for $t,u\in Z_k$, $v \in \Omega(1)$.
\item\label{item:2} $c_{v\cdot w}^{v\cdot x} = v \cdot c_w^x$ for $v \in \Omega(1)$.
\item\label{item:3} $c_w^x (x\cdot c) = c_w^x (w\cdot c)$ for $c \in \mathbb Z[Z_k]$.
\item\label{item:4} $c_{w_1 w_2}^{x_1 x_2}=c^{x_1}_{w_1} (x_1 \cdot c^{x_2}_{w_2})$ if $x_i, w_i\in W(1)$, $x_i \le w_i$, $\ell(x_1x_2)=\ell(x_1)+\ell(x_2), \ell(w_1w_2)=\ell(w_1)+\ell(w_2)$.
\item\label{item:1} $c_w^x=c^{xv}_{wv} $ if $ v\in W(1)$, $\ell(xv)=\ell(x)+\ell(v), \ell(wv)=\ell(w)+\ell(v).$
\item $c_w^1=c_w$ for $w \in W^{\aff}(1)$.
\item $c_{ w}^{ x}\in c_{  v}^{ x} \, \mathbb Z[Z_k] \quad \text{  for
$ x,v,w\in W(1)$ such that} \ x\leq   v \leq   w $.
\end{enumerate} 
 \end{proposition}

These properties come from the definition of $c_w^x$ and properties of the $c(s)$ ($s \in S^{\aff}(1)$), as well as Example~\ref{exc} and Lemma \ref{xw}.
Items~\ref{item:2}--\ref{item:4} are first proved for $x$, $w$, $x_i$, $w_i$ in $W^{\aff}(1)$ and then extended to $W(1)$.
Item~\ref{item:1} is a consequence of \ref{item:4} and \ref{item:5}.

\subsection{The Iwahori-Matsumoto expansion of  \texorpdfstring{$T^*_w$}{T*\_w} modulo \texorpdfstring{$q$}{q}}
\label{sec:iwah-mats-expans-1}
We compute  the triangular decomposition  of $T_{w}^*$ modulo $q$;
with the notation of   \eqref{T*}, we will prove  the congruence in $\mathbb Z[Z_k]$:   for $ x, w\in  W(1)$ and $  x\leq w$,
\begin{equation} \label{congr} c^*( w,x) \equiv 
 (-1)^{\ell(w)-\ell(x) }c_{ w}^{x} \ \mod q .\end{equation}
For $h,h'\in \mathcal H_{\mathbb Z}$, we write $h\equiv h' \mod q$ if   $h-h'\in q \mathcal H_{\mathbb Z} $. An equivalent formulation of the congruence is:
  
\begin{theorem}\label{*} Suppose that $\tilde w\in W (1)$ lifts $w\in W $. We have 
$$T_{\tilde w}^*\equiv  \sum_{x\in W, x\leq w} (-1)^{\ell(w)-\ell(x) } k^*_x  \ \mod q,\quad k^*_x= c_{\tilde w}^{\tilde x} T_{\tilde x} \ \text{ for any $\tilde x\in W (1)$ lifting $x $.}$$

\end{theorem}
 
 \begin{proof}  We assume $w\in W^{\aff}$. We can reduce to this case   because $c^*(wv,xv)=c^*(w,x),
  c_{wv}^{xv}=c_{w}^{x}$ for $x,w\in W^{\aff}(1), x\leq w, v \in \Omega(1)$ (Remark \ref{braid}, Proposition~\ref{prop}).  
  
One easily checks the theorem when $\ell(w)=0$ or $\ell(w)=1$.
   For $t\in Z_k$, $T_{t}^*=T_t$ and $c_{t}^{t}=1$.
For $s\in  S^{\aff}(1)$, $ T_{ s}^*=T_{ s}- c(s)$ and $c_s^s=1, c_s^1=c(s)$. 

In general we prove the theorem  by induction on $\ell(w)$.  Assume that $\ell(w)\geq 1$ and apply the braid relation to $\tilde w= \tilde w_1\tilde  s$  in 
 $  W^{\aff}(1)$ lifting $w=w_1s$ with $\ell(w)=\ell(w_1)+\ell(s)=\ell(w_1)+1$. By induction $T_{\tilde w}^*=T_{\tilde w_1}^* T_{\tilde  s}^* $ is congruent modulo $q$ to 
 $$  \sum_{x\leq w_1} (-1)^{\ell(w_1) -\ell(x) }  c_{\tilde w_1}^{\tilde x} T_{\tilde x} T_{\tilde  s}^*= \sum_{x\leq w_1} (-1)^{\ell(w) -\ell(x) }  c_{\tilde w_1}^{\tilde x} T_{\tilde x}c(\tilde s)+\sum_{x\leq w_1} (-1)^{\ell(w_1) -\ell(x) }  c_{\tilde w_1}^{\tilde x} T_{\tilde x}   T_{\tilde  s} .$$ 

The first sum  on the right-hand side equals
$$S_1=\sum_{x\leq w_1} (-1)^{\ell(w) -\ell(x) } c_{\tilde w}^{\tilde x} T_{\tilde x} $$
because  $T_{\tilde x}c(\tilde s)=(x\cdot c(\tilde s)) T_{\tilde x}$ and $c_{\tilde w_1}^{\tilde x} (x\cdot c(\tilde s)) =c_{\tilde w}^{\tilde x} $ by Proposition \ref{prop}.  
To analyze the  second sum $S_2$  on the right-hand side, as in \cite[IV.9]{MR3600042} we divide the set $\{ x\in W \ | \ x\leq w_1\} $ into the disjoint union $X \sqcup Y \sqcup Ys$ where
$$X=\{x\in W \ | \ x\leq w_1, xs \not\leq w_1\}, \ Y=\{x\in W \ | \ xs <x \leq w_1 \}.$$
We examine separately the contribution of $X$   and of $Y \sqcup Ys$. For $x\in X$ we have $x<xs$. The   contribution of $X$  in $S_2$ is 
\begin{align*}S_2(X)=\sum_{x\in X} (-1)^{\ell(w_1) -\ell(x) }  c_{\tilde w_1}^{\tilde x} T_{\tilde x}   T_{\tilde  s}=\sum_{x\in X} (-1)^{\ell(w_1) -\ell(x) }  c_{\tilde w_1}^{\tilde x} T_{\tilde x\tilde  s}=\sum_{x\in Xs} (-1)^{\ell(w) -\ell(x) }  c_{\tilde w \tilde s^{-1}   }^{\tilde x \tilde s^{-1}  } T_{\tilde x  } .\end{align*}
 For $x\in Xs$ we have  $xs<x$ hence $ c_{\tilde w \tilde s^{-1}   }^{\tilde x \tilde s^{-1}  }=c_{\tilde w    }^{\tilde x   }$ (Proposition   \ref{prop}). We have $Xs=\{ x\in W \ | \ x\leq w , x \not\leq w_1\}$ 
 \cite[IV.9 Lemma 2]{MR3600042}. Hence,   
 $$S_1+S_2(X)=\sum_{x\leq w} (-1)^{\ell(w)-\ell(x) }   c_{\tilde w}^{\tilde x} T_{\tilde x} .$$
  We now show that 
 the   contribution  of $Y \sqcup Ys$ in $S_2$ lies in $q \mathcal H_{\mathbb Z}$ (hence the theorem).  The contribution of  $Y \sqcup Ys$ is  
  $$S_2(Y \sqcup Ys)= \sum_{x\in Y} (-1)^{\ell(w_1) -\ell(x) }  (c_{\tilde w_1}^{\tilde x} T_{\tilde x}-  c_{\tilde w_1}^{\tilde x \tilde s} T_{\tilde x \tilde s}  )T_{\tilde  s} . $$
  We have  $c_{\tilde w_1}^{\tilde x \tilde s} =c_{\tilde w}^{\tilde x}= c_{\tilde w_1}^{\tilde x } (x\cdot c(\tilde s))=  c_{\tilde w_1}^{\tilde x } (xs\cdot c(\tilde s))$ by Proposition \ref{prop} and Lemma~\ref{s.s}, as $xs<x<w_1<w=w_1s$. Therefore $c_{\tilde w_1}^{\tilde x \tilde s} T_{\tilde x \tilde s} = c_{\tilde w_1}^{\tilde x } (xs\cdot c(\tilde s))T_{\tilde x \tilde s}= c_{\tilde w_1}^{\tilde x } T_{\tilde x \tilde s} c(\tilde s)$, and
 $$c_{\tilde w_1}^{\tilde x} T_{\tilde x}-  c_{\tilde w_1}^{\tilde x \tilde s} T_{\tilde x \tilde s} =c_{\tilde w_1}^{\tilde x} T_{\tilde x \tilde s}T_{ \tilde s} -   c_{\tilde w_1}^{\tilde x } T_{\tilde x \tilde s} c(\tilde s)=  c_{\tilde w_1}^{\tilde x} T_{\tilde x \tilde s} (T_{ \tilde s} - c(\tilde s))=c_{\tilde w_1}^{\tilde x} T_{\tilde x \tilde s} T_{ \tilde s}^* .$$
 As $T_{ \tilde s}^*  T_{\tilde  s}= q(s)  \tilde s^2$ and $q$ divides $q(s)$ we have $S_2(Y \sqcup Ys)\in q \mathcal H_{\mathbb Z}$.
 \end{proof}

  \subsection{The Iwahori-Matsumoto expansion of \texorpdfstring{$E_{o_J}(w)$}{E\_oJ(w)}}\label{sec:iwah-mats-expans}
 Let $J\subset \Delta$ and $P_J=M_JN_J$ the corresponding parabolic subgroup of $G$ containing $B$. The group $I\cap M_J$ is a pro-$p$-Iwahori subgroup of $M_J$ and we can apply to $M_J$ and  $I\cap M_J$ the theory of the pro-$p$ Iwahori Hecke algebra  given in the preceding sections for $G$ and $I$. We indicate with an index $J$ the objects associated to $M_J$ instead of $G$. 
 
 On the positive side: the root system $\Phi_J$ of $M_J$ is generated by $J$, the Weyl group $W_{J,0}=( {\cn}\cap M_J)/ Z$ of $M_J$ is   generated by the $s_\alpha$ for $\alpha \in J$,  the Iwahori Weyl group $W_J= ({\cn}\cap M_J)/ Z^0$ of $M_J$ is  a semidirect product
 $W_J=\Lambda \rtimes W_{J,0} $, 
  the  sets $\mathfrak S_J$ and $W_J^{\aff}$ are contained in $\mathfrak S$ and $W^{\aff}$, and we have the semidirect product $W_J=W_J^{\aff}\rtimes \Omega_J$ where $\Omega_J$ is the normalizer of $S_J^{\aff}$ in $W_J$. The pro-$p$ Iwahori Weyl group $W_J(1)=({\cn}\cap M_J)/ Z(1)$ of $M_J$ is the inverse image of $W_J$ in $W(1)$, ${}_1W^{\aff}_J$ is the inverse image of $W_J^{\aff}$ in $W(1)$ and 
 $W_J(1)=  {}_1W^{\aff}_J \,\Omega_J(1)$, where $\Omega_J(1)$ is the inverse image of $\Omega_J$ in $W(1)$. The  pro-$p$ Iwahori Hecke ring $\mathcal H_{J,\mathbb Z}$ of $M_J$ admits the bases $(T^J_w)_{w\in W_J(1)}$, $(T^{J,*}_w)_{w\in W_J(1)}$, $(E^J_{o}(w))_{w\in W_J(1)}$ for spherical orientations $o$ of $V_{J,\ad}$, and the integral Bernstein basis $(E^J(w))_{w\in W_J(1)}$.
 We have  $q^J(w)=q(w)$ for $w\in \mathfrak S_J$ and $ c^J(w)=c(w)$ for $w \in \mathfrak S_J(1) $ \cite[Thm.\ 2.21]{Vigneras-prop-IV}.

   On the negative side: the set $S_J^{\aff}$ of simple reflections     is not contained in $S^{\aff}$, the  length $\ell_J$ of $W_J$ is not the restriction of $\ell$, $\Omega_J$ is not contained in $\Omega$, the Bruhat order $\leq _J$ of $W_J^{\aff}$ is not the restriction of the Bruhat order $\leq $ of $W^{\aff}$,  the functions $w \mapsto q^J_{w}:W_J\to q^{\mathbb N},  (w_1,w_2) \mapsto q^J_{w_1,w_2}: W_J\times W_J\to q^{\mathbb N}, w \mapsto c_w^J: W_J(1)\to \mathbb Z[Z_k]$ are not the restrictions of the functions $w \mapsto q_w, (w_1,w_2) \mapsto q_{w_1, w_2}, w \mapsto c_w$ for $W$ and $W(1)$. 
   The linear injective map respecting the Iwahori-Matsumoto bases $$\iota_J:\mathcal H_{J,\mathbb Z}\to\mathcal H_\mathbb Z \quad T_w^J\to T_w$$  does not respect products. 
   
\begin{definition} \label{positive}    An element  $z\in Z$ is called \emph{$J$-positive} if $\langle \alpha, v(z)\rangle \geq 0$ for all $\alpha \in \Phi^+\setminus \Phi_J^+$. When  $z\in Z$ of  image $\lambda \in \Lambda$ is $J$-positive,   $ \lambda w \in W_J$ is  called \emph{$J$-positive} for all $w\in W_{J,0}$, and lifts of $ \lambda w$ in $W_J(1)$ are also called $J$-positive.   \end{definition}
 \begin{remark}\label{pos} 
$Z^+$ is the set of $z\in Z$ which are $J$-positive for all $J\subset \Delta$.

For $w_1, w_2 \in W_J(1), w_1\leq_J w_2$, if $w_2$ is $J$-positive the same is true for $w_1$ \cite[Lemma 4.1]{arXiv:1406.1003}.
\end{remark} 

\begin{notation} 
For $w\in W(1)$ or $W$, let $ n(w) \in \cn$ denote an element with image $w$; when $w\in W$ the image of $n(w)$ in $W(1)$ is a  lift  $\tilde n(w) $ of $w$. In particular, when $w\in W_0 = \cn^0/Z^0 \subset W$ we have $n(w) \in \cn^0$.  We do not require the lifts $n(w)\in \cn^0$ for $w\in W_0$ to satisfy the relations of \cite[IV.6 Proposition]{MR3600042}. The advantage is that this allows us to check compatibilities and to avoid some silly mistakes.
\end{notation}

 We have \cite[Thm.\ 1.4]{MR3437789}:
\begin{itemize}
\item The $\mathbb Z$-submodule of  $ \mathcal H_{J,\mathbb Z}$ with basis $T_w^J$ for the $J$-positive elements $w\in W_J(1)$  is a subalgebra  $ \mathcal H_{J,\mathbb Z}^+$ of $ \mathcal H_{J,\mathbb Z}$, called the \emph{$J$-positive subalgebra}.

\item $ \mathcal H_{J,\mathbb Z}$ is a localization of $ \mathcal H_{J,\mathbb Z}^+$. 

\item The  restriction of $\iota_J$ to  $\mathcal H_{J,\mathbb Z}^+$  respects products.  

\item Another basis of $ \mathcal H_{J,\mathbb Z}^+$  is $T^{J,*}_{w}$ for the $J$-positive elements $w\in W_J(1)$ (by the triangular decomposition \eqref{T*} and  Remark \ref{pos}).   

\item Similarly, for any spherical orientation $o$ of $V_{J,\ad}$, the elements $E_o^J(w)$ for the $J$-positive elements $w\in W_J(1)$ form a basis of $ \mathcal H_{J,\mathbb Z}^+$  (by the triangular decomposition \eqref{Et} and  Remark \ref{pos}).   
\end{itemize}          
 
Let $w_J$ denote the longest element of $W_{J,0}$.
For $z\in Z$, the integral Bernstein elements  
$E_{o^+}^J(z)=E_{o_J^+}^J(z) \in  \mathcal H_{J,\mathbb Z}$  associated to the  orientation $o_J^+$ of $V_{J,\ad}$ of  dominant Weyl chamber $\mathfrak D^+_J$ 
and $E_{o_J}(z)\in  \mathcal H_{\mathbb Z}$ associated to the orientation  $o_J$ of $V_{\ad}$ of Weyl chamber  
$\mathfrak D_{o_J}=w_J(\mathfrak D^-)  $  satisfy:

   \begin{lemma}  When $z\in Z$ is $J$-positive,  $\iota_J(E_{o^+}^J(z))= E_{o_J}(z)$.\end{lemma}
\begin{proof} The proof follows the arguments of 
\cite[Lemma 3.8]{MR3263136}, \cite[Lemma 4.6]{arXiv:1406.1003}, \cite[Prop.\ 2.19]{MR3437789}. 
 Let $z\in Z$.
 The element  $v(z)$ lies in the image  by $w_J$ of the dominant Weyl chamber $\mathfrak D^+$ of $V_{\ad}$  if and only if
 \begin{equation}\label{eq} \langle \alpha, v(z)\rangle \geq 0 \ \text{ for } \alpha \in w_J( \Phi^+)=(\Phi^+\setminus \Phi_J^+ )\cup \Phi_J^-. 
 \end{equation}
When $v(z)\in w_J(\mathfrak D^+) \Leftrightarrow \nu(z)=-v(z)\in  w_J(\mathfrak D^-)$ we have $\nu_J(z)\in \mathfrak D_J^+$ because 
 \[ \langle \alpha, v(z)\rangle \geq 0 \ \text{ for } \alpha \in   \Phi_J^- \Leftrightarrow \ \langle \alpha, \nu_J(z)\rangle \geq 0 \ \text{ for } \alpha \in   \Phi_J^+.  \]
Thus when $v(z) \in w_J(\mathfrak D^+)$ the integral Bernstein elements  
$E_{o^+}^J(z)=E_{o_J^+}^J(z) \in  \mathcal H_{J,\mathbb Z}$  and $E_{o_J}(z)\in  \mathcal H_{\mathbb Z}$   satisfy 
 \begin{equation}\label{eq2}  E_{o^+}^J(z) = T^{J}(z), \quad  E_{o_J}(z)= T(z), \quad \iota_J(E_{o^+}^J(z))= E_{o_J}(z).  
 \end{equation}
On the other hand, let $z, z_1, z_2 \in Z$ such that  $z=z_1 z_2^{-1}$ and  $\lambda_1, \lambda_2\in \Lambda$  the images of $z_1,z_2$.  For any orientation $o$ of $V_{\ad}$ (resp.\ $V_{J,\ad}$), we have in  $\mathcal H_{\mathbb Z}$  (resp.\ $\mathcal H_{J,\mathbb Z})$
 \begin{equation}\label{eq3}   E_o(z_1) q_{\lambda_2} =  q_{\lambda_1, \lambda_2^{-1}} E_o(z)E_o(z_2) \quad \text{(resp.\ } \ E^J_o(z_1) q^J_{\lambda_2} =  q^J_{\lambda_1, \lambda_2^{-1}} E^J_o(z)E^J_o(z_2)).
 \end{equation}
This follows from  the multiplication formula  in \S \ref{B3}  which gives  in  $\mathcal H_{\mathbb Z}$ 
\[ E_o(z_1) E_o(z_2^{-1}) =  q_{\lambda_1, \lambda_2^{-1}} E_o(z),\quad E_o(z_2)E_o(z_2^{-1})=q_{\lambda_2, \lambda_2^{-1}} = q_{\lambda_2} \]
and the analogous formula in $\mathcal H_{J,\mathbb Z}$.
    For $z\in Z$ general, we can find  $ z_1,z_2 $ as above such that $v(z_1), v(z_2)$ lie in $ w_J(\mathfrak D^+)$. For such elements we obtain from 
\eqref{eq2} and \eqref{eq3} that
 \begin{equation}\label{eq4} 
q_{\lambda_1, \lambda_2^{-1}} E_{o_J}(z)T(z_2)= q_{\lambda_2}T(z_1) ,\quad 
q_{\lambda_1, \lambda_2^{-1}}^J E_{o^+}^J(z)T^{J}(z_2)= q_{\lambda_2}^J T^{J}(z_1).
 \end{equation}

We now suppose that $z\in Z$ is $J$-positive.  
We choose $ z_1,z_2 \in Z$ such that $z = z_1 z_2^{-1}$ and $v(z_1), v(z_2)\in  w_J(\mathfrak D^+)$, in particular $z_1, z_2$ are $J$-positive.
 As $ E_{o+}^J(z)$ and $T^J(z_i)$ lie in $\mathcal H_{J,\mathbb Z}^+$, the algebra homomorphism $\iota_J: \mathcal H_{J,\mathbb Z}^+\to \mathcal H_{\mathbb Z}$ applied to the second  formula in \eqref{eq4}  gives
 \[ q^J_{\lambda_1, \lambda_2^{-1}} \iota_J(E_{o+}^J(z)) T(z_2)= q^J_{\lambda_2} T(z_1).\]
In  $\mathcal H_{\mathbb Q}$ where $T(z)$ is invertible   we have, using again~\eqref{eq4},
$$\iota_J(E_{o+}^J(z)) = (q^J_{\lambda_1, \lambda_2^{-1}})^{-1}q^J_{\lambda_2} T(z_1) T(z_2)^{-1}=  (q^J_{\lambda_1, \lambda_2^{-1}})^{-1}q^J_{\lambda_2} q _{\lambda_1, \lambda_2^{-1}} q_{\lambda_2}^{-1}E_{o_J}(z).$$
The coefficient of $T(z)$ in the   Iwahori-Matsumoto expansion of $\iota_J(E_{o+}^J(z))$ and of $E_{o_J}(z)$ being $1$, we deduce $ q^J_{\lambda_1, \lambda_2^{-1}}(q^J_{\lambda_2})^{-1}= q _{\lambda_1, \lambda_2^{-1}}q_{\lambda_2}^{-1}$ and $\iota_J(E_{o+}^J(z))= E_{o_J}(z)$ in $\mathcal H_{\mathbb Q}$ hence also in 
$\mathcal H_{\mathbb Z}$.  
 \end{proof}  

Suppose  $z\in Z^+$ with images $\tilde \lambda\in \Lambda^+(1), \lambda\in \Lambda^+$. We have $E^J_{o^+}(z)= T^{J,*}(z)$  and $z$ is $J$-positive hence $E_{o_J}(z)= \iota_J(T^{J,*}(z))$.  By the  triangular Iwahori-Matsumoto expansion of $T^{J,*}(z)$  \eqref{T*},  %
 \begin{equation}\label{Jz}  
E_{o_J}(z)= \sum_{x \in W_J, x \le_J w} c^{J,*}(\tilde \lambda,\tilde x) T(\tilde x).
\end{equation}
(In particular, by \eqref{Et}, $c_{o_J}(\tilde \lambda,\tilde x)=c^{J,*}(\tilde \lambda,\tilde x)$ for $\tilde x\in W_J(1)$  with $\tilde x\leq _J \tilde \lambda$.)
 For later use we need  the value of $E_{o_{J'}}(z n(w_J w_{J'})^{-1} )$ for $ J'\subset J \subset \Delta$. The computation  will use  \eqref{Jz} and the  following Lemma \ref{leJ} (whose proof uses Lemma  \ref{OJ}).    Recall the surjective map $\Phi\to \Phi_a$  \eqref{Phia} respecting positive roots.
 
  \begin{lemma}[{\cite[Lemma 2.9 ii]{MR3366919}}]\label{OJ} 
  Let $w\in W, \lambda\in \Lambda^+$  such that  $w\leq \lambda$. 
  Then there exists $\lambda_1\in  \Lambda^+$ such that  $\lambda_1 \le \lambda$ and $w\in W_0 \lambda_1 W_0$. 
  In particular, $\nu(\lambda_1)-\nu(\lambda) \in \sum_{\alpha \in \Delta} \mathbb Q_{\geq 0} \alpha^\vee$.
  \end{lemma}

  \begin{proof}
    Since our assumptions on $W$ are more general than in \cite{MR3366919} we give a brief sketch of the proof. 
    We have that $w \le w_\Delta \lambda$, the longest element of $W_0 \lambda W_0$.
    Choose $\lambda_1 \in \Lambda^+$ such that $w\in W_0 \lambda_1 W_0$. 
    Since $w_\Delta \lambda$, $w_\Delta \lambda_1$ are the longest elements of their double cosets, the lifting property of Coxeter groups \cite[Prop.\ 2.2.7]{bjorner-brenti} shows inductively that
    $w_\Delta \lambda_1 \le w_\Delta \lambda$, so $\lambda_1 \le w_\Delta \lambda$. By using the lifting property again we deduce that $\lambda_1 \le \lambda$.
    (We repeatedly use that $\ell(w \lambda) = \ell(w) + \ell(\lambda)$ for $w \in W_0$, $\lambda\in \Lambda^+$. This is a consequence of \eqref{length}.)
  \end{proof}
 
   \begin{lemma} \label{leJ} Let  $J'\subset J \subset \Delta$  and $ \lambda \in \Lambda$ such that  $\langle \alpha, v(\lambda) \rangle > 0$ for all $\alpha \in J\setminus J'$. 
     \begin{enumerate}\item
       For $\lambda_1 \in \Lambda^+$ such that $v(\lambda)- v(\lambda_1)\in \sum_{\beta\in J'}
       \mathbb Q_{\geq 0} \beta^\vee$, we have $\ \langle \gamma,v(\lambda_1 )\rangle >0$ for all
       $\gamma \in \Phi_{J}^+\setminus \Phi_{J'}^+$.
     \item 
       Suppose $\lambda \in \Lambda^+$ and $x\in W_{J'}$ with $x\leq_{J'} \lambda$.  Then
       $\ell(x)=\ell(x w_{J'} w_J)+\ell(w_{J'} w_J)$.
     \end{enumerate}
    \end{lemma}
    \begin{proof}  (i) For $\alpha \in J \setminus J'$ and $\beta \in J'$,  we have $\langle \alpha, \beta^\vee \rangle \leq 0$  hence 
$ \langle \alpha,   v(\lambda )\rangle \leq  \langle \alpha,   v(\lambda_1 )\rangle$. Let 
$\gamma  \in \Phi^+_{J} \setminus \Phi _{J'}^+$. There exists $\alpha \in J\setminus J'$ such that $\gamma-\alpha $ is a sum of roots in $\Phi^+$. Since $\lambda_1\in \Lambda^+$, 
$\langle \gamma-\alpha ,v(\lambda_1 )\rangle \geq 0$ hence $\langle  \alpha,v(\lambda_1 )\rangle \leq  \langle \gamma,v(\lambda_1 )\rangle  $ and $\langle  \alpha,v(\lambda )\rangle \leq  \langle \gamma,v(\lambda_1 )\rangle  $. 
Hence $\langle \gamma, v(\lambda_1)\rangle >0$ for $\gamma \in \Phi_{J}^+\setminus  \Phi_{J'}^+$. 

(ii) There exists  $\lambda_1\in  \Lambda^{+,J'}$ such that $x\in W_{J',0} \lambda_1 W_{J',0}$ and  $v(\lambda)- v(\lambda_1)\in \bigoplus_{\beta\in J'} \mathbb Q_{\geq 0} \beta^\vee$ (Lemma \ref{OJ},  $v=-\nu$). In particular, $0 \le \langle \alpha,v(\lambda)\rangle \le \langle \alpha,v(\lambda_1)\rangle$ for $\alpha \in J \setminus J'$, hence
$\lambda_1 \in \Lambda^+$.
We write $x=\lambda_x v_x$ with $\lambda_x= v_1\cdot \lambda_1\in \Lambda$ and $ v_1,v_x\in W_{J',0}$. 

As $\Phi_{J}^+\setminus  \Phi_{J'}^+$ is stable by $W_{J',0}$ and  $\langle \gamma, v(\lambda_1)\rangle >0$ for $\gamma \in \Phi_{J}^+\setminus  \Phi_{J'}^+$ by (i) we have  
$$
\langle \gamma, v(\lambda_x)\rangle >0 \quad \text{for} \ \gamma \in \Phi_{J}^+\setminus  \Phi_{J'}^+.
$$
 By the length formula \eqref{length2},  $\ell(x w_{J'} w_J)=\ell (\lambda_x v_xw_{J'} w_J)$  is equal to
$$\ell(x w_{J'} w_J)=\ell(\lambda_x)-\ell(v_xw_{J'} w_J)+2 |\{\alpha \in \Phi^+_a  \cap v_xw_{J'}w_J(\Phi_a^-),  \langle \alpha_a, v(\lambda_x)\rangle \leq 0\}| .$$
As $v_x\in W_{J',0}$ we have $\ell(v_x w_{J'} w_J)=\ell(w_J)-\ell(v_x w_{J'})=\ell(w_J)-\ell( w_{J'})+\ell(v_x)=\ell(v_x)+\ell( w_{J'} w_J)$. 
Hence 
$\ell(\lambda_x)-\ell(v_xw_{J'} w_J)=\ell(\lambda_x)-\ell(v_x ) -\ell( w_{J'} w_J)$.  We have   

  \noindent $\Phi_a^+ \cap v_xw_{J'}w_J(\Phi_a^-)=\Phi_a^+ \cap[(\Phi_a^-\setminus \Phi_{a,J}^-)\cup ( \Phi_{a,J}^+\setminus \Phi_{a,J'}^+)\cup v_x(\Phi^-_{a,J'} )]=( \Phi_{a,J}^+\setminus \Phi^+_{a,J'})\cup (\Phi_a^+  \cap v_x (\Phi_a^-))$,
and $\langle \alpha_a , v(\lambda_x)\rangle >0$ for $\alpha_a \in \Phi_{a,J}^+\setminus \Phi^+_{a,J'}$. Hence

 \noindent $\ell(x w_{J'} w_J)+\ell(w_{J'} w_J)=\ell(\lambda_x)-\ell(v_x) + 2 |\{\alpha_a \in \Phi^+_a  \cap v_x (\Phi_a^-),  \langle \alpha_a, v(\lambda_x)\rangle \leq 0\}|=\ell(x)$.
  \end{proof} 
 
 \begin{proposition}\label{EJ} For $J'\subset J \subset \Delta$ and $z\in Z^+$ of image $\tilde \lambda\in \Lambda^+(1)$ and $\lambda\in \Lambda^+$ such that $\langle \alpha, v(\lambda) \rangle > 0$ for all $\alpha \in J\setminus J'$,
 $$E_{o_{J'}}(z n(w_J w_{J'})^{-1})=   \sum_{x\in W_{J'}, x\leq _{J'}\lambda } c ^{J',*}(\tilde \lambda,\tilde x) T (\tilde x n(w_J w_{J'})^{-1}) $$
 for any lifts $\tilde x\in W_{J'}(1)$ of $x \in W_{J'}$.
  \end{proposition}
   \begin{proof} We have $\ell(\lambda)=\ell(\lambda  w_{J'} w_J)+\ell(w_{J'} w_J)$ by  Lemma \ref{leJ}, and the multiplication formula in \S \ref{B3} gives 
 $$E_{o_{J'}}(z)= E_{o_{J'}}(z n(w_J w_{J'})^{-1}) E_{o_{J'}\cdot w_{J'}w_{J}}( n(w_J w_{J'})).$$
The orientation  $o_{J'}\cdot w_{J'}w_{J}$ of Weyl chamber $w_Jw_{J'} (\mathcal D_{o_{J'}})=w_J( \mathcal D^-)= \mathcal D_{o_{J}}$ is $o_J$ and $E_{o_J}( n(w_J w_{J'}))= T(n(w_J w_{J'}))$ \cite[Example 5.32]{MR3484112}, so
$$E_{o_{J'}}(z)= E_{o_{J'}}(z n(w_J w_{J'})^{-1})T(n(w_J w_{J'})).$$ Applying \eqref{Jz} and  Lemma \ref{leJ}
$$E_{o_{J'}}(z n(w_J w_{J'})^{-1})T(n(w_J w_{J'}))=   \sum_{x\in W_{J'}, x\leq _{J'}\lambda } c ^{J',*}(\tilde \lambda,\tilde x) T(\tilde x n(w_J w_{J'})^{-1})T(n(w_J w_{J'})).
$$
In $\mathcal H_{\mathbb Q}$, the basis element $T(n(w_J w_{J'}))$ is invertible and we deduce
$$E_{o_{J'}}(z n(w_J w_{J'})^{-1})=   \sum_{x\in W_{J'}, x\leq _{J'}\lambda } c ^{J',*}(\tilde \lambda,\tilde x) T(\tilde x n(w_J w_{J'})^{-1}).$$
This remains true in $\mathcal H_{\mathbb Z}$.
  \end{proof}
 \begin{remark} Comparing with  \eqref{T*}, \eqref{Et},   Proposition \ref{EJ} implies
 $$c_{o_{J'}}(\tilde \lambda n(w_J w_{J'})^{-1}, \tilde x n(w_J w_{J'})^{-1}
 )=c^{J',*}(\tilde \lambda,\tilde x) $$
for $J'\subset J\subset \Delta$  and $\tilde \lambda, \tilde x\in W(1)$  lifting $\lambda\in \Lambda^+,x\in W_{J'}, x\leq _{J'} \lambda$.
\end{remark} 

 \subsection{\texorpdfstring{$\psi(c(s))$}{psi(c(s))} for a simple affine reflection }\label{sec:psi-cs}

Let $\psi: Z^0\to C^\times$ be a character. It is trivial on $Z^0\cap M'_ {\Delta'_{\psi}}$  (Definition  \ref{delta'psi}) by the following lemma.
 
\begin{lemma} For  $J\subset \Delta$, the group $Z^0\cap M'_J$ is generated by  
$Z^0\cap M'_\alpha$ for $\alpha \in J$.
\end{lemma}  
 \begin{proof}  Let   $\langle \cup_{\alpha \in J}Z^0\cap M'_\alpha \rangle$ denote the group generated by the  $Z^0\cap M'_\alpha$ for $\alpha \in J$. This group is contained in $Z^0\cap M'_J$ and  $Z^0\cap M'_J$ is contained in the kernel of $\nu$.   The group $Z\cap M'_J$ is generated by  
$Z\cap M'_\alpha$ for $\alpha \in J$ \cite[II.6 Prop.]{MR3600042}\ and the group $Z\cap M'_\alpha$ is generated by $Z^0\cap M'_\alpha$ and $a_\alpha$ (Definition \ref{delta'psi}) \cite[\S III.16]{MR3600042}. The group $Z$ normalizes $M'_\alpha$ and $Z^0$ hence 
$$Z\cap M'_J =  \langle \cup_{\alpha \in J}Z^0\cap M'_\alpha \rangle \prod _{\alpha\in J}a_\alpha^{\mathbb Z}.$$
The group $Z^0 $ is contained in the kernel of $\nu$ and  $\nu(a_\alpha)=\alpha_a^\vee$. The $\alpha_a^\vee$ for $\alpha \in J$ are linearly independent, hence an identity $\sum_{\alpha \in J} n(\alpha)\alpha_a^\vee=0$ with $n(\alpha)\in \mathbb Z$ implies $n(\alpha)=0$ for all $\alpha\in J$. We get
$  Z\cap M'_J \cap \Ker \nu =  \langle \cup_{\alpha \in J}Z^0\cap M'_\alpha \rangle,$ hence $Z^0\cap M'_J$ is contained in $\langle \cup_{\alpha \in J}Z^0\cap M'_\alpha \rangle$.  \end{proof} 
 
  As in \S \ref{Notation},  $\overline {Z^0\cap M'_J}$ denotes the image of $Z^0\cap M'_J$ in $Z_k^{\aff}$.

\begin{remark}\label{a-1} For $\alpha \in \Delta$, the group  $\,\overline {Z^0\cap M'_\alpha}\,$  is different from the group   $Z_{k,s_{\alpha}}$ defined in Remark \ref{cs}. 
The group $\,\overline {Z^0\cap M'_\alpha}\,$   is generated by $Z_{k,s_\alpha}$ and another  group $Z_{k,s_{\alpha_a-1}} $ such that for an admissible lift $\tilde s_{\alpha_a -1}$ of $s_{\alpha_a -1}$ the value  $c( \tilde s_{\alpha_a -1}) \in \mathcal H_C$ is given by a formula like  \eqref{cp}  for $ c(\tilde s_\alpha)$ with $Z_{k,s_{\alpha_a-1}}$ instead of $Z_{k,s_\alpha}$  \cite[IV.24 Claim, IV.25--28]{MR3600042}. The group  $\,\overline {Z^0\cap M'_\alpha}\,$   is also generated by $Z_{k,s_{\alpha}}$ and  $s_\alpha(Z_{k,s_{\alpha_a-1}}) $ because  $\,\overline {Z^0\cap M'_\alpha}\,$   and $Z_{k,s_{\alpha}}$ are  normalized by $s_\alpha$. The set $\Delta'_\psi $ (Definition \ref{delta'psi}) is therefore contained in the set
 \begin{equation}\label{deltapsi}
 \Delta(\psi):=\{ \alpha \in \Delta \ | \ \psi \ \text{is trivial on } Z_{k,s_\alpha}\}.
 \end{equation}
 \end{remark}

\begin{lemma} \label{vani}\

  \begin{enumerate}\item 
    Let $J\subset \Delta$ and $\tilde \tau \in {}_1\mathfrak
    S_J$. %
    Then $c(\tilde \tau)\in \mathbb Z\,[ \,\overline {Z^0\cap M'_J}\,]$. When $J\subset
    \Delta'_\psi$, we have $\psi(c(\tilde \tau))=-1$.
  \item Let $\alpha \in \Delta \setminus \Delta'_\psi$. Then $\psi (c(\tilde s_\alpha) \, c( \tilde
    s_{\alpha_a-1} ))=\psi (c(\tilde s_\alpha)\, (s_\alpha \cdot c( \tilde s_{\alpha_a-1} )))= 0$.
  \end{enumerate}
\end{lemma}
 
\begin{proof} (i) This  follows  from Remark \ref{cs} applied to the Levi subgroup $M_J$ of $G$. (Recall that $c^J(w) = c(w)$.)

  (ii)   By hypothesis $\psi$ is not trivial on  the  image of $Z^0\cap M'_\alpha$ in $Z_k^{\aff}$, hence if $\psi$ is trivial on  $Z_{k,s_\alpha}$, then  $\psi$  is not  trivial on  $Z_{k,s_{\alpha_a-1}} $ and on $s_\alpha(Z_{k,s_{\alpha_a-1}}) $. By formula \eqref{cp} and Remark \ref{a-1},   $\psi (c(\tilde s_\alpha))=0$ (resp.\   $\psi  (c( \tilde s_{\alpha_a -1}))=0 $,  resp.\   $\psi  (s_\alpha\cdot c( \tilde s_{\alpha_a -1}))=0 $) if and only if $\psi$ is not trivial on $Z_{k,s_\alpha}$ (resp.\ $Z_{k,s_{\alpha_a-1}}$, resp.\ $s_\alpha(Z_{k,s_{\alpha_a-1}}) $).
 \end{proof}

 \subsection{\texorpdfstring{$\psi(c_{w}^{x})$}{psi(c\_w\^{}x)} for dominant translations}\label{sec:psi-cwx}

Let $\psi: Z^0\to C^\times$ be a character and  $\tilde x, \tilde w\in W (1)$ lifting 
 $ x,    w \in \Lambda^+$ such that $\tilde x \leq \tilde w $. To compute $ \psi(c_{\tilde  w}^{\tilde  x})$  we need some knowledge of the 
 reduced expressions of the elements of $\Lambda^+$.  This is obtained in the following lemmas.

 \begin{lemma}\label{lalpha} Let $\alpha \in \Delta$, $\lambda\in \Lambda^+$ such that $\lambda_\alpha \lambda\in \Lambda^+$ and let $\lambda=s_1\cdots s_{n} u$ with $s_i\in S^{\aff},u\in \Omega$  be  a reduced expression. 
  Then there exist $k_1<k_2$ such that 
  \begin{itemize}\item 
    $\lambda_\alpha\lambda=s_1\cdots s_{k_1-1} s_{k_1+1} \cdots s_{k_2-1} s_{k_2+1} \cdots s_{n} u$ is a reduced expression,
    and
  \item  
    $\{ (s_1\cdots s_{k_1-1})\cdot s_{k_1}, \ (s_1\cdots s_{k_1-1} s_{k_1+1} \cdots s_{k_2-1})\cdot s_{k_2}\} = \{ s_\alpha,
    s_\alpha \lambda_\alpha\} \ \text{or} \ \{ s_\alpha, \lambda_\alpha s_\alpha \}.$
  \end{itemize}
 \end{lemma}
 \begin{proof} As in Lemma \ref{firststep} we have 
 \[\lambda_\alpha \lambda  < s_{\alpha}\lambda_\alpha \lambda < \lambda \] 
because $\ell (s_{\alpha}\lambda_\alpha \lambda)=\ell(\lambda^{-1} \lambda_\alpha^{-1} s_\alpha) = \ell(\lambda_\alpha \lambda)+1
 = \ell( \lambda)-1$ (using \eqref{length2}), and we have
  $s_{\alpha} \lambda_\alpha =  s_{\alpha_a +1}\in \mathfrak S$.
 By the strong exchange condition %
 there exists $i$  such that 
$s_{\alpha}  \lambda_\alpha s_1\cdots s_i =   s_1\cdots s_{i-1} $ and there exists $j$ such that either of the following hold:
 
(1) $j<i$,   $ s_{\alpha} s_1\cdots s_j=s_1\cdots s_{j-1}$: hence $(s_1\cdots s_{j-1}) \cdot s_j = s_\alpha$ and $(s_1\cdots s_{j-1} s_{j+1}\cdots s_{i-1})\cdot s_i= (s_\alpha s_1\cdots s_{i-1})\cdot s_i= s_\alpha\cdot s_{\alpha}  \lambda_\alpha = \lambda_\alpha s_\alpha$; we take $k_1=j,k_2=i$.
  
(2) $j>i$,   $s_{\alpha}s_1\cdots s_{i-1} s_{i+1}\cdots s_j= s_1\cdots s_{i-1} s_{i+1}\cdots s_{j-1}$:  hence $(s_1\cdots s_{i-1}) \cdot s_i = s_\alpha \lambda_\alpha$  and 
$(s_1\cdots s_{i-1} s_{i+1}\cdots s_{j-1})\cdot s_j=s_\alpha$; we take
$k_1=i, k_2=j$.
 \end{proof}
 \begin{remark} \label{lalphal} We will apply Lemma \ref{lalpha} as follows.  For a choice of lifts in $W(1)$, we have  $c_{\tilde \lambda }^{ \tilde \lambda \tilde \lambda_\alpha}= t(s_1\cdots s_{k_1-1}\cdot c(\tilde s_{k_1})) \,  (s_1\cdots s_{k_1-1}s_{k_1+1} \cdots s_{k_2-1}\cdot c(\tilde s_{k_2}))$ for some $t\in Z_k$, by definition of $c^x_w$. Hence, as $ \lambda_\alpha   s_{\alpha}=  s_{\alpha_a -1}, 
  s_{\alpha} \lambda_\alpha =  s_\alpha s_{\alpha_a -1}s_\alpha$, we have
 $$c_{\tilde \lambda }^{ \tilde \lambda \tilde \lambda_\alpha} \in
       c(\tilde s_\alpha)\, (s_\alpha\cdot c( \tilde s_{\alpha_a -1}))\mathbb Z[Z_k] \ \text{or} \ c(\tilde s_\alpha)c( \tilde s_{\alpha_a -1})\mathbb Z[Z_k].$$
    \end{remark}
 
By iteration of the lemma, we get:
 \begin{lemma} \label{alphan} Let  $ \lambda\in \Lambda^+, J\subset \Delta,n(\alpha)\in \mathbb N$ for $\alpha \in J$ such that  $\lambda \prod_{\alpha \in J}\lambda_\alpha  ^{m(\alpha)}\in \Lambda^+$ for all $m(\alpha)\in \mathbb N, m(\alpha)\leq n(\alpha)$, and  let $\lambda=s_1\cdots s_{n} u$ with $s_i\in S^{\aff},u\in \Omega$  be a reduced expression. 
    Then there exist $1\leq i_1<i_2 <\cdots <i_r\leq n$  such that 
    \begin{itemize}\item 
      $\lambda \prod_{\alpha \in \Delta}\lambda_\alpha ^{n(\alpha)}=s_{i_1}\cdots s_{i_r}u$ is a reduced expression, and
    \item 
      $(s_{i_1}\cdots s_{i_j})\cdot s_k $ lies in $W_J^{\aff} \subset W^{\aff}$ for any $0\leq j \leq r$ and $i_j<k <
      i_{j+1}$.
    \end{itemize}
 \end{lemma}
Here we let $i_0 = 0$, $i_{r+1} = n+1$.
  \begin{proof} We proceed by induction on $\sum_{\beta \in J} n(\beta)$.
  Let $\alpha \in J$ such that $n(\alpha)>0$.  Then $\lambda_1=   \lambda \prod_{\beta \in J}\lambda_\beta  ^{n(\beta)}=\lambda_2 \lambda_\alpha$ and $\lambda_2 \in \Lambda^+$. By the inductive hypothesis, there exist $i_1<i_2 <\cdots <i_r$  such that    $\lambda_2  =s_{i_1}\cdots s_{i_r}u$ is a reduced expression and 
    $(s_{i_1}\cdots s_{i_j})\cdot s_k $ lies in $W_J^{\aff}$ for any $0\leq j \leq r$ and $i_j<k < i_{j+1}$. From Lemma \ref{lalpha} there exist $a<b$  such that 
 $\lambda_1=  s_{i_1}\cdots s_{i_{a-1}}s_{i_{a+1}} \cdots  s_{i_{b-1}}s_{i_{b+1}} \cdots s_{i_r}u$ is a reduced expression   and  $\tau_1= (s_{i_1}\cdots s_{i_{a-1}})\cdot s_{i_a}, \ \tau_2= (s_{i_1}\cdots s_{i_{a-1}}s_{i_{a+1}} \cdots  s_{i_{b-1}})\cdot s_{i_b}$ are in $W_J^{\aff}$.  
 We prove that $(i'_1, \ldots, i'_{r-2})=(i_1,\ldots  , i_{a-1},i_{a+1}, \ldots ,  i_{b-1},i_{b+1}, \ldots ,i_r)$ satisfies the conditions of the lemma. Take  $0\leq j \leq r-2$ and $i'_j<k<i'_{j+1}$. Then  $(s_{i'_1}\cdots s_{i'_j})\cdot s_k $ lies in $W_J^{\aff}$. Indeed, if $k=i_a$ or $i_b$ this is the condition on $a$ and $b$.  Otherwise, take $j'$ such that $i_{j'}< k < i_{j'+1}$. Then
 $$
 (s_{i'_1}\cdots s_{i'_j})\cdot s_k = \begin{cases}
 (s_{i_1}\cdots s_{i_{j'}})\cdot s_k & \ \text{if} \ j'<a \ \text{(hence $j = j'$)},\\
  (\tau_1s_{i_1}\cdots s_{i_{j'}})\cdot s_k & \ \text{if} \ a\le j'<b \ \text{(hence $j = j'-1$)},\\
  (\tau_2 \tau_1s_{i_1}\cdots s_{i_{j'}})\cdot s_k  & \ \text{if} \ b\le j'  \ \text{(hence $j = j'-2$)}.
 \end{cases}
 $$
 In any case, this is in $W_J^{\aff}$ by the inductive hypothesis and because $\tau_1,\tau_2$ are in $W_J^{\aff}$.
 \end{proof}

\begin{remark}\label{nalphan}We will apply Lemma \ref{alphan} as follows. 
Keep the notation of the lemma, so $i_j < k < i_{j+1}$.
Let  $\alpha_{k}\in \Phi$ be a reduced root such that $s_k$ is the reflection in an affine hyperplane of the form $\alpha_k + r = 0$ ($r\in \mathbb R$).
We have $s_{i_1}\cdots s_{i_j} (\alpha_k)\in \Phi_J$, where $\Phi_J\subset \Phi$ denotes the root subsystem generated by $J$.  
Choose lifts  $\tilde s_{i_1},\ldots, \tilde s_{i_j}, \tilde s_k\in {}_1W^{\aff}$ of $s_{i_1},\ldots, s_{i_j}, s_k$ with $\tilde s_k$ admissible.
Writing $M_\beta' = \langle U_\beta, U_{-\beta} \rangle$ for any reduced root $\beta \in \Phi$, we have that $\tilde s_k$ lies in the image
of $\cn \cap M_{\alpha_k}'$ in $W(1)$. It follows that $\tilde s_{i_1}\cdots \tilde s_{i_j}\cdot \tilde s_k$ lies in the image of
$\cn \cap M_{s_{i_1}\cdots s_{i_j}(\alpha_k)}'$ in $W(1)$, so $\tilde s_{i_1}\cdots \tilde s_{i_j}\cdot \tilde s_k \in {}_1W_J^{\aff} \cap \mathfrak S(1) = {}_1 \mathfrak S_J$.
Hence by Lemma \ref{vani} we see that $s_{i_1}\cdots s_{i_j}\cdot c(\tilde s_k)=c(\tilde s_{i_1}\cdots \tilde s_{i_j}\cdot \tilde s_k) $ lies in $\mathbb Z[\, \overline{Z^0\cap M'_J}\,]$. Therefore $\psi(s_{i_1}\cdots s_{i_j}\cdot c(\tilde s_k))=-1$ if  $\psi $ is trivial on $Z^0\cap M'_J$. %
 \end{remark}

      We are now ready to compute $ \psi(c_{\tilde  w}^{\tilde  x})$ when $\tilde x, \tilde w$ are elements of the inverse image $\Lambda^+(1)$ of $\Lambda^+$ in $W(1)$.       
  \begin{theorem} \label{psic}   Let $\tilde x, \tilde w\in  \Lambda^+(1)$ lifting 
 $ x,    w \in \Lambda^+$ such that $ x \leq   w  $. Then 
  $$ \psi(c_{\tilde  w}^{\tilde  x})=\begin{cases} (-1)^{\ell( w)- \ell( x)} \ &\text{if $\tilde x\in \tilde w\prod _{\alpha \in \Delta'_{\psi}}a_\alpha^{\mathbb N}$,}\\
  0 \ &\text{if $x\not\in w\prod _{\alpha \in \Delta'_{\psi}}\lambda_\alpha^{\mathbb N}$}.
  \end{cases}$$
  \end{theorem} 
   
  \begin{proof}  
  We have   $ x=   w \prod_{\alpha \in \Delta}\lambda_\alpha^{n(\alpha)}$ with $n(\alpha)\in \mathbb N$ (Proposition \ref{L+order}).  For   $\tilde \lambda\in \Lambda^+(1)$, $c_{\tilde w \tilde \lambda}^{\tilde x \tilde \lambda}=c_{ \tilde w}^{\tilde x}$  (Proposition \ref{prop}), so
by Lemma 3.5 we may assume without loss of generality that $w \prod_{\alpha \in \Delta}\lambda_\alpha^{m(\alpha)} \in \Lambda^+$ for any $0 \le m(\alpha) \le n(\alpha)$.

Assume  $n(\alpha)>0$ for  some $\alpha \in \Delta \setminus \Delta'_\psi$. Let $\tilde w' = \tilde x \tilde\lambda_\alpha^{-1}$ for some lift $\tilde\lambda_\alpha$ of $\lambda_\alpha$, so $\tilde x = \tilde  w' \tilde\lambda_\alpha\le \tilde w' \le \tilde w$.
Then $c_{\tilde  w}^{\tilde  x} \in c_{\tilde  w'}^{\tilde  w' \tilde\lambda_\alpha}\mathbb Z[Z_k]$ by Proposition \ref{prop},
so $c_{\tilde  w}^{\tilde  x} \in c(\tilde s_\alpha)\, (s_\alpha\cdot c( \tilde s_{\alpha_a -1})) \mathbb Z[Z_k]$ or  $ c(\tilde s_\alpha)c( \tilde s_{\alpha_a -1})\mathbb Z[Z_k]$ by Remark~\ref{lalphal}.
Therefore $\psi(c_{\tilde  w}^{\tilde  x})=0$ by Lemma \ref{vani}.

Assume now $ n(\alpha)=0$ for all $\alpha \in \Delta \setminus \Delta'_\psi$ and that $\tilde x\in \tilde w\prod _{\alpha \in \Delta'_{\psi}}a_\alpha^{n(\alpha)}$. Take a reduced expression $\tilde w=\tilde s_1\cdots  \tilde s_n \tilde u$ where  $ \tilde s_1,\ldots,  \tilde s_n\in {}_1S^{\aff}$ are admissible and $\tilde  u\in \Omega(1)$.   
Let $J = \Delta'_\psi$.
By Lemma \ref{alphan} and Remark \ref{nalphan}, there exist $i_1<i_2 <\cdots <i_r$  such that 
\begin{itemize}\item 
  $x=w \prod_{\alpha \in \Delta}\lambda_\alpha ^{n(\alpha)}=s_{i_1}\cdots s_{i_r}u$ is a reduced expression,
\item $\tilde s_{i_1}\cdots \tilde s_{i_j}\cdot \tilde s_k \in {}_1 \mathfrak S_J$ for any $0\leq j \leq r$ and $i_j<k < i_{j+1}$, and
\item 
  $s_{i_1}\cdots s_{i_j}\cdot c(\tilde s_k)=c(\tilde s_{i_1}\cdots \tilde s_{i_j}\cdot \tilde s_k)\in \mathbb Z[\,
  \overline{Z^0\cap M'_J}\,]$ and $\psi (s_{i_1}\cdots s_{i_j}\cdot c(\tilde s_k))=-1$ for any $0\leq j \leq r$ and $i_j<k < i_{j+1}$.
\end{itemize}
We have $\tilde x = t \tilde s_{i_1}\cdots \tilde s_{i_r}u$ for some $t \in Z_k$. 
Taking the product of all $\tilde s_{i_1}\cdots \tilde s_{i_j}\cdot \tilde s_k \in {}_1 \mathfrak S_J$ we deduce that
$(\tilde w u^{-1})(t^{-1} \tilde x u^{-1})^{-1} = \tilde w\tilde x^{-1} t \in {}_1 W^{\aff}_J$. Since
$\tilde x^{-1} \tilde w = \prod _{\alpha \in J}a_\alpha^{-n(\alpha)} \in {}_1 W^{\aff}_J$, it follows by normality that $\tilde w \tilde x^{-1} \in {}_1 W^{\aff}_J$.
Thus $t \in Z_k \cap {}_1 W^{\aff}_J = Z_k^{\aff,J}$, so $\psi(t) = 1$. Therefore, from the definition of $c_{\tilde w}^{\tilde x}$ we get that
$\psi (c_{\tilde w}^{\tilde x})=(-1)^{n-r}$. \end{proof}

\section{Inverse Satake theorem when \texorpdfstring{$\Delta(V')\subset \Delta(V)$}{Delta(V') subset Delta(V)}}\label{ProofS}
\subsection{Value of \texorpdfstring{$\varphi_z$}{phi\_z} on a generator}\label{sec:value-phi-z} Let $V,V'$ be two irreducible representations of $K$ with parameters 
$(\psi_V, \Delta(V)), (\psi_{V'}, \Delta(V'))$ such that $\Delta(V')\subset \Delta(V)$, let 
 $\iota^{\op}:V^{U^0_{\op}}\congto V'^{U_{\op}^0}, \iota: V_{U^0}\congto V'_{U^0}$ be compatible linear isomorphisms  \eqref{iopi}, and let  \eqref{triangle}
$$z\in Z_G^+(V,V')=\{z\in Z^+ \ | \ z\cdot\psi_V=\psi_{V'}, \langle \alpha, v(z) \rangle > 0 \ \text{for all } \ \alpha \in \Delta(V)\setminus \Delta(V')\}.$$
The Satake transform
$S^G:\mathcal H_G(V,V')\to \mathcal H_Z(V_{U^0},V'_{U^0})$ 
is injective (cf.\ Definition \ref{SMGVV'}). 
After showing   that $\tau_z^{V_{U^0},V'_{U^0},\iota}$ belongs to the image of $S^G$ we will compute the value of  the unique  antecedent $\varphi_z$ on a generator of the representation $\ind_K^GV$ of $G$ (Proposition \ref{begin}). 
As a generator we take the function $f_v\in \ind_K^GV$ of support $K$ and value   at $1$  a non-zero element $v\in V^{U_{\op}^0}$. This generator $f_v$ is fixed by the pro-$p$ Iwahori group $I=K(1)U_{\op}^0$ and its image by a $G$-intertwiner $\ind_K^GV \to \ind_K^GV'$ is also fixed by $I$. The space $(\ind_K^GV')^I$ of $I$-invariants of $\ind_K^GV'$ is a right module  for the pro-$p$ Iwahori Hecke $C$-algebra $\mathcal H_C$. We will show that $\varphi_z(f_v)=f_{v'} h_{z}$ where  $f_{v'}\in \ind_K^GV'$ has support $K$ and value $v'=\iota^{\op}(v)$ at $1$, and $h_{z}\in \mathcal H_C$; then, we will describe $h_{z}$ using the elements $T^*_w$ and $E_{o_{\Delta (V')}}(w)$ of $\mathcal H_C$ for $w\in W(1)$. 

\begin{proposition}\label{begin} 
Suppose $z \in Z_G^+(V,V')$. There exists  $\varphi_z\in \mathcal H_G(V,V')$ such that $S^G(\varphi_z)=\tau_z^{V_{U^0},V'_{U^0},\iota}$. The value of  $\varphi_z$ on $f_v$ is $f_{v'} h_{z}$ where $$h_{z}=E_{o_{\Delta (V')}}(z n(w_{\Delta (V)}w_{\Delta (V')})^{-1})  T^*(n (w_{\Delta (V)}w_{\Delta (V')})).$$
\end{proposition}

Note that $E_{o_J}( z n(w)^{-1})T^*(n(w))$ does not depend on the choice of the lift $n(w)\in \cn$ of $w\in W$ because another choice differs only by multiplication by 
  $t\in Z^0$ and for $n,n'\in \cn$,  
$E_{o_J}( nt^{-1}) T^*(tn')=E_{o_J}( n)T(t^{-1}) T(t)T^*(n')= E_{o_J}( n) T^*(n')$.

\subsection{Embedding in \texorpdfstring{$\mathfrak X= \Ind_B^G (\ind_{Z(1)}^Z 1_C)$}{X = ind\_B\^{}G (ind\_\{Z(1)\}\^{}Z 1\_C)}}\label{S:5.2}
Proposition \ref{begin} is essentially the same as Theorem  \cite[IV.19 Thm.]{MR3600042}\ which implies the easier part of the change of weight theorem \cite[IV.I Thm.\ (i)]{MR3600042}. 
(See the end of \S\ref{S:5.2} for an explanation why it is essentially the same.)
The first step of the proof is to embed the two representations $\ind_K^GV$ and  $\ind_K^GV'$  of $G$ in the same representation 
\[\mathfrak X= \Ind_B^G (\ind_{Z(1)}^Z 1_C).\]
For a $C$-character $\psi$ of $Z^0$ let $e_\psi\in  \ind_{Z(1)}^Z 1_C$ denote the function of support $Z^0$ and equal to $\psi$ on $Z^0$. For $v\in V^{U_{\op}^0}\setminus\{0\}$ of image $\overline v\in V_{U^0}$,  let  $f_v\in \ind_K^GV$ (resp.\  $e_{\overline v}\in \ind_{Z^0}^ZV_{U^0}$) denote the function of support $K$ with $f_v(1)=v $  (resp.\ of support $Z^0$ with $e_{\overline v}(1)=\overline v $). We  recall the  injective intertwiner  \cite[Def.\ 2.1]{MR3001801}
$$I_V:\ind_K^GV\into \Ind_B^G (\ind_{Z^0}^ZV_{U^0})$$ such that  $I_V(f_v)(1)=e_{\overline v}$. We have the injective $Z$-intertwiner  
$$j_{\overline v}:\ind_{Z^0}^ZV_{U^0}\into \ind_{Z(1)}^Z 1_C$$ sending 
$e_{\overline v}$ to $e_{\psi_V}$.
 
\begin{definition} For $v\in V^{U_{\op}^0}\setminus\{0\}$, let 
$I_v:\ind_K^GV\into \mathfrak X$  be the injective $G$-equivariant map
such that  $I_v(f_v)(1)=e_{\psi_V}$.
 \end{definition}

The intertwiner $I_v$ is the composite of $I_V$ and    the injective  $G$-intertwiner $$\Ind_B^G(j_{\overline v}) :\Ind_B^G (\ind_{Z^0}^ZV_{U^0})\into \mathfrak X$$ induced by $ j_{\overline v}$. 
For $\varphi\in \mathcal H_G(V,V')$, the diagram
$$\xymatrix{\ind_K^GV\ar[r]^-{I_V}  \ar[d]_\varphi & \Ind_B^G (\ind_{Z^0}^ZV_{U^0}) \ar[d]^{S^G(\varphi)}\\
\ind_K^GV' \ar[r]_-{I_{V'}} &\Ind_B^G (\ind_{Z^0}^ZV'_{U^0})
}$$
is commutative \cite[\S 2]{MR3001801}.  For $z\in Z$, let $\tau(z)$ be the characteristic function of $z Z(1)$ seen as a $Z$-intertwiner $\ind_{Z(1)}^Z 1_C\to\ind_{Z(1)}^Z 1_C$. 
This makes $\ind_{Z(1)}^Z 1_C$ into a left $C[Z/Z(1)]$-module.
Let  $\overline v'=\iota(\overline v)$.
The diagram
$$\xymatrix{\ind_{Z^0}^ZV_{U^0} \ar[r]^{j_{\overline v}  }\ar[d]_{ \tau_z^{V_{U^0},V'_{U^0},\iota}}& \ind_{Z(1)}^Z 1_C\ar[d]^{\tau(z)}\\
\ind_{Z^0}^ZV'_{U^0}  \ar[r]_{j_{\overline v'}} &\ind_{Z(1)}^Z 1_C
}$$
is commutative. By functoriality, the diagram
$$\xymatrix@C=1.5cm{\Ind_B^G (\ind_{Z^0}^ZV_{U^0}) \ar[r]^-{\Ind_B^G(j_{\overline v} )}\ar[d]_{ \tau_z^{V_{U^0},V'_{U^0},\iota} }&  \mathfrak X\ar[d]^{\tau(z)}\\
\Ind_B^G(\ind_{Z^0}^ZV'_{U^0}) \ar[r]_-{\Ind_B^G (j_{\overline v'})} &   \mathfrak X}$$
is also commutative.  

\begin{proposition}\label{end2} 
Suppose $z \in Z_G^+(V,V')$. In the $(C[Z/Z(1)],\mathcal{H}_C)$-bimodule $\mathfrak X^I$ we have
$$\tau(z) I_v(f_v)= I_{v'}(f_{v'}) h_{z}, \  h_{z}=E_{o_{\Delta (V')}}(z n(w_{\Delta (V)}w_{\Delta (V')})^{-1})  T^*(n (w_{\Delta (V)}w_{\Delta (V')})).$$
\end{proposition}

This proposition implies Proposition \ref{begin}, as we now explain: we see in particular that $\tau(z) I_v(f_v) \in I_{v'}(\ind_K^G V')$,
so $\tau(z) I_v(\ind_K^G V) \in I_{v'}(\ind_K^G V')$. Thus there exists a unique $\varphi_z \in \mathcal H_G(V,V')$ such that the following diagram commutes:
$$\xymatrix{\ind_K^GV \ar[r]^-{I_v}  \ar@{-->}[d]_{\varphi_z} & \mathfrak X \ar[d]^{\tau(z)}\\
\ind_K^GV' \ar[r]_-{I_{v'}} &\mathfrak X.
}$$
By the above discussion and injectivity of $\Ind_B^G (j_{\overline v'})$ we deduce that $\tau_z^{V_{U^0},V'_{U^0},\iota} \circ I_V = I_{V'} \circ \varphi_z$.
We also have $S^G(\varphi_z) \circ I_V = I_{V'} \circ \varphi_z$. From the discussion of \cite[\S2]{MR3001801} it follows that $S^G(\varphi_z) = \tau_z^{V_{U^0},V'_{U^0},\iota}$
(both correspond to the map $I_{V'} \circ \varphi_z$ under the adjunction \cite[(2)]{MR3001801}, where we take $P = B$ and $W = \ind_{Z^0}^ZV'_{U^0}$).

Proposition~\ref{end2} is a variant of \cite[IV.19 Theorem]{MR3600042}. In loc.\ cit.\   one assumes  $\psi_V=\psi_{V'}=\psi$,  $\Delta(V)=\Delta(V')\sqcup\{\alpha\}$ and the representation $\mathfrak X$ of $G$  is replaced by $\mathfrak X_\psi=\Ind_B^G (\ind_{Z^0}^Z\psi)$. 
Identifying $V_{U^0} \simeq \psi_V$, $V'_{U^0} \simeq \psi_{V'}$ via our bases $\overline v, \overline v'$ we have the embeddings
$\Ind_B^G(j_{\overline v})  : \mathfrak X_{\psi_V} \into \mathfrak X$, $\Ind_B^G(j_{\overline v'}) : \mathfrak X_{\psi_{V'}} \into \mathfrak X$.
We need to   explain why certain arguments of \cite{MR3600042}  remain valid  or can be adapted to our more general setting.

\subsection{Proof in \texorpdfstring{$\mathfrak X^I$}{X\^{}I}}\label{sec:proof-in-X-I}
We  start the proof of  Proposition \ref{end2}. 
For $n(w)\in \cn^0$ lifting $w\in W_0$, the double coset $Bn(w)I$ does not depend on the choice of $n(w)$; we write $BwI=Bn(w)I$.

\begin{definition}For a $C$-character $\psi$ of $Z^0$, the function $f_{\psi,n(w_\Delta)}\in \mathfrak X^I$ has support $Bw_\Delta I$ and its value at $n(w_\Delta)^{-1}$ is $e_\psi$.
\end{definition}

The function $f_{\psi,n(w_\Delta)}$ is the image of the function $f_0 \in \mathfrak X^I_\psi$ of  \cite[IV.7 Definition]{MR3600042}  for a fixed choice of $n(w_\Delta)$.  As announced earlier, we first show $I_v(f_v)\in f_{\psi_V,n(w_\Delta)} \mathcal H_C$. 

\begin{lemma}\label{red} We have  $I_v(f_v)=f_{\psi_V,n(w_\Delta)} T(n(w_\Delta)n(w_{\Delta(V)})^{-1})T^*(n(w_{\Delta(V)}))$.
\end{lemma}
\begin{proof} This is obtained from \cite[IV.9 Proposition]{MR3600042} by applying the embedding $\mathfrak X_\psi \into \mathfrak X$, for a certain choice of $n(w_\Delta)$ and $n(w_{\Delta(V)})$. This is valid for any choice because for $t\in Z^0$,  the  product  $T(n t^{-1})T^*(tn')$ for $n,n'\in \cn$ does not depend on $t$, and neither does $f_{\psi_V,tn(w_\Delta)} T(tn )  = t f_{\psi_V,n(w_\Delta)} T(t)T(n)$,
recalling  
\begin{equation}\label{action}fh = \sum_{x\in I \backslash G} h(x) x^{-1} f \quad \text{for} \ h\in \mathcal H_C, \ f\in \mathfrak X^I,
\end{equation}
hence $fT(t)=t^{-1} f$.
\end{proof}

\begin{lemma}\label{recall} For a $C$-character $\psi$ of $Z^0$  and $z\in Z^+$ we have
$$\tau(z) f_{\psi,n(w_\Delta)}= f_{z\cdot\psi,n(w_\Delta)}T(n(w_\Delta)\cdot z).$$
\end{lemma}
\begin{proof} %
When $z\cdot\psi = \psi$ this is obtained from \cite[IV.10 Proposition]{MR3600042} by applying the embedding $\mathfrak X_\psi \into \mathfrak X$.
By loc.\ cit.,  the support of $f_{z\cdot\psi,n(w_\Delta)}T(n(w_\Delta)\cdot z)$ is $Bw_\Delta I$  and its value at $n(w_\Delta)^{-1}$ is
$f_{z\cdot\psi,n(w_\Delta)}( n(w_\Delta)^{-1} (n(w_\Delta)\cdot z^{-1}))=f_{z\cdot\psi, n(w_\Delta)}(z^{-1}n(w_\Delta)^{-1})= z^{-1}f_{z\cdot\psi, n(w_\Delta)}(n(w_\Delta)^{-1})= z^{-1} e_{z \cdot \psi} = \tau(z) e_\psi$.  Therefore $\tau(z) f_{\psi,n(w_\Delta)}= f_{z\cdot\psi,n(w_\Delta)}T(n(w_\Delta)\cdot z)$.
\end{proof}
Lemmas \ref{red} and \ref{recall} imply   
$$\tau(z) I_v(f_v)=f_{z\cdot\psi_V,n(w_\Delta)}T(n(w_\Delta)\cdot z)T(n(w_\Delta)n(w_{\Delta(V)})^{-1})T^*(n(w_{\Delta(V)})).$$
We want to show that the right-hand side is equal to 
$$I_{v'}(f_{v'})E_{o_{\Delta (V')}}(z n(w_{\Delta (V)}w_{\Delta (V')})^{-1})  T^*(n (w_{\Delta (V)}w_{\Delta (V')})).$$
\noindent This is a problem entirely in (the image in $\mathfrak X^I$ of) the $\mathcal H_C$-module $\mathfrak X_{\psi_{V'}}^I$ which is solved implicitly by  \cite[IV.19 Theorem]{MR3600042} for a special choice of lifts in $\cn^0$ of 
$w_\Delta,w_{\Delta(V)},w_{\Delta(V')}$ and when $\psi_V=\psi_{V'}, \Delta(V)=\Delta(V')\sqcup\{\alpha\}$. Checking the homogeneity, the choice of the lifts does not matter, but the hypothesis on the parameters of $V$ and of $V'$   forces us to analyze the proof of \cite[IV.19 Theorem]{MR3600042}. The sets $\Delta (V)$ and $\Delta (V')$ appear together only when the proof uses \cite[IV.19 Lemma]{MR3600042}. But this lemma is valid when $\Delta(V)$ is any subset   of $\Delta$ containing $\Delta(V')$. With our notation this lemma is:

\begin{lemma} For $\Delta(V')\subset J \subset \Delta$ we have 

$I_{v'}(f_{v'})= f_{z\cdot\psi_{V},n(w_\Delta)} T(n(w_\Delta)n(w_{J})^{-1})T^*(n(w_{J}) n(w_J w_{\Delta(V')})^{-1})T(n(w_J w_{\Delta(V')}))$.
\end{lemma}

We now consider the characters. The equality  $\psi_V=\psi_{V'}$ appears only when the proof uses \cite[IV.14 Theorem]{MR3600042} for $w=1$, but we can replace it by:

\begin{lemma}\label{w1} For a $C$-character $\psi$ of $Z^0$, $J\subset \Delta$  and $z\in Z$ we have
$$
 f_{z\cdot\psi,n(w_\Delta)}T(n(w_\Delta)n(w_{J})^{-1}) E_{o_J} (n(w_J)\cdot z)=
 \begin{cases} \tau (z)  f_{\psi,n(w_\Delta)}T(n(w_\Delta)n(w_{J})^{-1})  \ & \text{if} \ z\in Z^+\\
 0  \ & \text{if} \ z\not\in Z^+.
 \end{cases}
$$
\end{lemma}
\begin{proof} The  formula of Lemma \ref{recall} multiplied on the right by $T(n(w_\Delta)n(w_{J})^{-1}) $ is
\begin{align*}\tau(z) f_{\psi,n(w_\Delta)}T(n(w_\Delta)n(w_{J})^{-1})&= f_{z\cdot\psi,n(w_\Delta)}T(n(w_\Delta)\cdot z)T(n(w_\Delta)n(w_{J})^{-1}).
\end{align*} Suppose $z\in Z^+$.  In the pro-$p$ Iwahori Hecke algebra,
\[T(n(w_\Delta)\cdot z)T(n(w_\Delta)n(w_{J})^{-1}) = T(n(w_\Delta)n(w_{J})^{-1}) E_{o_J} (n(w_J)\cdot z).
\]
This follows from \cite[IV.15]{MR3600042} applied to $n(w_J)\cdot z$ instead of $\lambda$ and to $n(w_\Delta)n(w_{J})^{-1}$ instead of $n_{w^J}$  and $n(w_J)^{-1}$ instead of $\nu_{w_J}$. We get the formula of the lemma for $z\in Z^+$.   

Suppose now  $z\not\in Z^+$. As in   \cite[IV.15]{MR3600042} we take  $z_1\in Z^+$ such that $\langle \alpha,v_Z(z_1)\rangle > 0$ for any $\alpha\in\Phi^+$ and we multiply on the right by $E_{o_J} (n(w_J)\cdot z)$ the formula  that we just established for   $z_1\in Z^+$. Using   $E_{o_J} (n(w_J)\cdot z_1)E_{o_J} (n(w_J)\cdot z)=0$  we deduce 
$$0=\tau (z_1)  f_{\psi,n(w_\Delta)}T(n(w_\Delta)n(w_{J})^{-1})E_{o_J} (n(w_J)\cdot z),$$
and then we multiply on the left by the inverse $\tau(z_1^{-1}) $ of $ \tau (z_1)$ in $C[Z/Z(1)]$.  The result is valid for any $\psi$  and we replace $\psi$ by $z\cdot\psi$ to  get the   lemma for $z\not\in Z^+$. \end{proof}
By induction on $\ell(w)$  for $w\in W_{J,0}$, Lemma \ref{w1} is a particular case of a more general result,  as explained in \cite[IV.16--18]{MR3600042}
(again we see that the choice of representatives $n(w)$ for $w \in W_0$ is irrelevant):
 
\begin{lemma} For a $C$-character $\psi$ of $Z^0$, $J\subset \Delta$,  $z\in Z$ and $w\in W_{J,0}$, we have
\begin{align*}
 &f_{z\cdot\psi,n(w_\Delta)}T(n(w_\Delta)n(w_{J})^{-1})T^*(n(w)) E_{o_J} (n(w)^{-1}n(w_J) \cdot z)\\
& =
 \begin{cases} \tau (z)  f_{\psi,n(w_\Delta)}T(n(w_\Delta)n(w_{J})^{-1})T^*(n(w))  \ & \text{if} \ z\in Z^+\\
 0  \ & \text{if} \ z\not\in Z^+.
 \end{cases}
\end{align*}
\end{lemma}
Now applying the proof of  \cite[IV.19 Theorem]{MR3600042} we get Proposition \ref{end2}.
(Note that we still get $\ell(z n(w_{\Delta(V)}w_{\Delta(V')})^{-1}) = \ell(n(w_{\Delta(V)}w_{\Delta(V')})\cdot z) - \ell(n(w_{\Delta(V)}w_{\Delta(V')}))$, as $z \in Z_G^+(V,V')$.)
This ends the proof of Proposition \ref{begin}.

\subsection{Expansion of \texorpdfstring{$\varphi_z$}{phi\_z} in the basis \texorpdfstring{$(T_x)$}{(T\_x)} of \texorpdfstring{$\mathcal H_G(V,V')$}{H\_G(V,V')}}\label{sec:expansion-phi-z}
We now give the expansion  in the basis $(T_z^{V,V',\iota})_{z\in Z_G^+(V,V')/Z^0}$ of  $\mathcal H_G(V,V')$   (Proposition \ref{basisVV'}) of the function $\varphi_z $ given  in Proposition \ref{begin} by its value on a generator $f_v$ of $\ind_K^G V$:
\begin{equation}
\varphi_z (f_v)= f_{v'}  E_{o_{\Delta (V')}}(z n(w_{\Delta (V)}w_{\Delta (V')})^{-1})  T^*(n (w_{\Delta (V)}w_{\Delta (V')})).\label{eq:3}
\end{equation}

Recall that $Z_z^+(V,V') = Z^+\cap z\prod_{\alpha \in  \Delta'(V')}a_\alpha^\mathbb N$ is finite and contained in $Z_G^+(V,V')$ (Lemma~\ref{lm:Zz-contained-in-ZG}).
 
\begin{proposition} \label{expa}Let $z\in Z_G^+(V,V')$. The function $\varphi_z\in  \mathcal H_G(V,V')$  is equal to $$\sum_{x\in Z_z^+(V,V')} T_{x}^{V,V',\iota} .$$
\end{proposition}  

Clearly Propositions \ref{begin}  and \ref{expa} imply Theorem \ref{VcontV'}. 
 
\begin{proof} 
  Two elements $\varphi_1,\varphi_2\in  \mathcal H_G(V,V')$ such that $\varphi_1(f_v)|_{Z^+}=\varphi_2(f_v)|_{Z^+}$ are equal.  This follows from two properties: 
  \begin{enumerate}\item 
    a basis of $ \mathcal H_G(V,V')$ is $T_{z'}^{V,V',\iota}$ for $z'$ running through a system of representatives of
    $Z^+_G(V,V')/Z^0$.  So $\varphi_1= \sum _{z'} a_1(z') T_{z'}^{V,V',\iota} $ for some $a_1(z') \in C$.
  \item 
    $\varphi_1(f_v)(z')=a_1(z') v'$ for $z'\in Z^+_G(V,V')$ because of the lemma below.
  \end{enumerate}
\begin{lemma}   For $z' \in Z^+_G(V,V')$ the  function $T_{z'} ^{V,V',\iota} (f_v)\in \ind_K^G V'$ vanishes outside $Kz'K$ and is equal   to $v'$ at $z'$.
\end{lemma}
\begin{proof}For  $y\in G$, the value of $T_{z'} ^{V,V',\iota} (f_v)$ at $y$, 
$$T_{z'} ^{V,V',\iota} (f_v)(y)=\sum_{g\in Kz'K/K} T_{z'} ^{V,V',\iota}(g) (f_v(g^{-1}y))$$
is $0$ if $Kz'^{-1}Ky\cap K=\varnothing$ (hence  $T_{z'} ^{V,V',\iota} (f_v)$ vanishes outside $Kz'K$) and  $T_{z'} ^{V,V',\iota} (f_v)(z')= T_{z'} ^{V,V',\iota}(z') (f_v(1))= \iota^{\op}(v)=v'$.
\end{proof}
 Therefore it is enough to prove  that  $\varphi_z (f_v)|_{Z^+}=\sum_{x\in Z_z^+(V,V')} T_{x}^{V,V',\iota} (f_v)|_{Z^+}$, or equivalently,
 \begin{equation}\label{Z++}\varphi_z (f_v)(x)  =\begin{cases} v' \ & x \in Z_z^+(V,V'), \\
0 \ & x \in Z^+ \setminus Z^0 Z_z^+(V,V').
\end{cases}
\end{equation}

We now write  $J'=\Delta(V')$ and $J=\Delta(V)$. We  prove \eqref{Z++}   in two steps. In the first step we prove  \eqref{Z++}  assuming two claims which are proved in the second step.

{\bf A})  By the congruence  modulo $q$ of the Iwahori-Matsumoto expansion of   $E_{o_{J'}}(z n(w_Jw_{J'})^{-1}) $    (Propositions \ref{*} and \ref{EJ}),  we have \begin{align*}f_{v'} E_{o_{J'}}(z n(w_Jw_{J'})^{-1})  
=   \sum_{x\in W_{J'}, x\leq _{J'}\lambda } (-1)^{\ell_{J'}(\lambda)-\ell_{J'}(x)}
 \psi_{V'}^{-1}(c_{\tilde \lambda}^{\tilde x,J'}) \, f_{v'} T(\tilde x n(w_J w_{J'})^{-1}),
\end{align*}
where $\tilde \lambda$ is the image of $z$ in $\Lambda^+(1)$ and $\lambda$ the image of $z$ in $\Lambda^+$.
We used that  $f_{v'} c=\psi_{V'}^{-1}(c) f_{v'}$ for $c\in \mathbb Z[Z_k]$, as $f_{v'}T(t) = t^{-1}f_{v'}= \psi_{V'}(t^{-1})f_{v'}$ for $t\in Z_k$   \eqref{action}. We claim that
\begin{equation}\label{claim1}
  f_{v'} T(\tilde x n(w_J w_{J'})^{-1}) T^*(n (w_{J}w_{J'}))|_{Z^+}\neq   0 \ \Longrightarrow  x\in \Lambda ^+.
\end{equation}
Now for $x \in \Lambda^+$ we have  $x\leq_{J'}\lambda$ if and only if $x\in \Lambda^+\cap \lambda \prod_{\alpha \in  J'}\lambda_\alpha^{\mathbb N}$  (Proposition \ref{L+order}), and we know  the value of $\psi_{V'}^{-1}(c_{\tilde \lambda}^{\tilde x,J'})$ (Theorem \ref{psic}). Obviously  $\Delta'_{\psi_{V'}}= \Delta'_{\psi_{V'}^{-1}}$ and $J'\cap \Delta'_{\psi_{V'}}= \Delta'(V')$ hence
$x\in \Lambda^+\cap \lambda \prod_{\alpha \in  \Delta'(V')}\lambda_\alpha^{\mathbb N}$  (Proposition \ref{L+order}) if $\psi_{V'}^{-1}(c_{\tilde \lambda}^{\tilde x,J'})\neq 0$. 
Together with~\eqref{eq:3} we obtain
$$\varphi_z (f_v)|_{Z^+}=\sum_{  \tilde x\in \Lambda^+(1)\cap \tilde \lambda \prod_{\alpha \in  \Delta'(V')}a_\alpha^{\mathbb N}}f_{v'}T( \tilde xn(w_J w_{J'})^{-1}) T^*(n (w_{J}w_{J'}))|_{Z^+}.
$$ 
We claim  also that 
\begin{equation}\label{claim2}
  f_{v'}T( \tilde xn(w_J w_{J'})^{-1}) T^*(n (w_{J}w_{J'}))|_{Z^+}=f_{v'}T( \tilde xn(w_J w_{J'})^{-1}) T(n (w_{J}w_{J'}))|_{Z^+}.
\end{equation}
Assuming the claim, the braid relations and $\ell(x)=\ell(xw_{J'}w_J)+\ell(w_J w_{J'})$ (Lemma \ref{leJ}) imply
$$\varphi_z (f_v)|_{Z^+}=\sum_{\tilde x\in \Lambda^+(1)\cap \tilde \lambda \prod_{\alpha \in  \Delta'(V')}a_\alpha^{\mathbb N}}f_{v'}T( \tilde x)|_{Z^+}.$$
We finally compute $f_{v'}T( \tilde x)|_{Z^+}$.
\begin{lemma} For $z\in Z$, the function $f_{v'}T( z ) \in (\ind_K^G V')^I$ vanishes on $Z^+$ if  $z\not\in Z^+$, and $f_{v'}T( z )$ is the function of support $KzI$ with value $v'$ at  $z $ if $z\in Z^+$.
\end{lemma}
\begin{proof}
  The map $z\mapsto KzI:Z\to K\backslash G/I$  factors to a bijective map  $\Lambda\xrightarrow{\sim}K\backslash G/I$.  
    We have $KzI\cap Z^+= zZ^0$ if $z\in Z^+$ and $KzI\cap Z^+=\varnothing$ if $z\in Z\setminus Z^+$ and 
\begin{equation*}\label{action2}(f_{v'}T(z )) (z)= \sum_{x\in I \backslash I zI} f_{v'}(z x^{-1}).
\end{equation*}
  The support of $f_{v'}T( z ) $ is contained in $KzI $ hence $f_{v'}T( z ) \in (\ind_K^G V')^I$ vanishes on $Z^+$  if  $z\not\in Z^+$. In the displayed formula   $f_{v'}(zx^{-1})\neq 0$ implies  $z x^{-1}\in K\cap z I z^{-1}I$. 
Consider the Iwahori decomposition  $I= U_{\op}^0( I\cap B) $.
If $z\in Z^+$, we have $U_{\op}^0 \subset z U_{\op}^0 z^{-1} \subset U_{\op}$ and $z( I\cap B)z^{-1} \subset I \cap B$.
By intersecting with $K$ we get $U_{\op}^0 = K \cap z U_{\op}^0 z^{-1}$.
Hence $K\cap z I z^{-1}I=K\cap z U_{\op}^0 z^{-1}I = I$, so $(f_{v'}T(z )) (z)=f_{v'}(1)=v'$. 
\end{proof}

{\bf B}) We prove the two claims \eqref{claim1} and \eqref{claim2}. There are weak braid relations in $\mathcal H_C$ valid for any  pair of elements in $W(1)$.

\begin{lemma}  For $w_1,w_2\in W(1)$ there exists $w'_2\in W(1)$ with $w'_2\leq w_2$ and $T_{w_1} T_{w_2}\in C[Z_k] T_{w_1 w'_2}$.
\end{lemma}
\begin{proof} This is done by induction on $\ell(w_2)$. When $ \tilde s\in S ^{\aff} (1)$ we have $T_{w_1} T_{\tilde s}=  T_{w_1 \tilde s}$ if $w_1< w_1\tilde s$  and  $T_{w_1} T_{\tilde s}= T_{w_1 \tilde s^{-1}} T_{\tilde s}^2 = T_{w_1 \tilde s^{-1}} c(\tilde s) T_{\tilde s} = (w_1\cdot c(\tilde s)) T_{w_1} $  if  $ w_1\tilde s<  w_1$.
\end{proof}
As an application,  for $\tilde w_1, \tilde w_2\in W(1)$ lifting $w_1, w_2 \in W$, the triangular Iwahori-Matsumoto expansion of $T_{\tilde  w_2}^*$ and the weak braid relations imply
\begin{equation*} T_{\tilde w_1}(T_{\tilde w_2}^*-T_{\tilde w_2})\in \sum_{y\in W, y<w_2}C[Z_k] T_{\tilde w_1}T_{\tilde y}  \subset \sum_{y\in W, y<w_2}C[Z_k] T_{\tilde w_1 \tilde y  },
\end{equation*}
where $\tilde y\in W(1)$ lifts $y$. 
We use this result as follows:
$f_{v'}T_{\tilde w_1}T_{\tilde w_2}^*|_{Z^+} =f_{v'} T_{\tilde w_1}T_{\tilde w_2}|_{Z^+}$ if $f_{v'}T_{\tilde w_1 \tilde y  }|_{Z^+}=0$ for all 
 $y\in W$ with $y<w_2$.
The two claims \eqref{claim1} and \eqref{claim2}  follow from:

\begin{lemma}    Suppose  $\tilde w_1 \in W(1)$ lifts $w_1=x w_{J'} w_J$ with $x\in W_{J'}, x\leq _{J'}\lambda$, $ \lambda\in \Lambda^+,$ and $\tilde y\in W(1)$ lifts $y\in W_{J,0}$ with $y\leq w_{J}w_{ J'} $. Then $f_{v'}T_{\tilde w_1 \tilde y }$  vanishes on $Z^+$ except if $x\in \Lambda^+$ and $y=w_Jw_{J'}$. 
\end{lemma}
 \begin{proof}   Let $\lambda_x\in \Lambda$ and $v_x\in W_{J',0}$ such that $x=\lambda_x v_x$.  We have $\langle \gamma, v(\lambda_x )\rangle >0$ for $\gamma \in \Phi_J^+\setminus \Phi_{J'}^+$ by the proof of Lemma \ref{leJ}(ii).
 
We have $w_1y = \lambda_x v_x w_{J'} w_J y$ where     $v_xw_{J'}w_J y\in W_{J,0}$,  the support of $ f_{v'}T_{\tilde w_1 \tilde y } $ is contained in $K n(\lambda_x) n( v_x w_{J'}w_J y) I= K (n(v_xw_{J'}w_J y)^{-1}\cdot n(\lambda_x)) I$
and recalling the bijection  $\Lambda\to K\backslash G/I$, we have
$Z \cap  K (n(v_xw_{J'}w_J y)^{-1}\cdot n(\lambda_x)) I=
  Z^0 (n(v_xw_{J'}w_J y)^{-1}\cdot n(\lambda_x))$.
We have  $ \langle (v_xw_{J'}w_J y)^{-1}(\gamma), v((v_xw_{J'}w_J y)^{-1}\cdot \lambda_x)\rangle=\langle \gamma , v( \lambda_x)\rangle $. 
  If $v_xw_{J'}w_J y\not\in W_{J',0}$ there exists $\gamma \in \Phi_J^+\setminus  \Phi_{J'}^+$  with $(v_xw_{J'}w_J y)^{-1}(\gamma)<0$, hence  $ f_{v'}T_{\tilde w_1 \tilde y } $ vanishes on $Z^+$.
Hence we may assume that $v_x w_{J'}w_J y\in W_{J',0}$.
  
   We recall:
    
   \begin{lemma}[{\cite[IV.1, Exercise 3]{MR1890629}}]\label{Bou}Let $J\subset \Delta$. Every coset $wW_{J,0}$ in $W_0$ has a unique representative $d$ of minimal length. We have $\ell(du )=\ell(d)+\ell(u)$ for all $u\in W_{J,0}$. 
An element   $d\in W_{0}$ is the representative of minimal length in $dW_{J,0}$ if and only if $d(J)\subset \Phi^+$.
\end{lemma}

 The element $w_J w_{J'}$ is the  representative of minimal length of the coset $w_JW_{J',0}$.   
Since $v_xw_{J'}w_Jy  \in W_{J',0}$, we have $y  \in  w_JW_{J',0}$, so $y = w_J w_{J'}$, as $y\leq w_J w_{J'}$ by assumption.

We deduce that  $f_{v'}T_{\tilde w_1 \tilde y }$  vanishes on $Z^+$ if $y\neq  w_J w_{J'}$.
  
  Assume $y=w_J w_{J'}$. Then $\tilde x=\tilde w_1 \tilde y$ lifts $x=\lambda_x v_x$.  If $ f_{v'}T_{\tilde x} $ does not vanish on $Z^+$, then by above we have  $  v_x^{-1} \cdot \lambda_x \in \Lambda^+$. If $  v_x^{-1} \cdot \lambda_x \in \Lambda^+$
  then $\ell(x)=\ell(  v_x(v_x^{-1} \cdot \lambda_x))=\ell(v_x)+\ell(v_x^{-1} \cdot \lambda_x)$, and 
    by the braid relations $ f_{v'}T_{\tilde x}= f_{v'}T_{\tilde v_x}T_{ v_x^{-1} \cdot \lambda_x }$.  
        
The element $f_{v'}\in (\ind_K^G V')^I$ generates a subrepresentation of $K$ isomorphic to  $V'$.  The parameter  of 
  the character  of  $\mathcal H_C(K,I)$ acting on $C f_{v'}$ is $(\psi_{V'}^{-1}, J')$ (Lemma \ref{repKmodH}).   By \eqref{zero},   
 $ f_{v'}T_{\tilde v_x}  =0$  for  $v_x\in W_{J',0}-\{1\}$. We deduce that $ f_{v'}T_{\tilde x}=0$, except if $x\in \Lambda^+$ and $y=w_J w_{J'}$.
   \end{proof}
   This ends the proof of  \eqref{Z++}  hence of  
Proposition \ref{expa}.
     \end{proof}

\section{A simple proof of  the change of weight theorem for certain \texorpdfstring{$G$}{G}}\label{sec:simple-proof-change}
In this section, we give a simple proof of the change of weight theorem (Theorem~\ref{weight}) when $\mathbf G$ is split.
For $\GL_n$ (and more generally for any split group, see \S\ref{sec:more-general-case}) this gives a more elementary proof than the one in \cite{MR2845621} and \cite{MR3143708}, avoiding the Lusztig-Kato theorem.

Since $\mathbf G$ is split, $\mathbf Z$ is equal to $\mathbf S$ and $v_Z$ gives an isomorphism $X_*(\mathbf{S})\simeq S/S^0 = \Lambda$, and Bruhat-Tits theory gives a Chevalley group scheme $\mathcal G $ with generic fiber $\mathbf G$ and
  such that $\mathcal G (\cO)=K $ is the  special maximal compact open subgroup   of $G$ fixing $x_0$ \cite[3.4.2]{MR546588}. We have $\mathcal G (k) = G_k$, the root system $\Phi$ of $(G,S)$ identifies canonically with the root system of $(G_k,S_k)$. 

\begin{lemma}\label{lm:Z-cap-Malpha} Assume that $\mathbf G$ is $F$-split.
  For $\alpha\in\Delta$, we have $Z\cap M'_\alpha = \alpha^\vee(F^\times), \ Z^0\cap M'_\alpha =
  \alpha^\vee(\mathcal{O}^\times)$, and $Z_k\cap M'_{\alpha,k} = \alpha^\vee(k^\times)$.
\end{lemma}

\begin{proof}
  Note that $\mathbf M_\alpha^{\der}$ is a semisimple group of rank~1 and that $M_\alpha' \subset
  M_\alpha^{\der}$.  Hence the first two equalities are reduced to the case where $\mathbf G$ is
  semisimple of rank~1 and hence isomorphic to $\SL_2$ or $\PGL_2$ \cite[Thm.\ 7.2.4]{MR2458469}.
  In either case the first two equalities are easily verified by hand, noting that $\mathbf Z \cong
  \mathbb G_m$ and so the parahoric $Z^0$ is the maximal compact $\cO^\times \subset F^\times$.
  For the third equality, the same proof as for the first one works, but now one works over $k$ instead of $F$.
\end{proof}

By the lemma, for a character $\psi\colon Z_k\to C^\times$, which is also regarded as a character of $Z^0$ by the quotient map $Z^0\onto Z_k$, $\psi$ is trivial on $Z_k\cap M'_{\alpha,k}$ if and only if $\psi$ is trivial on $Z^0\cap M'_\alpha$.
Hence $\Delta(V) = \Delta'(V)$ for any irreducible representation $V$ of $K$.

In this section we prove Theorem~\ref{weight2}. We will first focus on the case when 
the center of $\mathbf G$ is a torus (i.e.\ smooth and connected) and the derived subgroup of $\mathbf G$ is simply connected.
In fact, just as in the first proof of Proposition~\ref{ISimpliesCW} we prove a stronger version
which we now state. Fix $\alpha, V,V'$ as in Theorem~\ref{weight2}.

\begin{theorem}\label{thm:weight2,simple}
Suppose that $\mathbf G$ is a split group whose center is a torus and whose derived subgroup is simply-connected.
Let $z\in Z^+$ such that $\langle \alpha, v_Z(z)\rangle >0$, i.e.\ $z \in Z_G^+(V,V')$.
Then there exist $G$-equivariant homomorphisms $\varphi:\ind_K^GV\to \ind_K^GV'$ and $\varphi':\ind_K^GV'\to \ind_K^GV$ satisfying
 $$S^G(\varphi)=\tau_{z}^{V'_{U^0},V_{U^0}},\quad 
 S^G(\varphi') =\tau_{z}^{V_{U^0},V'_{U^0}}  - \tau^{V_{U^0},V'_{U^0}} _{ z a_\alpha }.$$
If moreover $\langle \beta,v_Z(z) \rangle = 0$ for $\beta \in \Delta(V')$, then
$\varphi = T_z^{V',V}$ and $\varphi' = T_z^{V,V'}$.
\end{theorem}

\begin{remark}\label{rmk:weight2,simple}
Recall that we fixed an isomorphism of vector spaces $\iota\colon V_{U^0}\simeq V'_{U^0}$ \eqref{iopi}. This is also an isomorphism of representations of $Z^0$ because $\psi_V=\psi_{V'}$. 
We have isomorphisms $\mathcal{H}_Z(V_{U^0},V'_{U^0})\simeq \mathcal{H}_Z(V'_{U^0},V_{U^0})\simeq \mathcal{H}_Z(V'_{U^0},V'_{U^0})= \mathcal{H}_Z(V'_{U^0}) \simeq \mathcal{H}_Z(V_{U^0},V_{U^0})= \mathcal{H}_Z(V_{U^0})$ and    for $x\in Z$, $ \tau_x^{V_{U^0}} , \tau_x^{V'_{U^0}} $ correspond to each other  under the isomorphism $\mathcal{H}_Z(V_{U^0} ) \simeq \mathcal{H}_Z(V'_{U^0} )$, and we will just denote them by $\tau_x$.
We remark that since $Z = S$ is commutative, $\mathcal{H}_G(V_{U^0})$ is commutative.
\end{remark}

The basic idea of the proof is the following.
We construct many $G$-representations $\pi$ that contain the weight $V$ but not the weight $V'$. This implies that 
$\chi\otimes \ind_K^G V \not\simeq\chi\otimes \ind_K^G V'$ for any homomorphism $\chi : \mathcal{H}_G(V) \simeq \mathcal{H}_G(V') \to C$
that occurs in $\Hom_K(V,\pi)$. This in turn implies that $\chi(T_z^{V,V'}*T_z^{V',V}) = 0$ for such $\chi$.
When $z$ is as in Theorem~\ref{thm:weight2,simple} and chosen minimally, i.e.\ $\langle \alpha,v_Z(z) \rangle = 1$
and $\langle \beta,v_Z(z) \rangle = 0$ for $\beta \in \Delta\setminus\{\alpha\}$, then
it turns out that $S^G(T_z^{V,V'}*T_z^{V',V})$ is so constrained that it is forced to be equal to
$\tau_{z^2}  - \tau_{ z^2a_\alpha }$. %
By Lemma~\ref{first} we have $S^G(T_z^{V',V}) = \tau_{z}^{V'_{U^0},V_{U^0}}$, and we deduce that $S^G(T_z^{V,V'}) = \tau_{z}^{V_{U^0},V'_{U^0}}  - \tau_{ z a_\alpha }^{V_{U^0},V'_{U^0}}$. Using properties of $S^G$ it is then not difficult to deduce the theorem.

\subsection{The case of \texorpdfstring{$\GL_2$}{GL\_2}}\label{sec:case-of-gl2}

To warm up, in this section we illustrate the proof strategy by showing that $S^G(T_z^{V,V'}*T_z^{V',V}) = \tau_{z^2}  - \tau_{ z^2a_\alpha }$
when $\mathbf{G} = \GL_2$, $V$ is the trivial representation $ \trivrep_K$ of $K$, $V'$ is the Steinberg representation $\SSt_K$ of $K$, and $z = \diag(\varpi,1)$ where $\varpi$ is a uniformizer. 
We note that $\tau_\alpha = \tau_{\diag(\varpi^{-1},\varpi)}$, so $\tau_{z^2a_\alpha} = \tau_{\diag(\varpi,\varpi)}$.
The Satake homomorphism  $S^G$ satisfies (see \cite[proof of Prop.\ 6.3]{MR2845621} or Lemma~\ref{lm:support-satake}):
\begin{itemize}
\item  $ S^G(T_{z}^{V',V})(z')\neq 0$ implies    $v_Z(z')\in v_Z(z) + \mathbb{R}_{\le 0}\Delta^\vee$.
\item  The coefficient of $\tau_z^{V'_{U^0}, V_{U^0}}$ in $S^G(T_{z}^{V',V}) $ is $1$.
\end{itemize}
This also holds after switching $V$ and $V'$. This means that $S^G(T_{z}^{V',V}) \in  \tau^{V'_{U^0}, V_{U^0}}_z + \sum_{n < 0}C\tau^{V'_{U^0}, V_{U^0}}_{\diag(\varpi^{n + 1},\varpi^{-n})}$, similarly after switching $V$ and $V'$, and
   $ S^G(T_z^{V,V'})\circ S^G(T_z^{V',V}) \in \tau_{z^2} + \sum_{n< 0}C\tau_{\diag(\varpi^{n + 2},\varpi^{-n})}$.  The support of $S^G(f) \in \mathcal{H}_Z(\trivrep_{Z^0})$ is contained in $Z^+$ for any $f \in\mathcal{H}_G(\trivrep_K)$. For $n<0$,  if $ \diag(\varpi^{n + 2},\varpi^{-n}) \in Z^+$ then $n = -1$, so
$$ S^G(T_z^{V,V'}\circ T_z^{V',V}) = \tau_z^2 + c\tau_{\diag(\varpi,\varpi)}$$
for some $c\in C$.
  Let  $\chi_1: \mathcal{H}_Z(\trivrep_{Z^0})  \to C$ be the character such that $\chi_1(\tau_z)=\chi_1  (\tau_{\diag(\varpi,\varpi)})=1$.    We also denote by $\chi_1$ the character $\chi_1 \circ S^G$ of  $\mathcal{H}_G(\trivrep_K) \simeq \mathcal{H}_G(\SSt_K)$.
   The algebra   $\mathcal{H}_G(\trivrep_K)$ acts on the line $ \Hom_G(\ind_K^G\trivrep_K,\trivrep_G)$  by the 
 character $\chi_1 $ because the embedding $\trivrep_G \hookrightarrow \Ind_B^G \trivrep_Z$ implies  \[
\Hom_K(\trivrep_K,\trivrep_G)\hookrightarrow \Hom_K(\trivrep_K ,\Ind_B^G\trivrep_Z)=\Hom_K(\trivrep_K,\Ind_{B^0}^K\trivrep_Z) \simeq \Hom_{Z^0}(\trivrep_K|_{Z^0},\trivrep_Z|_{Z^0}),
\]
and the isomorphism  $\Hom_K(\trivrep_K,\trivrep_G) \to \Hom_{Z^0}(\trivrep_K|_{Z^0},\trivrep_Z|_{Z^0})$  is  $\mathcal{H}_G(\trivrep_K)$-equivariant via $S^G$ \cite[Lemma 2.14]{MR2845621}.  Hence $\trivrep_G$ is a quotient of $\chi_1\otimes\ind_K^G\trivrep_K $ and 
$$\chi_1\otimes\ind_K^G\trivrep_K\not\simeq \chi_1\otimes\ind_K^G\SSt_K.$$
(If these are isomorphic to each other, then we have a non-zero homomorphism $\ind_K^G\SSt_K\to \chi_1\otimes\ind_K^G\SSt_K\simeq \chi_1\otimes\ind_K^G\trivrep_K\to \trivrep_G$ which gives $\SSt_K\to \trivrep_G|_K$ by Frobenius reciprocity. This is a contradiction.)
    For a character $\chi: \mathcal{H}_Z(\trivrep_{Z^0})  \to C$ such that $\chi(\tau_z^2 + c\tau_{\diag(\varpi,\varpi)})\ne 0$, we have  $\chi\otimes\ind_K^GV\simeq \chi\otimes\ind_K^GV'$.   
Therefore $\chi_1(\tau_z^2 + c\tau_{\diag(\varpi,\varpi)}) = 0$,
hence $c = -1$ as desired.

\subsection{Reducibility and change of weight}\label{subsec:Reducibility and change of weight}
Until the end of \S\ref{sec:corollary}, fix $\mathbf G, \alpha, V,V'$ as in Theorem~\ref{thm:weight2,simple}.

Let $\chi\colon \mathcal{H}_Z(V_{U^0})\to C$ be a character. Since $Z^0\subset Z$ is normal, $\ind_{Z^0}^ZV_{U^0}$ is  a free $\mathcal{H}_Z(V_{U^0})$-module of rank $1$.  
 The character $\chi\otimes_{\mathcal{H}_Z(V_{U^0})}\ind_{Z^0}^ZV_{U^0}$ of $Z$ is $z\mapsto \chi (\tau_{z^{-1}})$  because $\tau_{z^{-1}} = z$ as endomorphisms
of $\ind_{Z^0}^ZV_{U^0}$; its restriction to $Z^0$ is $\psi_V$ because $\tau_{z^{-1}} = \psi_V(z)\tau_1 = \psi_V(z)$ in $\mathcal{H}_Z(V_{U^0})$
for $z\in Z^0$.  Since $\psi_V$ is trivial on $Z^0 \cap M'_\alpha$, $\tau_\alpha$ is well-defined. 
 
 Assume that $\chi(\tau_\alpha) = 1$.  The character $z\mapsto \chi (\tau_{z^{-1}})$ of $Z$ is trivial on  $Z \cap M'_\alpha = \alpha^\vee(F^\times)$, hence we can extend it to a  character of  $M_\alpha$ that is trivial on $U \cap M_\alpha$~(\cite[Proposition~3.3]{MR3143708}, \cite[II.7~Corollary 1]{MR3600042}).
We denote this extended character by $\sigma_\chi$.

\begin{lemma}[{\cite[III.18~Proposition]{MR3600042}}] \label{va} Assume that $\chi\colon \mathcal{H}_Z(V_{U^0})\to C$ satisfies $\chi(\tau_\alpha) = 1$.
Then $\Hom_K(V,\Ind_{P_\alpha}^G\sigma_\chi) \ne 0$ and $\Hom_K(V',\Ind_{P_\alpha}^G\sigma_\chi) = 0$.
\end{lemma}
\begin{proof}
By Frobenius reciprocity, the Iwasawa decomposition $G = P_\alpha K$ and using $P_\alpha^0=M_\alpha^0 N_\alpha^0$ we have 
$$\Hom_K(V_1,\Ind_{P_\alpha}^G\sigma_\chi) =\Hom_K(V_1,\Ind_{P_\alpha^0}^K\sigma_\chi)\simeq  \Hom_{M_\alpha^0}((V_1)_{N_\alpha^0},\sigma_\chi)$$
for any irreducible representation $V_1$ of $K$.
The parameter of $V_{N_\alpha^0}$ is  $(\psi_V,\{\alpha\})$, the parameter of $V'_{N_\alpha^0}$ is   $(\psi_V,\varnothing)$  \cite[III.10 Lemma]{MR3600042}.
On the other hand, the parameter of the character $\sigma_\chi|_{M_\alpha^0}$ is $(\psi_V,\{\alpha\})$ \cite[III.10 Remark]{MR3600042}.
\end{proof}
\begin{lemma}\label{vb}
Assume that $\chi\colon \mathcal{H}_Z(V_{U^0})\to C$ satisfies $\chi(\tau_\alpha) = 1$.
Then $$\chi\otimes_{\mathcal{H}_G(V)}\ind_K^GV\not\simeq\chi\otimes_{\mathcal{H}_G(V)}\ind_K^GV'.$$
\end{lemma}

\begin{proof} 
By definition of $\sigma_\chi$ we have an $M_\alpha$-equivariant map $\sigma_\chi \into \Ind_{B \cap M_\alpha}^{M_\alpha} (\chi \otimes _{\mathcal{H}_Z(V_{U^0})}\ind_{Z^0}^ZV_{U^0})$. By exactness of parabolic induction we get
\begin{align*}
\Hom_K(V,\Ind_{P_\alpha}^G\sigma_\chi)
&\hookrightarrow
\Hom_K(V,\Ind_B^G(\chi\otimes_{\mathcal{H}_Z(V_{U^0})}\ind_{Z^0}^ZV_{U^0}))\\
&\simeq
\Hom_{Z^0}(V_{U^0},\chi\otimes_{\mathcal{H}_Z(V_{U^0})}\ind_{Z^0}^ZV_{U^0}),
\end{align*}
and this map is $\mathcal{H}_G(V)$-linear with respect to $S^G$. 
The latter space is one-dimensional and  the Hecke algebra $\mathcal{H}_Z(V_{U^0})$ acts on this line by the character  $\chi$.
Hence a non-trivial homomorphism $\ind_K^GV\to \Ind_{P_\alpha}^G\sigma_\chi$ (which exists by Lemma~\ref{va}) factors through $\ind_K^GV\onto \chi\otimes_{\mathcal{H}_G(V)}\ind_K^GV$. If $\chi\otimes_{\mathcal{H}_G(V)}\ind_K^GV $ were isomorphic to $\chi\otimes_{\mathcal{H}_G(V)}\ind_K^GV'$,  we would have  a  non-zero homomorphism $\ind_K^GV' \onto \chi\otimes_{\mathcal{H}_G(V)}\ind_K^GV' \to \Ind_{P_\alpha}^G\sigma_\chi$  contradicting $\Hom_K(V',\Ind_{P_\alpha}^G\sigma_\chi) = 0$ (Lemma \ref{va}).
 \end{proof}

\subsection{Proof of Theorem~\ref{thm:weight2,simple} (minuscule case)}\label{sec:simple-proof-of-change-of-wt}
The hypothesis that the center of $\mathbf G$ is a torus is equivalent to $\mathbb Z \Phi$ being a direct summand of $X^*(\mathbf S)$, for example by \cite[(154)]{milne-iAG}.
Hence, for each $\alpha\in\Delta$ we have a fundamental coweight $\mu_\alpha\in X_*(\mathbf{S})$.
Namely we have $\langle \alpha,\mu_\alpha\rangle = 1$ and $\langle \beta,\mu_\alpha\rangle = 0$ for any $\beta\in\Delta\setminus\{\alpha\}$.
In this section we consider $z\in Z$ such that $v_Z(z)=\mu_\alpha$. %

The element $\tau_\alpha - 1\in \mathcal{H}_Z(V_{U^0} )$ is irreducible, since the derived subgroup of $\mathbf G$ is simply connected \cite[Remark~2.5 and Lemma~4.17]{MR3143708} (alternatively, one can argue as in Lemma~\ref{lm:irred}). Put  $f = S^G(T_z^{V,V'}* T_z^{V',V})$ in $ \mathcal{H}_Z(V_{U^0} )$.  Lemma \ref{vb} implies that  $\chi(f) = 0$ for any character $\chi\colon \mathcal{H}_Z(V_{U^0} )\to C$ such that  $\chi(\tau_\alpha) = 1$.  By the Nullstellensatz, we see that $f$ is contained in the radical of the ideal $(\tau_\alpha-1)$,
hence as $\tau_\alpha - 1$ is irreducible and $\mathcal{H}_Z(V_{U^0} )$ is a UFD, we deduce that
$f= f'(1 - \tau_\alpha)$ for some $f'\in \mathcal{H}_Z(V_{U^0} )$.
We will prove that $f' = \tau_{z^2}$.

Consider any $z' \in \supp f'$.
We claim that both $z'$ and $z' a_\alpha$ lie in $Z^+$ and that $v_Z(z' )\in 2v_Z(z) + \mathbb{R}_{\le 0}\Delta^\vee$. 
To see this, pick
$r, s \ge 0$ maximal such that $z' a_\alpha^i \in \supp f'$ for $-r \le i \le s$. Then $z' a_\alpha^{-r}$, $z' a_\alpha^{s+1} \in \supp f$,
so they both lie in $Z^+$. By convexity of the dominant region we deduce that $z'$, $z' a_\alpha \in Z^+$.
Similarly, as recalled in \S\ref{sec:case-of-gl2}, we know that  $v_Z(z' a_\alpha^i)\in 2v_Z(z) + \mathbb{R}_{\le 0}\Delta^\vee$ for $i \in \{-r, s+1\}$, hence by convexity we have $v_Z(z' )\in 2v_Z(z) + \mathbb{R}_{\le 0}\Delta^\vee$.

There exist $n_\beta\in\mathbb{R}_{\ge 0}$ for $\beta\in\Delta$ such that 
$v_Z(z') = 2\mu_\alpha  - \sum_{\beta\in \Delta} n_\beta\beta^\vee$. Recalling $v_Z(a_\alpha)=-\alpha^\vee$, we have 
$v_Z(z' a_\alpha) = 2\mu_\alpha - \alpha^\vee - \sum_{\beta\in \Delta} n_\beta\beta^\vee$.
Let $\gamma\in\Delta$. 
If $\gamma\ne \alpha$, then $\sum_{\beta\in\Delta}n_\beta\langle \gamma,\beta^\vee\rangle = -\langle \gamma,v_Z(z')\rangle \le 0$.
If $\gamma = \alpha$, then $\sum_{\beta\in\Delta}n_\beta\langle \gamma,\beta^\vee\rangle =  2 - \langle \alpha,\alpha^\vee\rangle - \langle \alpha,v_Z(z'a_\alpha)\rangle = - \langle \alpha,v_Z(z'a_\alpha)\rangle\leq 0$.
Hence $\sum_{\beta\in\Delta}n_\beta\langle \gamma,\beta^\vee\rangle \le 0$ for any $\gamma\in\Delta$.
Since $(d_\gamma\langle \gamma,\beta^\vee\rangle)_{\beta,\gamma\in\Delta}$ is positive definite for some $d_\gamma > 0$, we have $n_\beta = 0$ for any $\beta\in\Delta$.
We deduce that $z'\in z^2Z^0$  (as $Z^0$ is the kernel of $v_Z$).
So $f' \in C^\times \tau_{z^2}$.
Since the coefficient of $\tau_{z^2}$ in $f$ is $1$, we get $f = S^G(T_z^{V,V'}* T_z^{V',V}) = \tau_{z^2}  - \tau_{ z^2a_\alpha }$.

By Lemma~\ref{first} we have $S^G(T_z^{V',V}) = \tau_{z}^{V'_{U^0},V_{U^0}}$, hence we deduce that $S^G(T_z^{V,V'}) = \tau_{z}^{V_{U^0},V'_{U^0}}  - \tau_{ z a_\alpha }^{V_{U^0},V'_{U^0}}$.
This completes the proof of Theorem~\ref{thm:weight2,simple} when $v_Z(z)=\mu_\alpha$.

\subsection{Proof of Theorem~\ref{thm:weight2,simple} (general case)}\label{sec:simple-proof-of-change-of-wt2}

We consider now  $z\in Z^+$ such that $\langle \alpha, v_Z(z)\rangle >0$.
Take $z_0\in Z$ such that $v_Z(z_0) = \mu_\alpha$.
Then $z z_0^{-1} \in Z^+$ and 
from \eqref{imS} we deduce the existence of $\theta \in \mathcal{H}_G(V')$ such that $S^G(\theta) = \tau_{zz_0^{-1}}$.
Letting $\varphi = \theta * T_{z_0}^{V',V}$ and $\varphi' = T_{z_0}^{V,V'} * \theta$, we see from \S\ref{sec:simple-proof-of-change-of-wt} that
$S^G(\varphi) = \tau_z$ and $S^G(\varphi') = \tau_z- \tau_{ z a_\alpha }$.

In the special case that $\langle \beta,v_Z(z) \rangle = 0$ for $\beta \in \Delta(V')$, we have 
$\Delta(V') \subset \Delta_z \subset \Delta_{zz_0^{-1}}$, so
Lemma~\ref{first} shows that $\theta = T_{zz_0^{-1}}^{V',V'}$.
From Lemma~\ref{second} we then deduce that $\varphi = T_{z}^{V',V}$ and $\varphi' = T_{z}^{V,V'}$.

\subsection{A corollary}
\label{sec:corollary}

\begin{corollary}\label{cor:explicit-satake-simple}
  Suppose that $V$ is an irreducible  representation  of $K$ and that $z \in Z^+$ satisfies $\langle \alpha, v_Z(z)\rangle \ne 1$ for all $\alpha \in \Delta(V)$.
Then the image of $T_z\in \mathcal{H}_G(V)$ under the Satake transform $S^G$ is given by
  \[ S^G(T_z) = \tau_z \prod_{\alpha\in \Delta(V)\setminus \Delta_z} (1-\tau_\alpha). \]
\end{corollary}

\begin{proof}
  We induct on $\#(\Delta(V)\setminus \Delta_z)$. 
  If $\Delta(V)\subset \Delta_z$, then $S^G(T_z) = \tau_z$ by Lemma~\ref{first} and we are done.
  Otherwise we choose $\alpha \in \Delta(V)\setminus \Delta_z$ and take $z_0$ such that $v_Z(z_0) = \mu_\alpha$.
  Then $z z_0^{-2} \in Z^+$, as $\langle \alpha, v_Z(z)\rangle \ge 2$ by assumption. Define $V'$ by the parameter $(\psi_V, \Delta(V)\setminus \{\alpha\})$.
  Applying Lemma~\ref{second} twice (using that $\Delta(V') \subset \Delta_{z_0}$) we get that
  $T_z^{V,V} = T_{z_0}^{V,V'} * T_{z z_0^{-2}}^{V',V'} * T_{z_0}^{V',V}$. As $\Delta(V')\setminus \Delta_{z z_0^{-2}}$ is a proper
  subset of $\Delta(V)\setminus \Delta_z$ we get by induction that $S^G(T_{z z_0^{-2}}^{V',V'}) = \tau_{z z_0^{-2}} \prod_{\Delta(V')\setminus \Delta_z} (1-\tau_\beta)$.
  On the other hand, by Theorem~\ref{thm:weight2,simple} we have $S^G(T_{z_0}^{V',V}) = \tau_{z_0}$ and $S^G(T_{z_0}^{V,V'}) = \tau_{z_0}(1-\tau_\alpha)$.
  By combining these formulas we get the corollary.
\end{proof}

\begin{remark}
  It is not hard to deduce the corollary from Theorem~\ref{EIST}, noting that $z \prod_{\beta \in X} a_\beta \in Z^+$ for any subset
  $X \subset \Delta(V)\setminus \Delta_z$.
\end{remark}

\subsection{The general split case}\label{sec:more-general-case}

We now use two reduction steps to extend the above proof of Theorem~\ref{thm:weight2,simple} to the case
of general split groups $\mathbf G$.

\medskip
\noindent (1)
We remove  first the assumption on the center. Suppose that $\mathbf{G}$ is split with simply-connected derived subgroup.

Let $\mathbf G_1$ be the quotient of $\mathbf G \times \mathbf Z$ by the normal subgroup $\{ (z,z^{-1}) : z
\in \mathbf {Z_G} \}$, where $\mathbf{Z}_{\mathbf{G}}$ is the center of $\mathbf{G}$, as in \cite[5.18]{MR0393266}. Then the natural map $\mathbf G \to \mathbf G_1$
is a closed embedding that induces an isomorphism on derived subgroups. The natural map $\mathbf Z \to
\mathbf G_1$ to the second coordinate induces an isomorphism $\mathbf Z \congto \mathbf{Z_{G_1}}$.
In particular, $\mathbf G_1$ is as in Theorem~\ref{thm:weight2,simple}.
It follows that $\mathbf Z_1 := \mathbf Z \cdot \mathbf {Z_{G_1}} = (\mathbf Z \times \mathbf
Z)/\{ (z,z^{-1}) : z \in \mathbf {Z_G} \}$ is a minimal Levi (i.e.\ maximal $F$-torus) of $\mathbf G_1$.
Let $K_1$ be the hyperspecial parahoric subgroup of $G_1$ fixing the special point $x_0$.
Then we have $K = K_1 \cap G$, see Lemma~\ref{lm:parahoric-and-dual-z-ext}. We have (as in \cite[\S 3.2]{MR3143708}):

\begin{lemma}\label{lm:restrict-irreps}%
The following hold:
\begin{enumerate}
\item %
The restriction  to $K$ of any irreducible representation of $K_1$ is irreducible. Conversely, any irreducible  representation $V$ of $K$  extends to $K_1$. 
\item %
Let $V_1,V'_1$ be irreducible  representations of $K_1$ 
 and $V,V'$ their restrictions to $K$.
Then the restriction map $\varphi_1 \mapsto \varphi_1|_G$ gives an isomorphism between $\{\varphi_1\in \mathcal{H}_{G_1}(V_1,V'_1)\mid \supp \varphi_1\subset K_1ZK_1\}$ and $\mathcal{H}_{G}(V,V')$.
We have $S^G(\varphi_1|_G) = S^{G_1}(\varphi_1)|_Z$ for any $\varphi_1 \in \mathcal{H}_{G_1}(V_1,V'_1)$ with $\supp \varphi_1\subset K_1ZK_1$.
Moreover, we have $T_z^{V_1',V_1}|_G = T_z^{V',V}$ for any $z \in Z_G^+(V,V')$.
\end{enumerate}
\end{lemma}

Given $\alpha$, $V$, $V'$, $z \in Z^+_G(V,V')$ as in Theorem~\ref{thm:weight2,simple} we choose 
extensions $V_1, V'_1$ of $V, V'$ to $K_1$-representations and let $\varphi_1$, $\varphi'_1$ denote the
Hecke operators provided by Theorem~\ref{thm:weight2,simple}
for $G_1$, $V_1$, $V'_1$, $z$. Then, as the supports of $\tau_z$, $\tau_z - \tau_{z a_\alpha}$ are contained
in $Z(Z_1 \cap K_1)$, we deduce from the lemma that the supports of $\varphi_1$, $\varphi'_1$ are contained in $K_1 Z K_1$.
Hence we can take $\varphi = \varphi_1|_G$, $\varphi' = \varphi'_1|_G$. Similarly, Corollary~\ref{cor:explicit-satake-simple} continues to hold for $\mathbf{G}$.

\medskip
\noindent (2)
To remove the assumption on the derived subgroup, we use a $z$-extension.
(See \cite[\S3]{ct08} for more on $z$-extensions.)
Suppose that $\mathbf G$ is any split reductive group.
Choose a split $z$-extension $r\colon \widetilde{\mathbf{G}}\to \mathbf{G}$, i.e.\ an $F$-split group $\widetilde{\mathbf{G}}$ with simply connected derived subgroup which is a central extension of $\mathbf{G}$ and the kernel of $r$ is an ($F$-split) torus. In particular, part (1) above applies to $\wt{\mathbf G}$.
Set $\widetilde{\mathbf{Z}} = r^{-1}(\mathbf{Z})$; it is a maximal torus of $\widetilde{\mathbf{G}}$.
Let $\widetilde{K}\subset \widetilde{G}$ be the  special (maximal compact open) parahoric subgroup fixing $x_0$;
  the map $\widetilde{K}\to K$ is surjective \cite[Lemma 2.1]{MR3143708}, \cite[\S 3.5]{MR3331726}.

\begin{lemma}

Let $V_1,V_2$ be irreducible representations of $K$ and denote by $\widetilde{V}_1,\widetilde{V}_2$ their inflations to $\widetilde{K}$.
Then there exist algebra homomorphisms $\Theta_G : \mathcal{H}_{\widetilde{G}}(\widetilde{V}_1,\widetilde{V}_2)\to \mathcal{H}_G(V_1,V_2)$ and $\Theta_Z : \mathcal{H}_{\widetilde{Z}}((\widetilde{V}_1)_{U^0},(\widetilde{V}_2)_{U^0})\to \mathcal{H}_Z((V_1)_{U^0},(V_2)_{U^0})$ such that
\begin{enumerate}
\item $S^G \circ \Theta_G = \Theta_Z \circ S^{\wt G}$;
\item for $\widetilde{z}\in \widetilde{Z}^+$, $\Theta_G(T_{\widetilde{z}}^{\widetilde{V}_2,\widetilde{V}_1}) = T_z^{V_2,V_1}$ and $\Theta_Z(\tau_{\widetilde{z}}^{(\widetilde{V}_2)_{U^0},(\widetilde{V}_1)_{U^0}}) = \tau_z^{(V_2)_{U^0},(V_1)_{U^0}}$, where $z = r(\widetilde{z})$.
\end{enumerate}
\end{lemma}

To construct the algebra homomorphism $\Theta_G$,   we identify the category of representations of $G$ with the category of representations of $\wt G$ trivial on the kernel of the surjective homomorphism $r:\wt G \to G$, and we note 
 that Frobenius reciprocity (applied twice) induces a natural isomorphism
$\Hom_G(\ind_K^G V, \sigma) \simeq \Hom_{\wt G}(\ind_{\wt K}^{\wt G} \wt V,   \sigma)$ for representations $\sigma$ of $G$ (for any irreducible $K$-representation $V$ with inflation $\wt V$). 
In particular we get a $\wt G$-linear map $j_{V} : \ind_{\wt K}^{\wt G} \wt V \to \ind_K^G V$ corresponding to the identity map.
By Yoneda's lemma the above adjunction gives for any $\varphi \in \mathcal{H}_{\widetilde{G}}(\widetilde{V}_1,\widetilde{V}_2)$
a unique $\Theta_G(\varphi) \in \mathcal{H}_G(V_1,V_2)$ such that $j_{V_2} \circ \varphi = \Theta_G(\varphi) \circ j_{V_1}$. We  leave the details of the end of the proof of the lemma  to the reader. 

The lemma shows that Theorem~\ref{thm:weight2,simple} holds even for $G$ since it holds for $\wt G$: as $r : \widetilde Z \to Z$ is surjective, 
we can choose $\wt z$ with $r(\wt z) = z$. Suppose $\wt\varphi,\wt\varphi'$ are the Hecke operators provided by Theorem~\ref{thm:weight2,simple}
for $\wt G$, $\wt V$, $\wt V'$, $\wt z$. Then we can take $\varphi = \Theta_G(\wt\varphi)$, $\varphi' = \Theta_G(\wt\varphi')$.
Similarly, Corollary~\ref{cor:explicit-satake-simple} continues to hold for $\mathbf{G}$.

\appendix
\address[N.\ Abe]{Graduate School of Mathematical Sciences, the University of Tokyo, 3-8-1 Komaba, Meguro-ku, Tokyo 153-8914, Japan}
\email{abenori@ms.u-tokyo.ac.jp}
\address[F.\ Herzig]{Department of Mathematics, University of Toronto,
  40 St.\ George Street, Room 6290, Toronto, ON M5S 2E4, Canada}
\email{herzig@math.toronto.edu}
\makeatletter
\tracingmacros=1
\def\@tocwrite#1#2{}%
\section{A simple proof of the change of weight theorem for quasi-split groups}\label{sec:simple-proof-change-1}
\let\@tocwrite=\original@@tocwrite %
\@tocwrite{section}{A simple proof of the change of weight theorem for quasi-split groups (by N.\ Abe and F.\ Herzig)}
\makeatother

\begin{center}
N.\ Abe and F.\ Herzig
\end{center}
\markboth{N.\ ABE AND F.\ HERZIG}{SIMPLE PROOF FOR QUASI-SPLIT GROUPS}
\medskip

\renewcommand{\O}{{\mathcal O}}
\newcommand{\tor}{_{\mathrm{tor}}}
\newcommand{\s}{^\times}
\newcommand{\dual}{^\vee}
\newcommand{\ang}[1]{\langle #1 \rangle}
\newcommand{\kind}{\ind_K^G}
\newcommand{\vp}{\varphi}

The purpose of the appendix is to show that the simple proof of \S\ref{sec:simple-proof-change} extends to quasi-split groups.

Suppose that $\mathbf G$ is a quasi-split connected reductive group over $F$.
As in \S\ref{sec:elem-a-alpha}, recall that if $\mathbf H$ is any connected reductive $F$-group, then $H'$ denotes the subgroup of $H$ generated by the unipotent
radicals of all minimal parabolics.  By Kneser--Tits (see e.g.\ \cite[II.3 Prop.]{MR3600042}) we know that $H' = H^{\der}$ if $\mathbf H^{\der}$ is simply connected with no anisotropic
factors. (Note that the second condition is automatic if $\mathbf H$ is quasi-split.)  Similarly we define $H'$ for $\mathbf H$
connected reductive over $k$ and know that $H' = H^{\der}$ if $\mathbf H^{\der}$ is simply connected.

We also recall that all special parahoric subgroups $K$ in this paper are associated to special
points in the apartment of $S$. We let $\red : K \onto G_k$ denote the natural reduction map whose kernel is the
pro-$p$ radical (i.e.\ largest normal pro-$p$ subgroup) of $K$.

\begin{theorem}\label{thm:change-of-weight}
  There exists a special parahoric subgroup $K$ of $G$ such that the following holds.

  Suppose that $V$, $V'$ are irreducible representations of $K$ and $\alpha \in \Delta$ such that $\psi_V = \psi_{V'}$ and
  $\Delta(V) = \Delta(V') \sqcup \{\alpha\}$, and let $z\in Z^+$ such that $\langle \alpha, v_Z(z)\rangle >0$.  Then
  there exist $G$-equivariant homomorphisms $\varphi:\ind_K^GV\to \ind_K^GV'$ and
  $\varphi':\ind_K^GV'\to \ind_K^GV$ satisfying
  $$S^G(\varphi)=\tau_{z},\quad 
  S^G(\varphi') =\tau_{z} - \tau_{ z a_\alpha }.$$ If moreover $\langle \beta,v_Z(z)
  \rangle = 0$ for $\beta \in \Delta(V')$, then $\varphi = T_z^{V',V}$ and $\varphi' = T_z^{V,V'}$.

  Any choice of $K$ works, provided the adjoint quotient $\mathbf G_{\ad}$ of $\mathbf G$ does not have a simple factor isomorphic to $\Res_{E/F}
  \PU(m+1,m)$ for some $E/F$ finite separable and $m \ge 1$.
\end{theorem}

\begin{remark}
  This is enough to establish Theorems 1--3 of \cite{MR3600042} for quasi-split $\mathbf G$, avoiding 
  \cite[\S IV]{MR3600042}, since the proofs given there only require one choice of $K$.
\end{remark}

\begin{remark}
  There exist quasi-split groups $\mathbf G$ and special parahoric subgroups $K$ for which the
  conclusion %
  of Theorem~\ref{thm:change-of-weight} fails. We claim that it suffices to show that $\psi_V(Z^0
  \cap M_\alpha') \ne 1$ for some $\mathbf G$, $K$, $V$, $\alpha$ as in
  Theorem~\ref{thm:change-of-weight}. Under this condition, Theorem~\ref{EIST} tells us that the
  image $S^G(\mathcal{H}_G(V',V))$ has $C$-basis $\tau_z$, where $z$ runs through a system of
  representatives of $Z_G^+(V',V)/Z^0$ in $Z_G^+(V',V)$. If Theorem~\ref{thm:change-of-weight} were
  true, then for $z \in Z_G^+(V',V)$ the element $\tau_z-\tau_{ z a_\alpha }$ would lie in $S^G(\mathcal{H}_G(V',V))$, so
  $za_\alpha \in Z_G^+(V',V)$. However, for large $n$ we have $za_\alpha^n \not\in Z^+$.

  For example, if $\mathbf G = \SU(2,1)$ defined by a ramified separable quadratic
  extension of $F$, then we can choose $K$ such that $\mathbf G_k \cong \PGL_2$ and if $\# k$ is odd, then $\red(Z^0 \cap
  M_\alpha') = Z_k$ strictly contains $Z_k \cap M_{\alpha,k}'$ (where $\Delta = \{\alpha\}$). Or, suppose
  that $\mathbf G = \SU(2,1)$ defined by the unramified separable quadratic extension. Then for any non-hyperspecial $K$ we have $\mathbf G_k
  \cong \U(1,1)$, and then $\red(Z^0 \cap M_\alpha') = Z_k$ strictly contains $Z_k \cap M_{\alpha,k}'$ (where
  $\Delta = \{\alpha\}$). In either case we can therefore choose $V$ such that $\psi_V(Z^0 \cap M_\alpha') \ne 1$.
\end{remark}

\subsection{On special parahoric subgroups}
\label{sec:spec-parah-subgr}

\begin{proposition}\label{prop:choosing-K}
  There exists a special parahoric subgroup $K$ of $G$ such that
  for any $\alpha \in \Delta$ the image of $M_\alpha' \cap K$ in $G_k$ is equal to $M_{\alpha,k}'$.  Any choice of
  $K$ works, provided the adjoint group $\mathbf G_{\ad}$ does not have a simple factor isomorphic to $\Res_{E/F} \PU(m+1,m)$ for
  some $E/F$ finite separable and $m \ge 1$.
\end{proposition}

\begin{proof}
  \emph{Step 1: We show that for any quasi-split $\mathbf G$ such that $\mathbf G^{\der}$ simply connected we can choose a special parahoric subgroup $K$ such that $\red(G' \cap K) = G_k'$.}

  Since $\mathbf G$, and hence $\mathbf M_\alpha$, have simply-connected derived subgroups and $\mathbf G$ is quasi-split,
  we know that $G' = G^{\der}$ and $M_\alpha' = M_\alpha^{\der}$.
  Note that the pro-$p$ radical of $G' \cap K = G^{\der} \cap K$ is normal in $K$ and hence contained in the pro-$p$ radical
  of $K$. Hence we obtain a commutative diagram with injective horizontal arrows as follows:
  \begin{equation*}
    \xymatrix{G^{\der} \cap K \ar@{^{(}->}[r]\ar@{->>}[d] & K \ar@{->>}[d] \\
      (G^{\der})_k\ar@{^{(}->}[r] & G_k}
  \end{equation*}
  Note that the bottom map induces an isomorphism $(G^{\der})_k' \congto G_k'$ (since $U$ and $U_{\op}$ are contained in $G^{\der}$).
  It thus suffices to show that the inclusion $(G^{\der})_k' \subset (G^{\der})_k$ is an equality, and hence it's enough
  to show that $(\mathbf G^{\der})_k$ is semisimple and simply connected (for a suitable choice of $K$). 

  Note in the following that our choice of special $K$ is given by a subset $X \subset \Delta_{\loc}$ of the relative local
  Dynkin diagram of $\mathbf G$ \cite[1.11]{MR546588}, or equivalently of $\mathbf G^{\der}$, consisting of one special vertex in each
  component of $\Delta_{\loc}$.  (We write $\Delta_{\loc}$, $\Delta_{1,\loc}$ instead of $\Delta$, $\Delta_1$ in 
  \cite{MR546588} in order to avoid confusion.)
 
  We first determine for which $K$ we have that $(\mathbf G^{\der})_k$ is semisimple. The absolute rank of $(\mathbf G^{\der})_k$ is the
  relative rank of $\mathbf G^{\der}$ over the maximal unramified extension, i.e., it's $|\Delta_{1,\loc}|$ minus the number of
  components of $\Delta_{1,\loc}$.  On the other hand, the absolute semisimple rank of $(\mathbf G^{\der})_k$ equals the number of
  absolute simple roots of $(\mathbf G^{\der})_k$, i.e., the cardinality of $\Delta_{1,\loc} - \cup_{v \in X} O(v)$ in the notation
  of Tits, by \cite[3.5.2]{MR546588}.  It thus suffices to show that for any $v \in X$, $O(v)$ contains precisely one point
  of each component of $\Delta_{1,\loc}$ (it always contains at least one).

  Looking at the tables in \cite{MR546588} and keeping in mind the reduction steps to the absolutely almost simple case in
  \cite[1.12]{MR546588}, we see that any choice of $X$ works, as long as it does not contain any non-hyperspecial vertices in
  type ${}^2 A_{2m}'$ (in which case we can take the hyperspecial ones). In other words, we can always choose a special
  parahoric $K$ such that $(\mathbf G^{\der})_k$ is semisimple, and any $K$ works in case the adjoint group $\mathbf G_{\ad}$ does
  not have a simple factor isomorphic to $\Res_{E/F} \mathbf H$, where $\mathbf H \cong \PU(m+1,m)$ is unramified and $E/F$ is finite separable.

  Next we recall from \cite[\S3.5]{MR546588} that, since $\mathbf G^{\der}$ is semisimple and simply connected, the residual group
  $(\mathbf G^{\der})_k$ has simply connected derived subgroup, provided we let $K$ correspond to a subset $X$ satisfying the
  condition in the last sentence of \cite[\S3.5]{MR546588}, i.e.\ $\cup_{v \in X} O(v)$ contains a ``good special vertex''
  out of each connected component of $\Delta_{1,\loc}$.  Note that by Tits' tables this is always possible (in fact even if
  $\mathbf G$ isn't quasi-split).  Now note from Tits' tables that when $\mathbf G$ is quasi-split, his condition on $X$ is always
  satisfied, except when $\mathbf G_{\ad}$ has a factor of type ${}^2 A^{(1)}_{2m,m}$ and the special vertex at the long end is
  chosen. (In other words, $\mathbf G_{\ad}$ has a simple factor isomorphic to $\Res_{E/F} \mathbf H$, where $\mathbf H \cong \PU(m+1,m)$ is
  ramified and $E/F$ is finite separable.)  In this case we choose the special vertex at the other end.

  By combining the above, we see that we can always choose a special parahoric $K$ such that $(\mathbf G^{\der})_k$ is semisimple
  simply connected (and hence $\red(G' \cap K) = G_k'$), and any $K$ works in case the adjoint group $\mathbf G_{\ad}$ does
  not have a simple factor isomorphic to $\Res_{E/F} \PU(m+1,m)$ and $E/F$ is finite separable (or equivalently when the root
  system $\Phi$ is reduced).

  \emph{Step 2: We prove the proposition in the case where $\mathbf G^{\der}$ simply connected.}

  From Step 1 we know that $\red(M_\alpha' \cap K) = M_\alpha'$ for $\alpha \in \Delta$, provided $\mathbf M_{\alpha,\ad}$ isn't
  isomorphic to $\Res_{E/F} \PU(2,1)$ for some $E/F$. By considering indices of quasi-split groups, for example in \cite{bib:tits-ss-groups}, it follows that
  there is at most one exceptional $\alpha$ in each component of $\Delta$, namely the exceptional $\alpha$ are precisely the
  multipliable simple roots in components of $\Delta$ of type $\mathrm{BC}_r$.

  Suppose first that $\mathbf G^{\der}$ is almost simple, and suppose that there is an exceptional $\alpha \in \Delta$, i.e.\
  $\mathbf M_{\alpha,\ad} \cong \Res_{E/F} \PU(2,1)$ for some $E/F$. Then the choice of a special point for $\mathbf M_\alpha$ coming from Step 1
  corresponds to a choice of $\alpha$-wall $H_\alpha$ in the reduced building of $G$. (By $\alpha$-wall we just mean
  an affine hyperplane parallel to $\ker(\alpha)$.)
  Now choose arbitrary $\beta$-walls $H_\beta$ for $\beta \in \Delta-\{\alpha\}$. Then the special parahoric subgroup defined by
  the special point $\cap_{\beta \in \Delta} H_\beta$ works for this proposition.

  In general the reduced apartment of $G$ (for $S$) is a product of reduced apartments for all the almost simple factors of $G^{\der}$,
  and we obtain a desired special point by taking a product of special points that work for the almost simple factors (previous
  paragraph).

  \emph{Step 3: We deduce the proposition in general.}

  Suppose that $\mathbf G$ is any quasi-split group. Pick a $z$-extension $\pi : \mathbf{\wt G} \to \mathbf G$ of $\mathbf G$. Then $\mathbf{\wt G}$ and $\mathbf G$ have the same
  reduced building, and by Step 1 above we can choose a special point $x$ corresponding to a special parahoric $\wt K$ of $\wt G$
  such that $\red(\wt G' \cap \wt K) = {\wt G}_k'$. We will show that $\red(G' \cap K) = G_k'$. The argument showing that
  $\red(\wt M_\alpha' \cap \wt K) = {\wt M}_{\alpha,k}'$ implies $\red(M_\alpha' \cap K) = M_{\alpha,k}'$ is completely analogous.

  We have $\pi(\wt G') = G'$. If $K$ denotes the special parahoric of $G$ corresponding to $x$, then we also have $\pi(\wt K)
  = K$ (see part (d) of the proof of \cite[Proposition 3]{bib:haines-rap}). We claim that $\pi(\wt G' \cap \wt K) = G' \cap
  K$. Suppose that $g \in G' \cap K$ and pick $\wt g \in \wt G'$ such that $\pi(\wt g) = g$. Then $\wt g$ fixes the special
  point $x$ and it is in the kernel of the Kottwitz homomorphism (since $\wt G'$ is contained in that kernel). Hence $\wt g
  \in \wt K$, proving the claim. Similarly we see that $\pi(\wt U \cap \wt K) = U \cap K$ and $\pi({\wt U}_{\op} \cap \wt K) =
  U_{\op} \cap K$.

  Now note that the image under $\pi$ of the pro-$p$ radical of $\wt K$ is contained in the pro-$p$ radical of $K$.
  Hence we get a commutative diagram
  \begin{equation*}
    \xymatrix{\wt K \ar@{->>}[r]^{\pi}\ar@{->>}[d]_{\red} & K \ar@{->>}[d]^{\red} \\
      {\wt G}_k\ar@{->>}[r]^{\o\pi} & G_k}
  \end{equation*}
  and by the previous paragraph we see that $\o\pi({\wt G}_k') = G_k'$. It follows that
  $\red(G'\cap K) = \red(\pi(\wt G' \cap \wt K)) = \o\pi(\red(\wt G' \cap \wt K)) = \o\pi({\wt G}_k') = G_k'$.
\end{proof}

\begin{remark}\label{rk:morphism}
  Surely the map $(G^{\der})_k \to G_k$ in Step 1 of the proof arises from a closed immersion $(\mathbf G^{\der})_k \to \mathbf
  G_k$ of algebraic groups, but we do not know a reference. %
\end{remark}

\begin{corollary}\label{cor:choosing-K}
  For any $K$ for which Proposition~\ref{prop:choosing-K} holds, we have that $\red(Z^0 \cap M_\alpha') = Z_k \cap 
  M_{\alpha,k}'$ for any $\alpha \in \Delta$. 
\end{corollary}

\begin{proof}
  Choose $K$ as in Proposition~\ref{prop:choosing-K}. Let $K(1) := \ker(K \to G_k)$. Then $Z^0 K(1) = \red^{-1}(Z_k)$ and we
  deduce by the proposition that $Z_k \cap M_{\alpha,k}' = \red(Z^0K(1) \cap M_{\alpha}') = \red(Z^0 \cap M_\alpha')$,
  noting that we have an Iwahori decomposition $M_\alpha \cap K(1) = (Z \cap K(1))(U_\alpha \cap K(1))(U_{-\alpha} \cap
  K(1))$ and that $U_\alpha$, $U_{-\alpha}$ are contained in $M_\alpha'$.
\end{proof}

\subsection{Setup for the proof of Theorem~\ref{thm:change-of-weight}}
\label{sec:setup}

\emph{In Sections~\ref{sec:setup}--\ref{sec:first-reduction} we will assume that $\mathbf G^{\der}$ is simply connected and $\mathbf G/\mathbf
  G^{\der}$ is coflasque.} In Section~\ref{sec:second-reduct-step} we will reduce the general case to that one by using a
suitable $z$-extension.

We recall that an $F$-torus $\mathbf T$ is said to be \emph{coflasque} if we have $H^1(F', X^*(\mathbf T)) = 0$ for all finite separable
extensions $F'/F$ \cite[\S0.8]{ct08}. Note that any induced torus is coflasque.  We remark that if $\mathbf T$ is
coflasque, then $H^1(F'', X^*(\mathbf T)) = 0$ for any separable algebraic extension $F''/F$ (because by inflation-restriction it equals
$H^1(F'' \cap F(\mathbf T), X^*(\mathbf T))$, where $F(\mathbf T)$ is the splitting field of $\mathbf T$). 

We now observe that our assumptions on $\mathbf G$ imply that $\mathbf Z$ is a coflasque torus since (i) $\mathbf Z \cap \mathbf G^{\der}$ is
an induced torus because $\mathbf G^{\der}$ is simply connected and $\mathbf G$ is quasi-split, and (ii) any extension of a coflasque
torus by an induced torus is split (by Shapiro's lemma).

Let $\Gamma_F = \Gal(F\sep/F)$ with inertia subgroup $I_F$ and $\sigma$ a topological generator of $\Gamma_F/I_F$.
Let $L$ denote the fixed field of $I_F$, i.e.\ the maximal unramified extension of $F$.
Let $\Phi\abs$ (resp.\ $\Delta\abs$) denote the set of
absolute (resp.\ absolute simple) roots.  

\begin{lemma}\label{lm:max-cpt}
  Under the above assumptions, we have:
  \begin{enumerate}
  \item the group $X_*(\mathbf Z)_{I_F}$ is torsion-free;
  \item the group $\Lambda = Z/Z^0$ is a finite free $\Z$-module;
  \item any special parahoric $K$ of $G$ is maximal compact.
  \end{enumerate}
\end{lemma}

\begin{proof}
  We first show that if $\Gamma$ is a profinite group acting smoothly on a finite free $\Z$-module
  $X$, then the finite groups $H^1(\Gamma, X)$ and $\Hom_\Gamma(X,\Z)\tor$ are dual. By
  inflation-restriction, as $X$ is torsion-free, we reduce to the case where $\Gamma$ is finite
  (replacing $\Gamma$ with the finite quotient that acts faithfully on $X$). As
  $H^1(\Gamma, X) = \hat H^1(\Gamma, X)$ and $\Hom_\Gamma(X,\Z)\tor = \hat H^{-1}(\Gamma, \Hom(X,\Z))$,
  we conclude by \cite[Prop.\ 3.1.2]{nsw}.

  For our coflasque torus $\mathbf Z$ we conclude that $(X_*(\mathbf Z)_{I_F})\tor = 0$, as it is
  dual to $H^1(I_F, X^*(\mathbf Z))$. Hence $\Lambda \cong X_*(\mathbf Z)_{I_F}^\sigma$ \cite[Cor.\
  11.1.2]{haines-rostami} is a finite free $\Z$-module.  This implies that any $K$ is maximal
  compact \cite[Prop.\ 11.1.4]{haines-rostami}.
\end{proof}

By \cite[\S 7.2]{kottwitz} we have a $\sigma$-equivariant commutative diagram
\begin{equation}\label{eq:4}
  \begin{gathered}
    \xymatrix{\mathbf Z(L) \ar@{->>}[r]^-{w_{\mathbf Z}}\ar[dr]_-{v_{\mathbf Z}} & X_*(\mathbf Z)_{I_F} \ar[d]^{q_{\mathbf Z}} \\
      & \Hom(X^*(\mathbf Z)^{I_F},\Z),}
  \end{gathered}
\end{equation}
where $q_{\mathbf Z}([\lambda])(\mu) = \ang{\lambda,\mu}$ and
$v_{\mathbf Z}(z)(\mu) = \ord_F (\mu(z))$ (where the valuation $\ord_F$ is normalized so that $\ord_F(F\s) =
\Z$).  By Lemma~\ref{lm:max-cpt}(i) and \cite[\S 7.2]{kottwitz}, $q_{\mathbf Z}$ is an isomorphism. Since the
composite map $j : X_*(\mathbf Z)^{I_F} \into X_*(\mathbf Z) \onto X_*(\mathbf Z)_{I_F}$ becomes an
isomorphism after $\otimes \Q$, we get a $\sigma$-equivariant isomorphism $(q_{\mathbf Z} \circ j) \otimes \R
: (X_*(\mathbf Z)\otimes \R)^{I_F} \congto \Hom(X^*(\mathbf Z)^{I_F},\R)$. Let $\omega :
\Hom(X^*(\mathbf Z)^{I_F},\Z) \into (X_*(\mathbf Z)\otimes \R)^{I_F}$ denote the restriction of the
inverse of $(q_{\mathbf Z} \circ j) \otimes \R$ to the lattice $\Hom(X^*(\mathbf Z)^{I_F},\Z)$.

By taking $\sigma$-invariants in diagram \eqref{eq:4} composed with $\omega$ we obtain
\begin{equation*}
  \begin{gathered}
    \xymatrix{Z \ar@{->>}[r]^-{w_Z}\ar[dr]_-{v_Z} & X_*(\mathbf Z)^\sigma_{I_F} \ar@{^{(}->}[d]^{q_Z} \\
      & {(X_*(\mathbf Z)\otimes \R)^{\Gamma_F}}\rlap{${} = X_*(\mathbf S)\otimes \R$,}}
  \end{gathered}
\end{equation*}
where $w_Z$ is the Kottwitz homomorphism and $v_Z$ is as in~\S\ref{Notation}.
Explicitly, for $\lambda \in X_*(\mathbf Z)$,
\begin{equation}\label{eq:11}
  (\omega \circ q_{\mathbf Z})([\lambda]) = \frac 1{\#(I_F \cdot \lambda)} \sum_{\lambda' \in I_F \cdot \lambda} \lambda' \in (X_*(\mathbf Z)\otimes \R)^{I_F}.
\end{equation}

A root $\alpha \in \Phi$ determines a finite separable extension $F_\alpha/F$: it is the fixed field of the stabilizer of any lift $\wt \alpha
\in \Phi\abs$. (All lifts are $\Gamma_F$-conjugate, so the choice doesn't matter. Cf.\ \cite[4.1.3]{bib:BT2}.)  Let $\varepsilon_\alpha = e(F_\alpha/F)$ denote the
ramification degree.

\begin{lemma}\label{lm:image-in-Lambda}
  The image of $Z \cap M_\alpha'$ in $\Lambda$ is a direct summand. 
  Its image under $v_Z$ in $X_*(\mathbf S) \otimes \R$ is identified with
  $\Z \cdot \frac 1{\varepsilon_\alpha} \alpha_0\dual$, where $\alpha_0$ is the greatest multiple of $\alpha$ that is contained
  in $\Phi$.
\end{lemma}

\begin{proof}
  Note that $X_*(\mathbf Z \cap \mathbf M_\alpha^{\der})$ is a permutation module (a basis is given by all absolute simple coroots that restrict to
  $\alpha$), i.e.\ $\mathbf Z \cap \mathbf M_\alpha^{\der}$ is an induced torus. Similarly, $(\mathbf Z \cap \mathbf G^{\der})/(\mathbf Z \cap \mathbf
  M_\alpha^{\der})$ and $\mathbf Z \cap \mathbf G^{\der}$ are induced tori. Therefore, as $\mathbf Z/(\mathbf Z \cap \mathbf G^{\der})$ is coflasque by
  assumption, we deduce that $\mathbf Z/(\mathbf Z \cap \mathbf M_\alpha^{\der})$ is coflasque and hence that the sequence $1 \to \mathbf Z \cap \mathbf
  M_\alpha^{\der} \to \mathbf Z \to \mathbf Z / (\mathbf Z \cap \mathbf M_\alpha^{\der}) \to 1$ is split exact. The natural map $j : {Z \cap
    M_\alpha^{\der}}\to Z$ is compatible with the induced map $j_* : X_*(\mathbf Z \cap \mathbf M_\alpha^{\der})_{I_F}^\sigma \to X_*(\mathbf
  Z)_{I_F}^\sigma$ with respect to the functorial Kottwitz maps $w_{Z \cap M_\alpha^{\der}}$, $w_Z$. The map $j_*$ is clearly a
  split injection of finite free $\Z$-modules.

  As $X_*(\mathbf Z \cap \mathbf M_\alpha^{\der})$ has $\Z$-basis all $\wt\alpha \in \Phi\abs$ lifting $\alpha$, the image of $X_*(\mathbf Z
  \cap \mathbf M_\alpha^{\der})_{I_F}^\sigma$ in $X_*(\mathbf Z)_{I_F}^\sigma$ is generated by $[\sum_{\Phi'} \wt\alpha\dual] \in X_*(\mathbf
  Z)_{I_F}^\sigma$, where $\Phi' \subset \Phi\abs$ is a set of representatives for the $I_F$-orbits on the set of roots
  lifting $\alpha$. Using \eqref{eq:11} we see that it is identified with $\frac 1{\varepsilon_\alpha} \sum \wt\alpha\dual$ in $X_*(\mathbf
  S) \otimes \R$, where $\wt\alpha \in \Delta\abs$ now runs through all lifts of $\alpha$. By the lemma below this is equal to
  $\frac 1{\varepsilon_\alpha} \alpha_0\dual$.
\end{proof}

\begin{lemma}\label{lm:coroots}
  Let us drop temporarily all assumptions in \S\ref{sec:setup} about $\mathbf G$, and only assume that it is a quasi-split connected reductive $F$-group. Suppose
  that $\alpha \in \Delta$. Then $\alpha_0\dual = \sum \wt\alpha\dual$ in $X_*(\mathbf Z)$, where the sum is over all
  lifts $\wt\alpha$ of $\alpha$ in $\Phi\abs$.
\end{lemma}

\begin{proof}
  We may replace $\mathbf G$ with $\mathbf M_\alpha^{\der}$ and hence assume that $\mathbf G$ is semisimple and $\Delta = \{\alpha\}$.  Then
  $\Delta\abs = \{\wt \alpha_1,\dots,\wt \alpha_n\}$ for the lifts $\wt \alpha_i$ of $\alpha$ in $\Phi\abs$ and the
  cocharacters $\wt\alpha_i^\vee$ span $X_*(\mathbf Z) \otimes \Q$. In particular, as $\Gamma_F$ acts transitively on
  $\Delta\abs$, we see that $\alpha^\vee = c \sum \wt\alpha_i\dual$ for some constant $c \in \Q$. Note that $2\alpha \in
  \Phi$ if and only if $\wt \alpha_1 + \wt \alpha_i \in \Phi\abs$ for some $i > 1$ if and only if $\ang{\wt \alpha_1, \wt
    \alpha_i\dual} < 0$ (hence equal to $-1$) for some $i > 1$.

  If $2\alpha \not \in \Phi$, then the $\wt \alpha_i$ are pairwise orthogonal and $\ang{\alpha,\alpha\dual} = 2$ yields $c = 1$.
  Otherwise, since $\Gamma_F$ acts transitively on $\Delta\abs$ and the Dynkin diagram has no loops, it follows that 
  $\ang{\wt \alpha_1, \wt \alpha_i\dual} = -1$ for a unique $i > 1$. Then $\ang{\alpha,\alpha\dual} = 2$ yields $c = 2$.
\end{proof}

\begin{remark}\label{rk:v_Z-a_alpha}
  Lemma~\ref{lm:image-in-Lambda}, together with \cite[III.16 Notation]{MR3600042}, shows that
  $v_Z(a_\alpha) = -\frac 1{\varepsilon_\alpha} \alpha_0\dual$. Recall that in
  \S\ref{sec:elem-a-alpha} we also defined integers $e_\alpha$. By comparing with \cite[IV.11
  Example 3]{MR3600042} we deduce that $e_\alpha = 2\varepsilon_\alpha$ if $2\alpha \in \Phi$ and
  $e_\alpha = \varepsilon_\alpha$ otherwise.  Alternatively, we can see this by comparing
  \cite[4.2.21]{bib:BT2} with \cite[(39)]{MR3484112}.
\end{remark}

\subsection{Basic case}
\label{sec:basic}

\emph{We assume that $1 \to \mathbf {Z_G} \to \mathbf Z \to \mathbf Z/\mathbf {Z_G} \to 1$ is a split exact sequence of $F$-tori.}  
In particular, the center $\mathbf {Z_G}$ of $\mathbf G$ is a torus. 
We continue to assume that $\mathbf G^{\der}$ is simply connected and $\mathbf G/\mathbf G^{\der}$ is coflasque,
as in \S\ref{sec:setup}.

\emph{Suppose that $K$ is any special parahoric subgroup for which Proposition~\ref{prop:choosing-K} holds.}

Fix an $F$-splitting $\theta : \mathbf Z \onto \mathbf {Z_G}$ of the exact sequence $1 \to \mathbf {Z_G} \to \mathbf Z \to \mathbf Z/\mathbf {Z_G} \to 1$.  Since $X^*(\mathbf Z/\mathbf {Z_G}) = \oplus_{\Delta\abs} \Z \wt \alpha$, we have
a \emph{canonical} absolute fundamental coweight $\lambda_{\wt \beta} \in X_*(\mathbf Z)$ for any $\wt \beta \in \Delta\abs$,
normalized by demanding that it be orthogonal to $\theta^* X^*(\mathbf {Z_G})$. 
These are permuted by the action of $\Gamma_F$.
Thus for any simple root $\beta \in \Delta$ we obtain a \emph{canonical} relative fundamental coweight $\lambda_\beta \in X_*(\mathbf S) = X_*(\mathbf Z)^{\Gamma_F}$
by taking the sum of $\lambda_{\wt \beta} \in X_*(\mathbf Z)$ for all lifts $\wt \beta \in \Delta\abs$ of $\beta$.
(It is the unique fundamental coweight for $\beta$ that is orthogonal to $\theta^* X^*(\mathbf {Z_G})$.)

\begin{lemma}\label{lm:alpha-spacing}
  We have $\Lambda = \Z \frac 1{\varepsilon_\alpha} \lambda_\alpha \oplus \ker \alpha$ inside $X_*(\mathbf S) \otimes \R$.
\end{lemma}

\begin{proof}
  Note that $X_*(\mathbf Z) = \bigoplus \Z \lambda_{\wt\beta} \oplus (\Z \Phi\abs)^\perp$, where $\wt\beta$ runs through
  $\Delta\abs$.  It follows that $X_*(\mathbf Z)_{I_F}^\sigma$ is the direct sum of $\Z [\sum_{\Phi'} \lambda_{\wt \alpha}]$,
  where $\Phi'$ is as in the proof of Lemma~\ref{lm:image-in-Lambda}, and a module that is orthogonal to $\alpha$.  As in the
  proof of Lemma~\ref{lm:image-in-Lambda} we see that $[\sum \lambda_{\wt \alpha}]$ is identified with $\frac
  1{\varepsilon_\alpha} \lambda_\alpha \in X_*(\mathbf S) \otimes \R$.
\end{proof}

As $\alpha \in \Delta(V)$, Corollary~\ref{cor:choosing-K} shows that $\psi_V(Z^0 \cap M'_\alpha) =
1$. In particular, $\tau_\alpha \in \HH_Z(\psi_V)$ is well-defined.

\begin{lemma}\label{lm:irred}
  The element $1-\tau_\alpha$ of $\HH_Z(\psi_V)$ is irreducible.
\end{lemma}

\begin{proof}
  As the character $\psi_V : Z^0 \to C\s$ is trivial on $Z^0 \cap M_\alpha'$, we can extend it to a
  character $\eta : Z \to C\s$ that is trivial on $Z \cap M_\alpha'$.  We get an isomorphism $\iota : \HH_Z(\psi_V) \congto
  \HH_Z(1) = C[\Lambda]$, defined by $\iota(f)(z) = \eta(z)^{-1} f(z)$ for $z \in Z$. In particular, $\iota(\tau_z) =
  \eta(z)^{-1}\tau_z$. Thus it suffices to show that $\iota(1-\tau_\alpha) = 1-\tau_{a_\alpha}$ is irreducible in $C[\Lambda]$.
  By Lemma~\ref{lm:image-in-Lambda} and freeness of $\Lambda$ we can extend $x_1 := a_\alpha$ to a $\Z$-basis $x_1,\dots,x_r$ of $\Lambda$.
  Obviously, $1-x_1$ is irreducible in $C[x_1^{\pm 1},\dots,x_r^{\pm 1}]$.
\end{proof}

Recall that for any $z \in Z^+$ with $\ang{\alpha,z} > 0$ we have intertwining operators $T_z^{V',V} : \kind V \to \kind V'$ and
$T_z^{V,V'} : \kind V' \to \kind V$ supported on the double coset $KzK$. %

\begin{proposition}\label{prop:change-of-weight}
  Suppose $z \in Z$ such that $v_Z(z) = \frac 1{\varepsilon_\alpha} \lambda_\alpha$.  Then $S^G(T_z^{V',V}) = \tau_{z}$ and
  $S^G(T_z^{V,V'}) = \tau_{z} (1-\tau_\alpha)$ in $\HH_Z(\psi_V)$.
\end{proposition}

\begin{proof}
  We have that $S^G(T_z^{V',V}) = \tau_{z}$ by Lemma~\ref{first} and the coefficient of $\tau_z$ in $S^G(T_z^{V,V'})$ is 1.
  It thus suffices to show that $\psi \in C \tau_{z^2} (1-\tau_{\frac 1{\varepsilon_\alpha} \alpha_0\dual})$, where 
  $\psi = S^G(T_z^{V,V'} * T_z^{V',V}) \in \HH_Z(\psi_V)$.

  Pick any algebra homomorphism $\chi: \HH_Z(\psi_V) \to C$. Then as in \S\ref{subsec:Reducibility and change of weight} we know that the character
  $\sigma_\chi := \chi \otimes_{\HH_Z(\psi_V)}\ind_{Z^0}^Z \psi_V$ of $Z$ is given by $z\mapsto \chi (\tau_{z^{-1}})$, and
  that the restriction of $\sigma_\chi$ to $Z^0$ equals $\psi_V$.
  Assume now that $\chi(\tau_\alpha) = 1$. We know that $\sigma_\chi$ is trivial on the image of $Z^0 \cap M_\alpha'$ by
  above. Moreover, $Z \cap M_\alpha'$ is generated by $Z^0 \cap M_\alpha'$ and $a_\alpha$, so $\sigma_\chi$ is trivial on $Z
  \cap M_\alpha'$, as $\sigma_\chi(a_\alpha) = \chi(\tau_\alpha^{-1}) = 1$.
  As $M_\alpha = \langle Z, U_{\pm \alpha}\rangle$, we have an isomorphism $Z/(Z \cap M_\alpha') \cong M_\alpha/M_\alpha'$,
  so $\sigma_\chi$ extends to a smooth character of $M_\alpha$, which we still denote by $\sigma_\chi$. By Frobenius
  reciprocity, the induced representation $\Ind_{P_{\alpha}}^G \sigma_\chi$ contains $V$ but not $V'$, and the Hecke
  eigenvalues of $V$ in $\Ind_{P_{\alpha}}^G \sigma_\chi$ are given by $\chi$ via $S^G$ (see Lemma~\ref{va} and the
  proof of Lemma~\ref{vb}). As in \S\ref{sec:simple-proof-of-change-of-wt} we deduce that $\chi(\psi) = 0$.

  We saw that $\chi(1-\tau_\alpha) = 0$ implies that $\chi(\psi) = 0$. By the Nullstellensatz we get that $\psi$ is contained
  in the radical of the ideal $(1-\tau_\alpha)$, hence by Lemma~\ref{lm:irred} and the fact that $\HH_Z(\psi_V) (\approx C[\Lambda])$ is a UFD,
  we see that $\psi = \psi' (1-\tau_\alpha)$ for some $\psi' \in \HH_Z(\psi_V)$.

  As in \S\ref{sec:simple-proof-of-change-of-wt}, by Lemma~\ref{lm:support-satake}, we now see that if $z' \in Z$ is in the support of $\psi'$, then
  \begin{gather}
    z' \in Z^+,\ z' a_\alpha \in Z^+; \label{eq:13}\\
    v_Z(z') \le_{\R} \textstyle\frac 2{\varepsilon_\alpha} \lambda_\alpha,\ v_Z(z' a_\alpha) \le_{\R} \textstyle\frac 2{\varepsilon_\alpha} \lambda_\alpha.\label{eq:14}
  \end{gather}
  (This follows since for $z' \in \supp \psi$ we have $z' \in Z^+$ and $v_Z(z') \le_{\R} \textstyle\frac 2{\varepsilon_\alpha} \lambda_\alpha$.)
  From \eqref{eq:14} we can write
  \begin{equation}
    \label{eq:15}
    v_Z(z') = \textstyle\frac 2{\varepsilon_\alpha} \lambda_\alpha - \sum_\Delta n_\beta \beta^\vee
  \end{equation}
  for some $n_\beta \in \R_{\ge 0}$. Hence by Remark~\ref{rk:v_Z-a_alpha},
  \begin{equation}
    \label{eq:16}
    v_Z(z' a_\alpha) = \textstyle\frac 2{\varepsilon_\alpha} \lambda_\alpha - \textstyle\frac 1{\varepsilon_\alpha} \alpha_0\dual - \sum_\Delta n_\beta \beta^\vee.
  \end{equation}
  For $\gamma \in \Delta-\{\alpha\}$ we pair \eqref{eq:15} with $\gamma$ and deduce that $\sum_\Delta n_\beta \langle \gamma,\beta^\vee\rangle \le 0$.

  \emph{Case 1: $2\alpha \not \in \Phi$, so $\alpha_0\dual = \alpha\dual$.} 
  We pair \eqref{eq:16} with $\alpha$ and deduce that $\sum_\Delta n_\beta \langle \alpha,\beta^\vee\rangle \le 0$. Hence as in \S\ref{sec:simple-proof-of-change-of-wt}
  we get that $n_\beta = 0$ for all $\beta \in \Delta$, so $\psi'$ is a scalar multiple of $\tau_{z^2}$, as required.

  \emph{Case 2: $2\alpha \in \Phi$, so $\alpha_0\dual = \frac 12 \alpha\dual$}. The above proof goes through, provided we
  show
  \begin{equation}
    \label{eq:12}
    \langle \alpha,v_Z(z')\rangle \ge \textstyle\frac 1{\varepsilon_\alpha},\ \langle \alpha,v_Z(z' a_\alpha)\rangle \ge \textstyle\frac 1{\varepsilon_\alpha}
  \end{equation}
  for any $z' \in \supp \psi'$. For this it is enough to show that $\langle \alpha,v_Z(z')\rangle \ge \textstyle\frac 1{\varepsilon_\alpha}$
  for any $z' \in \supp \psi$. As $S^G(T_z^{V',V}) = \tau_{z}$ by Lemma~\ref{first} it suffices to show that $\langle \alpha,v_Z(z')\rangle \ge 0$
  for any $z' \in \supp S^G(T_z^{V,V'})$. In fact, we will show that $\langle \alpha,v_Z(z')\rangle \ge 0$
  for any $z' \in \supp S^G(\vp)$ and any $\vp \in \HH_G(V_1,V_2)$ (where $V_1, V_2$ are irreducible representations of $K$).

  By \cite[\S7.9]{MR3331726}, it suffices to show that $z'^{-1} (U_\alpha \cap K) z'$ is a proper subgroup of $U_\alpha \cap K_+$
  for $z' \in Z$ such that $\ang{\alpha,z'} < 0$.  Using notation as in \cite[\S6]{MR3331726} we can write $z'^{-1}(U_\alpha \cap
  K)z' = U_{\alpha, g(\alpha)-\ang{\alpha,z'}}U_{2\alpha, g(2\alpha)-2\ang{\alpha,z'}}$ and $U_\alpha \cap K_+ = U_{\alpha,
    g^*(\alpha)}U_{2\alpha, g^*(2\alpha)}$. Recall that $g^*(\beta) = g(\beta)_+$ if a jump occurs in the
  $U_{\beta,u}$-filtration (modulo $U_{2\beta}$ if $2\beta$ is a root) at $u = g(\beta)$ and $g^*(\beta) = g(\beta)$
  otherwise. Also note the set of jumps of the $U_{\beta,u}$-filtration (modulo $U_{2\beta}$) are invariant under shifts by
  $\ang{\beta,z'}$ (as $Z$ acts on the apartment with all its structures). For any fixed $\beta \in \{\alpha,2\alpha\}$ it
  follows that $U_{\beta,g(\beta)-\ang{\beta,z'}} \subset U_{\beta,g^*(\beta)}$ and if equality holds, then the
  $U_{\beta,u}$-filtration (modulo $U_{2\beta}$) jumps precisely at the elements $u \in g(\beta) + \ang{\beta,z'}\Z$. Thus $z'^{-1}
  (U_\alpha \cap K) z' \subset U_\alpha \cap K_+$ and if equality holds, then the $U_{\beta,u}$-filtration (modulo
  $U_{2\beta}$) jumps precisely at the elements $u \in g(\beta) + \ang{\beta,z'}\Z$ for $\beta \in \{\alpha,2\alpha\}$;
  in particular, $g(2\alpha) = 2g(\alpha)$ from the definition of $g$.

  By \cite[4.2.21]{bib:BT2} the jumps in the $U_{2\alpha,u}$-filtration occur when $u \in \ord_F(F_\alpha^0 - \{0\})$ and in
  the $U_{\alpha,u}$-filtration (modulo $U_{2\alpha}$) occur when $u \in \frac 12 \ord_F(\ell) + \ord_F(F_\alpha\s)$. Here,
  $F_\alpha^0$ denotes the elements of $F_\alpha$ that are of trace~0 in the separable quadratic extension $F_\alpha/F_{2\alpha}$,
  $\ell \in F_\alpha$ denotes an element of trace~1 of maximum possible valuation.  Note that $F_\alpha^0-\{0\}$ is principal
  homogeneous under the $F_{2\alpha}\s$-action, so the spacing of the jumps in the $U_{2\alpha,u}$-filtration is
  $\ord_F(F_{2\alpha}\s)$.  The spacing of the jumps in the $U_{\alpha,u}$-filtration (modulo $U_{2\alpha}$) is
  $\ord_F(F_{\alpha}\s)$.

  So if equality holds above, then $F_\alpha/F_{2\alpha}$ is ramified and $g(2\alpha) = 2g(\alpha)$. We finish by showing
  that this is impossible. By the previous paragraph we can pick $\ell' \in F_\alpha^0-\{0\}$ of the same valuation as
  $\ell$.  As $F_\alpha/F_{2\alpha}$ is ramified we can scale $\ell'$ by an element of $\O_{F_{2\alpha}}\s$ such that
  $\ord_F(\ell-\ell') > \ord_F(\ell)$. This contradicts that $\ell$ has maximum possible valuation among elements of trace~1.
  (Alternatively, from Tits' tables in \cite{MR546588} the affine root system can only be non-reduced if 
  the adjoint group has a factor isomorphic to $\Res_{E/F} H$, where $H \cong \PU(m+1,m)$ is unramified and $E/F$ is finite separable
  and in that case the extension $F_\alpha/F_{2\alpha}$ is unramified.)
\end{proof}

We can now deduce Theorem~\ref{thm:change-of-weight} from Proposition~\ref{prop:change-of-weight} exactly as in \S\ref{sec:simple-proof-of-change-of-wt2}, replacing $\mu_\alpha$ there by $\frac 1{\varepsilon_\alpha} \lambda_\alpha$.  (It is still true, by
Lemma~\ref{lm:alpha-spacing}, that if $z \in Z^+$ with $\ang{\alpha,v_Z(z)} > 0$ and $v_Z(z_0) = \frac 1{\varepsilon_\alpha} \lambda_\alpha$
then $z z_0^{-1} \in Z^+$.)

\subsection{First reduction step}
\label{sec:first-reduction}

We continue to assume that $\mathbf G^{\der}$ is simply connected and $\mathbf G/\mathbf G^{\der}$ is coflasque.  We now reduce to the basic case
(\S\ref{sec:basic}).

\begin{proposition}\label{prop:dual-z-ext}
  There exists a quasi-split connected reductive group $\mathbf G_1$ containing $\mathbf G$ as a closed normal subgroup such that
  \begin{enumerate}
  \item $\mathbf G_1^{\der} = \mathbf G^{\der}$;
  \item the torus $\mathbf G_1/\mathbf G_1^{\der}$ is coflasque;
  \item $1 \to \mathbf {Z_{G_1}} \to \mathbf Z_1 \to \mathbf Z_1/\mathbf {Z_{G_1}} \to 1$ is a split exact sequence of $F$-tori.
  \end{enumerate}
  Here, $\mathbf Z_1$ denotes the minimal Levi $\mathbf Z \cdot \mathbf {Z_{G_1}} = \mathbf C_{\mathbf G_1}(\mathbf Z)$ of $\mathbf G_1$.
\end{proposition}

\begin{proof}
  We define $\mathbf G_1$ and $\mathbf Z_1$ exactly as in~\S\ref{sec:more-general-case}(1), so in particular (i) holds.
  The exact sequence $1 \to
  \mathbf {Z_{G_1}} \to \mathbf Z_1 \to \mathbf Z/\mathbf {Z_G} \to 1$, where the second map is induced by the first projection, has a canonical splitting
  induced by $\mathbf Z \to \mathbf Z \times \mathbf Z$, $z \mapsto (z,z^{-1})$. This implies (iii). Finally, consider the short exact sequence
  $1 \to \mathbf G/\mathbf G^{\der} \to \mathbf G_1/\mathbf G_1^{\der} \to \mathbf Z/\mathbf {Z_G} \to 1$. The first term is coflasque by assumption and the last term
  is induced because it is the maximal torus in the quasi-split adjoint group $\mathbf G/\mathbf {Z_G}$. Hence $\mathbf G_1/\mathbf G_1^{\der}$ is coflasque and
  (ii) follows.
\end{proof}

Hence the group $\mathbf G_1$ is as in \S\ref{sec:basic}. The reduced buildings of $G$ and $G_1$ are canonically identified
with each other (as the reduced building only depends on the adjoint group), in particular there is a natural bijection
between special parahoric subgroups of these two groups.  Denote by $K_1$ any special parahoric subgroup of
$G_1$ %
and let $K$ denote the corresponding special parahoric subgroup of $G$.

\begin{lemma}\label{lm:parahoric-and-dual-z-ext}
  We have $K = K_1 \cap G$.
\end{lemma}

\begin{proof}
  Consider the commutative diagram given by functoriality of the Kottwitz homomorphism. (Note that the codomains simplify, since $\mathbf G^{\der} = \mathbf G_1^{\der}$ is simply connected. See~\cite[\S7.4]{kottwitz}.)
  \begin{equation*}
    \xymatrix{G\mathstrut \ar[r]^-{w_G}\ar@{^{(}->}[d] & X_*(\mathbf G/\mathbf G^{\der})_{I_F}^\sigma \ar[d] \\
    G_1 \ar[r]^-{w_{G_1}} & X_*(\mathbf G_1/\mathbf G_1^{\der})_{I_F}^\sigma}
  \end{equation*}
  We claim that the vertical arrow on the right is injective. The first term in the short exact sequence 
  $1 \to \mathbf G/\mathbf G^{\der} \to \mathbf G_1/\mathbf G_1^{\der} \to \mathbf Z/\mathbf {Z_G} \to 1$ of $F$-tori is coflasque, so $X_*(\mathbf G/\mathbf G^{\der})_{I_F}$ is torsion-free, as noted
  in the proof of Lemma~\ref{lm:max-cpt}. Let $\Gamma$ be a finite quotient of $I_F$ through which it acts on the character groups of the tori in the sequence.
  Then $H_1(\Gamma, X_*(\mathbf Z/\mathbf {Z_G}))$ is torsion, as $\Gamma$ is finite, so $X_*(\mathbf G/\mathbf G^{\der})_{I_F} \to X_*(\mathbf G_1/\mathbf G_1^{\der})_{I_F}$ is injective,
  which implies the claim.

  Since the reduced buildings of $G$ and $G_1$ are naturally identified and parahoric subgroups are the fixers of
  facets in the kernel of the Kottwitz homomorphism, it follows that $K = K_1 \cap G$.
\end{proof}

\begin{lemma}\label{lm:extend}
  The restriction to $K$ of any irreducible representation of $K_1$ is irreducible. Conversely, any irreducible  representation of $K$  extends to $K_1$. 
\end{lemma}

\begin{proof}
  Note that as $K \lhd K_1$, the pro-$p$ radical of $K$ is normal in $K_1$, so we get a commutative diagram as follows:
  \begin{equation}\label{eq:17}
    \begin{gathered}
      \xymatrix{K \ar@{^{(}->}[r]\ar@{->>}[d] & K_1 \ar@{->>}[d] \\
        G_k\ar@{^{(}->}[r] & G_{1,k}}
    \end{gathered}
  \end{equation}  
  Note that $G_{1,k}' \subset G_k \subset G_{1,k}$.  It is enough to show that any irreducible
  representation of $G_{1,k}$ restricts irreducibly to $G_{1,k}'$, and hence to $G_k$. (Then if $V$ is an
  irreducible representation of $G_k$, any irreducible quotient of $\Ind_{G_k}^{G_{1,k}} V$ extends $V$ to
  $G_{1,k}$.)

  We will prove more generally that if $\mathbf H$ is any connected reductive group over $k$ and $V$ an irreducible representation
  of $H$, then the restriction of $V$ to $H'$ is irreducible. Suppose first that the derived subgroup $\mathbf H^{\der}$ is simply
  connected. Then $H' = H^{\der}$. We know that we can lift $V$ to an irreducible representation of $\mathbf H$ with
  $q$-restricted highest weight (where $q = \# k$), cf.\ \cite[Appendix, (1.3)]{MR2541127}. Then its restriction to $\mathbf H^{\der}$ is still irreducible with $q$-restricted
  highest weight (noting that $\mathbf H$ is generated by its center and $\mathbf H^{\der}$). Hence $V$ restricted to $H^{\der}$
  remains irreducible by the result we just cited.

  For the general case pick a $z$-extension $\pi : \mathbf{\wt H} \onto \mathbf H$, so $\mathbf R := \ker\pi$ is an induced torus and $\mathbf{\wt H}^{\der}$ is simply connected.
  We have a commutative diagram with exact rows:
  \begin{equation*}
    \xymatrix{
      1 \ar[r] & R \cap {\wt H}'\mathstrut \ar[r]\ar@{^{(}->}[d] & {\wt H}'\mathstrut \ar[r]\ar@{^{(}->}[d] & H'\mathstrut \ar[r]\ar@{^{(}->}[d] & 1 \\
      1 \ar[r] & R \ar[r] & {\wt H} \ar[r] & H \ar[r] & 1
    }
  \end{equation*}
  By inflation we can consider $V$ as irreducible representation $\wt V$ of ${\wt H}$ that is trivial on $R$. 
  By above we know the restriction of $\wt V$ to ${\wt H}'$ is irreducible, and hence so is the restriction of $V$ to $H'$.
\end{proof}

\begin{remark}\label{rk:morphism2}
  As in Remark~\ref{rk:morphism} we expect that the map $G_k\to G_{1,k}$ arises from a closed immersion $\mathbf G_k \to \mathbf G_{1,k}$.
\end{remark}

\begin{lemma}\label{lm:K-K1}
  Proposition~\ref{prop:choosing-K} holds for $(G,K)$ if and only if it holds for $(G_1,K_1)$.
  More precisely, we have $\red(M_\alpha' \cap K) = M_{\alpha,k}'$ inside $G_k$ if and only if
  $\red(M_{1,\alpha}' \cap K_1) = M_{1,\alpha,k}'$ inside $G_{1,k}$.
\end{lemma}

\begin{proof}
  Fix $\alpha \in \Delta$. We note that $\mathbf M_\alpha\lhd \mathbf M_{1,\alpha}$ for the Levi subgroups defined by $\alpha$ and that
  by Lemma~\ref{lm:parahoric-and-dual-z-ext} we have $M_\alpha \cap K \lhd M_{1,\alpha} \cap K_1$ for the corresponding
  special parahoric subgroups. Hence, restricting the top row of diagram~\eqref{eq:17} (applied to Levi subgroups defined by
  $\alpha$), we get a commutative diagram
  \begin{equation*}
    \xymatrix{M_\alpha' \cap K \ar@{^{(}->}[r]\ar@{->}[d] & M_{1,\alpha}' \cap K_1 \ar@{->}[d] \\
      M_{\alpha,k}\ar@{^{(}->}[r] & M_{1,\alpha,k}}
  \end{equation*}
  Note that the top row is an isomorphism (by Lemma~\ref{lm:parahoric-and-dual-z-ext}, as $M_\alpha' = M_{1,\alpha}'$) and that the bottom row induces an isomorphism between the vertical images,
  as well as between $M_{\alpha,k}'$ and $M_{1,\alpha,k}'$. The lemma follows.
\end{proof}

Choose now any $K$ such that Proposition~\ref{prop:choosing-K} holds for $(G,K)$; equivalently,
Proposition~\ref{prop:choosing-K} holds for $(G_1,K_1)$, by Lemma~\ref{lm:K-K1}. From Corollary~\ref{cor:choosing-K} and
since $\alpha \in \Delta(V)$, we see that $\psi_V(Z^0 \cap M_\alpha') = 1$. Now we deduce in exactly the same way as in
\S\ref{sec:more-general-case}(1) that Theorem~\ref{thm:change-of-weight} holds for $(G,K)$, since we know it holds for $(G_1,K_1)$ by
\S\ref{sec:basic}.

\subsection{Second reduction step}
\label{sec:second-reduct-step}

Suppose now that $\mathbf G$ is any quasi-split group. We will reduce to the previous case.
The following result is proved by Colliot-Th\'el\`ene \cite[Prop.\ 4.1]{ct08}.

\begin{proposition}\label{prop:coflasque-resolution}
  The group $\mathbf G$ has a \(quasi-split\) $z$-extension $\mathbf{\wt G}$ such that $\mathbf{\wt G}/\mathbf{\wt G}^{\der}$ is a coflasque torus.
\end{proposition}

Hence the group $\mathbf{\wt G}$ is as in \S\ref{sec:first-reduction}. Now choose any special parahoric subgroup $\wt K$ of $\wt
G$ for which Proposition~\ref{prop:choosing-K} holds. Let $K$ denote the corresponding special parahoric subgroup of $G$. It
follows from Step 3 of the proof of Proposition~\ref{prop:choosing-K} that Proposition~\ref{prop:choosing-K} holds also for $(G,K)$. From
Corollary~\ref{cor:choosing-K} and since $\alpha \in \Delta(V)$, we see that $\psi_V(Z^0 \cap M_\alpha') = 1$. Now we deduce
in exactly the same way as in \S\ref{sec:more-general-case}(2) that Theorem~\ref{thm:change-of-weight} holds for $(G,K)$, since we know
it holds for $(\wt G,\wt K)$ by \S\ref{sec:first-reduction}.

\makeatletter
\def\@tocwrite#1#2{}%
\subsection*{Acknowledgments}
We thank Tasho Kaletha and Marie-France Vign\'eras for some helpful discussions.
\let\@tocwrite=\original@@tocwrite %
\makeatother

\bibliographystyle{amsalpha}
\bibliography{bib}

\providecommand{\bysame}{\leavevmode\hbox to3em{\hrulefill}\thinspace}
\providecommand{\MR}{\relax\ifhmode\unskip\space\fi MR }
\providecommand{\MRhref}[2]{%
  \href{http://www.ams.org/mathscinet-getitem?mr=#1}{#2}
}
\providecommand{\href}[2]{#2}
\begin{thebibliography}{AHHV17}

\bibitem[Abe13]{MR3143708}
Noriyuki Abe, \emph{On a classification of irreducible admissible modulo {$p$}
  representations of a {$p$}-adic split reductive group}, Compos. Math.
  \textbf{149} (2013), no.~12, 2139--2168. \MR{3143708}

\bibitem[Abe19]{arXiv:1406.1003}
\bysame, \emph{Modulo {$p$} parabolic induction of pro-{$p$}-{I}wahori {H}ecke
  algebra}, J. Reine Angew. Math. \textbf{749} (2019), 1--64. \MR{3935898}

\bibitem[AHHV17]{MR3600042}
N.~Abe, G.~Henniart, F.~Herzig, and M.-F. Vign\'eras, \emph{A classification of
  irreducible admissible mod {$p$} representations of {$p$}-adic reductive
  groups}, J. Amer. Math. Soc. \textbf{30} (2017), no.~2, 495--559.
  \MR{3600042}

\bibitem[BB05]{bjorner-brenti}
Anders Bj{\"o}rner and Francesco Brenti, \emph{Combinatorics of {C}oxeter
  groups}, Graduate Texts in Mathematics, vol. 231, Springer, New York, 2005.
  \MR{2133266}

\bibitem[Bou02]{MR1890629}
Nicolas Bourbaki, \emph{Lie groups and {L}ie algebras. {C}hapters 4--6},
  Elements of Mathematics (Berlin), Springer-Verlag, Berlin, 2002, Translated
  from the 1968 French original by Andrew Pressley. \MR{1890629 (2003a:17001)}

\bibitem[BT84]{bib:BT2}
F.~Bruhat and J.~Tits, \emph{Groupes r\'eductifs sur un corps local. {II}.
  {S}ch\'emas en groupes. {E}xistence d'une donn\'ee radicielle valu\'ee},
  Inst. Hautes \'Etudes Sci. Publ. Math. (1984), no.~60, 197--376. \MR{MR756316
  (86c:20042)}

\bibitem[CE04]{MR2057756}
Marc Cabanes and Michel Enguehard, \emph{Representation theory of finite
  reductive groups}, New Mathematical Monographs, vol.~1, Cambridge University
  Press, Cambridge, 2004. \MR{2057756 (2005g:20067)}

\bibitem[CL76]{MR0396731}
R.~W. Carter and G.~Lusztig, \emph{Modular representations of finite groups of
  {L}ie type}, Proc. London Math. Soc. (3) \textbf{32} (1976), no.~2, 347--384.
  \MR{MR0396731 (53 \#592)}

\bibitem[CT08]{ct08}
Jean-Louis Colliot-Th{\'e}l{\`e}ne, \emph{R\'esolutions flasques des groupes
  lin\'eaires connexes}, J. Reine Angew. Math. \textbf{618} (2008), 77--133.
  \MR{2404747}

\bibitem[DL76]{MR0393266}
P.~Deligne and G.~Lusztig, \emph{Representations of reductive groups over
  finite fields}, Ann. of Math. (2) \textbf{103} (1976), no.~1, 103--161.
  \MR{0393266 (52 \#14076)}

\bibitem[Her09]{MR2541127}
Florian Herzig, \emph{The weight in a {S}erre-type conjecture for tame
  {$n$}-dimensional {G}alois representations}, Duke Math. J. \textbf{149}
  (2009), no.~1, 37--116. \MR{MR2541127}

\bibitem[Her11a]{MR2845621}
\bysame, \emph{The classification of irreducible admissible mod {$p$}
  representations of a {$p$}-adic {${\rm GL}_n$}}, Invent. Math. \textbf{186}
  (2011), no.~2, 373--434. \MR{2845621}

\bibitem[Her11b]{bib:satake}
\bysame, \emph{A {S}atake isomorphism in characteristic {$p$}}, Compos. Math.
  \textbf{147} (2011), no.~1, 263--283. \MR{2771132}

\bibitem[HR08]{bib:haines-rap}
G.~Haines and M.~Rapoport, \emph{On parahoric subgroups. {A}ppendix to
  {T}wisted loop groups and their affine flag varieties}, Adv. Math.
  \textbf{219} (2008), no.~1, 188--198.

\bibitem[HR10]{haines-rostami}
Thomas~J. Haines and Sean Rostami, \emph{The {S}atake isomorphism for special
  maximal parahoric {H}ecke algebras}, Represent. Theory \textbf{14} (2010),
  264--284. \MR{2602034}

\bibitem[HV12]{MR3001801}
Guy Henniart and Marie-France Vign{\'e}ras, \emph{Comparison of compact
  induction with parabolic induction}, Pacific J. Math. \textbf{260} (2012),
  no.~2, 457--495. \MR{3001801}

\bibitem[HV15]{MR3331726}
\bysame, \emph{A {S}atake isomorphism for representations modulo {$p$} of
  reductive groups over local fields}, J. Reine Angew. Math. \textbf{701}
  (2015), 33--75. \MR{3331726}

\bibitem[Kot97]{kottwitz}
Robert~E. Kottwitz, \emph{Isocrystals with additional structure. {II}},
  Compositio Math. \textbf{109} (1997), no.~3, 255--339. \MR{1485921}

\bibitem[Mil]{milne-iAG}
J.~Milne, \emph{Algebraic {G}roups}, course notes v.2.00, available at
  \url{http://www.jmilne.org/math/CourseNotes/ala.html}.

\bibitem[NSW00]{nsw}
J\"urgen Neukirch, Alexander Schmidt, and Kay Wingberg, \emph{Cohomology of
  number fields}, Grundlehren der Mathematischen Wissenschaften [Fundamental
  Principles of Mathematical Sciences], vol. 323, Springer-Verlag, Berlin,
  2000. \MR{1737196}

\bibitem[Oll14]{MR3263136}
Rachel Ollivier, \emph{Compatibility between {S}atake and {B}ernstein
  isomorphisms in characteristic {$p$}}, Algebra Number Theory \textbf{8}
  (2014), no.~5, 1071--1111. \MR{3263136}

\bibitem[Oll15]{MR3366919}
\bysame, \emph{An inverse {S}atake isomorphism in characteristic {$p$}},
  Selecta Math. (N.S.) \textbf{21} (2015), no.~3, 727--761. \MR{3366919}

\bibitem[Rap05]{MR2141705}
Michael Rapoport, \emph{A guide to the reduction modulo {$p$} of {S}himura
  varieties}, Ast\'erisque (2005), no.~298, 271--318, Automorphic forms. I.
  \MR{2141705}

\bibitem[Spr09]{MR2458469}
T.~A. Springer, \emph{Linear algebraic groups}, second ed., Modern Birkh\"auser
  Classics, Birkh\"auser Boston Inc., Boston, MA, 2009. \MR{2458469
  (2009i:20089)}

\bibitem[Tit66]{bib:tits-ss-groups}
J.~Tits, \emph{Classification of algebraic semisimple groups}, Algebraic
  {G}roups and {D}iscontinuous {S}ubgroups ({P}roc. {S}ympos. {P}ure {M}ath.,
  {B}oulder, {C}olo., 1965), Amer. Math. Soc., Providence, R.I., 1966, 1966,
  pp.~33--62. \MR{0224710}

\bibitem[Tit79]{MR546588}
\bysame, \emph{Reductive groups over local fields}, Automorphic forms,
  representations and {$L$}-functions ({P}roc. {S}ympos. {P}ure {M}ath.,
  {O}regon {S}tate {U}niv., {C}orvallis, {O}re., 1977), {P}art 1, Proc. Sympos.
  Pure Math., XXXIII, Amer. Math. Soc., Providence, R.I., 1979, pp.~29--69.
  \MR{MR546588 (80h:20064)}

\bibitem[Vig]{Vigneras-prop-IV}
Marie-France Vign{\'e}ras, \emph{{T}he pro-{$p$} {I}wahori-{H}ecke algebra of a
  reductive {$p$}-adic group {IV} ({L}evi subgroup and central extension)},
  preprint, available at
  \url{https://webusers.imj-prg.fr/~marie-france.vigneras/recent.html}.

\bibitem[Vig06]{MR2192484}
\bysame, \emph{Alg\`ebres de {H}ecke affines g\'en\'eriques}, Represent. Theory
  \textbf{10} (2006), 1--20 (electronic). \MR{2192484 (2006i:20005)}

\bibitem[Vig15]{MR3437789}
\bysame, \emph{The pro-p {I}wahori {H}ecke algebra of a reductive p-adic group,
  {V} (parabolic induction)}, Pacific J. Math. \textbf{279} (2015), no.~1-2,
  499--529. \MR{3437789}

\bibitem[Vig16]{MR3484112}
\bysame, \emph{The pro-{$p$}-{I}wahori {H}ecke algebra of a reductive
  {$p$}-adic group {I}}, Compos. Math. \textbf{152} (2016), no.~4, 693--753.
  \MR{3484112}

\end{thebibliography}
  
\end{document}